\documentclass[11pt]{amsart}
\usepackage{color}
\usepackage{amsmath,amssymb,amsfonts,amsthm,bm}
\usepackage{graphicx}
\usepackage{float}
\usepackage{enumerate}
\usepackage{appendix}
  \usepackage[shortlabels]{enumitem}
\usepackage{centernot}

\usepackage{color}

\def\N{\mathbb N}
\def\R{\mathbb R}
\def\Z{\mathbb Z}

\def\cC{\mathcal C}

\def\cE{\mathcal E}
\def\cH{\mathcal H}

\def\cM{\mathcal M}

\def\cO{\mathcal O}

\def\cU{\mathcal U}

\newcommand{\esp}[1]{\quad \text{#1} \quad}
\newcommand{\defn}{\overset{\rm def}{=}}

\newcommand\DIV{{\rm DIV}\,}
\renewcommand\div{{\rm div}\,}
\renewcommand\lim{{\rm lim}\,}

\renewcommand\sup{{\rm sup}\,}

\renewcommand\log{{\rm log}\,}

\newcommand{\with}{\quad\!\hbox{with}\!\quad}
\newcommand{\andf}{\quad\!\hbox{and}\!\quad}
\newcommand{\Sum}{\displaystyle \sum}

\newcommand{\Int}{\displaystyle \int}

\def\dV{\delta\!V}

\def\ddj{\dot \Delta_j}

\newtheorem{theorem}{Theorem}[section]
 
 \newtheorem{lemma}[theorem]{Lemma}
 \newtheorem{proposition}[theorem]{Proposition}
 \theoremstyle{definition}
 \newtheorem{definition}[theorem]{Definition}
 \theoremstyle{remark}
 \newtheorem{remark}[theorem]{Remark}
 
 \numberwithin{equation}{section}

\newcommand{\wt}{\widetilde}

\renewcommand{\div}{\mbox{div}\,}

\newcommand{\ro}{\rho}
\newcommand{\n}{\nabla}
\newcommand{\pt}{\partial_t}
\newcommand{\pal}{\partial_\alpha}
\newcommand{\pbe}{\partial_\beta}

\newcommand{\sumab}{\sum_{\alpha,\beta=1}^d}
\newcommand{\suma}{\sum_{\alpha=1}^d}
\newcommand{\Rd}{\mathbb{R}^d}
\newcommand{\cd}{\frac{d}{2}}

\newcommand{\tV}{\widetilde{V}}

\newcommand{\Dj}{\Delta_j}
\newcommand{\DDj}{\Dot{\Delta}_j}

\newcommand{\intd}{\int_{\mathbb{R}^d}}

\newcommand{\Lde}{L^2}

\newcommand{\normede}[1]{\left\Vert #1\right\Vert_{L^2}}

\newcommand{\normeinf}[1]{\left\Vert #1\right\Vert_{L^\infty}}
\newcommand{\BH}[3]{\dot{B}^{#1}_{#2,#3}}
\newcommand{\NBH}[4]{\left\Vert {#4} \right\Vert_{\dot{B}^{#1}_{#2,#3}}}

\newcommand{\NB}[4]{\left\Vert {#4} \right\Vert_{B^{#1}_{#2,#3}}}
\newcommand{\LpNB}[5]{\left\Vert {#4} \right\Vert_{L^{#5}_T(B^{#1}_{#2,#3})}}
\newcommand{\LptNB}[5]{\left\Vert {#4} \right\Vert_{L^{#5}_t(B^{#1}_{#2,#3})}}

\newcommand{\LpNBH}[5]{\left\Vert {#4} \right\Vert_{L^{#5}_T(\Dot{B}^{#1}_{#2,#3})}}
\newcommand{\TLpNB}[5]{\left\Vert {#4} \right\Vert_{\widetilde{L}^{#5}_T(B^{#1}_{#2,#3})}}

\newcommand{\TLpNBH}[5]{\left\Vert {#4} \right\Vert_{\widetilde{L}^{#5}_T(\dot{B}^{#1}_{#2,#3})}}

\newcommand{\B}[3]{B^{#1}_{#2,#3}}

\begin{document}
\title[]{Local  well-posedness in the critical regularity setting  for  hyperbolic systems  with partial diffusion}
\author{Jean-Paul Adogbo \& Raph\"{a}el Danchin }
\date{}
\begin{abstract}   This paper is dedicated to the local existence theory of the Cauchy problem for 
a general class of symmetrizable hyperbolic partially diffusive systems (also called hyperbolic-parabolic systems) in
the whole space $\Rd$ with  $d\ge 1$.
We address the question of well-posedness for large data having
critical Besov regularity in the spirit of previous works by the second author on the compressible Navier-Stokes equations.
Compared to the pioneering of Kawashima in \cite{Kawashima83} and to the more recent 
work  by  Serre  in \cite{Serr10}, we take advantage of the partial parabolicity of the system 
to consider data  in functional  spaces that need not be embedded in the set of Lipschitz functions. 
This is in sharp contrast with the classical well-posedness theory of (multi-dimensional) hyperbolic systems where it is mandatory. 
A leitmotiv of our analysis is to require less regularity for the  components experiencing a direct diffusion, 
than for the hyperbolic components. We then use an energy method that is performed on the system after
spectral localization  and a suitable  G\r{a}rding inequality.
 As an example, we consider the Navier-Stokes-Fourier equations. 
\end{abstract}

\maketitle

Many physical phenomena are modelled by first order hyperbolic equations with degenerate dissipative or diffusive terms. This is the case for example in gas dynamics, where the mass is conserved during the evolution, but the momentum balance includes a diffusion (viscosity) or friction (relaxation) term. In this paper we consider  systems of the form 
\begin{equation}
\label{eq:brut}
 \pt u +\suma \pal L^\alpha (u) =
 \sum_{\alpha,\beta=1}^d\pal (B^{\alpha\beta}(u) \pbe u) +f(u,\nabla u),
\end{equation}
	in which $u:(0,T)\times \Rd \longrightarrow \mathbb{U}$ is the unknown. The phase space $ \mathbb{U}$ is an open convex subset of $\mathbb{R}^n$. The nonlinearities are encoded in the smooth functions
	\[ L^\alpha: \mathbb{U}   \longrightarrow \R^n,\quad B^{\alpha\beta}: \mathbb{U}   \longrightarrow \mathcal{M}_{n}(\R) \esp{and} f:  \mathbb{U}\times \mathcal{M}_{n\times d}(\R)\to \R^n.  \]

Among the systems having the form  \eqref{eq:brut} are the  Navier-Stokes-Fourier equations, the magneto-hydrodynamics equations and electromagnetism equations \cite[chap.~6]{Kawashima83}, the supercritical fluid models with chemical
reactions \cite{GioMat13}, the Baer-Nunziato  system \cite{BurCrinTan23}, etc. In each case, diffusion (e.g. thermal conduction or viscosity) 
acts on some components of the unknown, while other components remain unaffected.
\smallbreak
It is well known  since the works by  A. Majda in \cite{Madja84}  and D. Serre in \cite{Serre97} (see also \cite[chapter 10]{GavSerre07}) that general systems of conservation laws (that is \eqref{eq:brut} with $ B\equiv0$) which are \emph{Friedrichs-symmetrizable} 
supplemented with smooth decaying data admit local-in-time strong solutions, that
  may develop singularities (shock waves) in finite time even if the initial data are small perturbations of a constant solution.

The picture changes drastically if the system under consideration possesses diffusive terms.
In his   seminal work   \cite{Kawashima83} on  partially diffusive hyperbolic  systems,  S. Kawashima 
proved  the local existence for a class of systems of type \eqref{eq:brut} supplemented with initial  data in $H^s(\Rd)$ with $s>d/2+2,$ and  exhibited  a sufficient  condition  for global well-posedness
for  small data belonging to $H^s(\Rd)$ with $s>d/2+3.$
This  condition  is now known as the \textit{Kawashima-Shizuta condition}.
It will be discussed in a forthcoming  paper \cite{DanADOglob}, the present work being 
dedicated to the local well-posedness theory for, possibly, large data.

Later, D. Serre in \cite{Serre09} made the link  between the requirements made by
S. Kawashima,  the notion of  entropy-dissipativity  (see Definition \ref{dfn:entro} below) and the constancy of the range of the symbol $B(\xi; u)$, see assumption \textbf{A} below. In \cite{Serr10}, he  provided the normal form for \eqref{eq:brut} close to that used by S. Kawashima and Y. Shizuta in \cite{KawaSui88}. This enabled him 
to enlarge the class of  initial data for local well-posedness to $H^s(\Rd)$ with $s>1+d/2$ (see also the recent paper \cite{Angel23} by F. Angeles).
\smallbreak
In the theory of  multi-dimensional purely  symmetric (or symmetrizable) hyperbolic systems, 
two barriers seem insurmountable: going below  Lipschitz regularity for the initial data
(which, in the Sobolev spaces setting $H^s,$ corresponds to $s>1+d/2$), and  beyond an $L^2$-type functional framework. 
In this paper, we strive for well-posedness results for partially diffusive systems, 
in a Sobolev setting (in fact, in the optimal Besov setting) 
that does not require embedding in $C^{0,1}.$ 
In this endeavour, we shall  keep in mind the paper \cite{Dan01glob} by the second author 
dedicated to the compressible Navier-Stokes system -- a model hyperbolic system with partial diffusion, where 
one  component (the density) is taken in the homogeneous Besov space $\dot B^{\frac d2}_{2,1}(\Rd)$
while the other components (the velocity) belong to  $\dot B^{\frac d2-1}_{2,1}(\Rd).$

Compared to the classical theory presented above, this particular example reveals that, in some cases, it is possible
to reduce by one derivative the regularity of the non-dissipated component, and  by \emph{two}  derivatives
that of the dissipated component (namely the velocity). 
A fundamental  observation is that to get optimal results in terms of regularity, one has to  work with 
one less derivative  for the component that experiences direct diffusion. 
\smallbreak
The present work aims at extending  the example of the compressible Navier-Stokes system,
 to the much more general class of systems that has been considered by D. Serre in \cite{Serr10,Serre09}.
We here concentrate on the well-posedness issue for (possibly) large data, and prepare the ground 
for a forthcoming paper  \cite{DanADOglob} dedicated to the global existence issue, large time asymptotics and diffusion limit for small data.


\section{Results} \label{s:results}

As a first, specifying the structure of the class of systems under consideration is in order. 
Following  D. Serre in \cite{Serre09}, we assume that the system of conservation laws associated to \eqref{eq:brut}  admits
a strongly convex entropy $\eta$ (that is,  $D^2 \eta (u)$ is positive definite for all $u\in  \mathbb{U}$) 
with  flux $q,$  namely,  for all smooth solution $u$ of 
\begin{align*}    \pt u+ \sum_{\alpha=1}^d\partial_\alpha L^\alpha(u)=0,\end{align*}
we have
\begin{align*}
    \pt \eta (u) + {\rm div} q(u) = 0.
\end{align*}
We define the partial and  total symbol  of the second order term in \eqref{eq:brut} to be
\begin{equation}
    \label{def:symbol}
    B^\alpha(\xi,u)\defn \sum_{\beta=1}^d B^{\alpha\beta}(u) \xi_\beta \esp{and}   B(\xi,u)\defn\sumab \xi_\alpha\xi_\beta B^{\alpha\beta}(u),\quad \xi \in \mathbb{R}^d,\quad u\in  \mathbb{U},
\end{equation}
and assume that   \eqref{eq:brut} is  entropy-dissipative, that is, there exists a continuous and positive function 
$\omega$ such that  
\begin{equation}\label{dfn:entro}
D^2\eta(u)\left( X,  B(\xi,u)X \right)\ge \omega(u) \suma \left|   B^\alpha(\xi,u) X\right|^2, \quad
 \xi \in \mathbb{R}^d,\ \forall  u\in  \mathbb{U},\ \forall X\in  \mathbb{R}^n.
\end{equation}
Following Serre's work in \cite{Serre09}, we make the following:
\paragraph{\bf Assumption A} 
The range of  $ B(\xi,u)$  is independent of   $\xi\in \Rd\backslash\{0\}$  and of $u \in  \mathbb{U} $. 
\medbreak
Hence  there exists $n_1$ in $\{0,\cdots,n\}$ 
such that  the range of  $ B(\xi,u)$ 
is isomorphic to $\{0\}\times \mathbb{R}^{n_2}$ with $n_2\defn n-n_1.$
Performing a linear change of coordinates then reduces the study to the case where  the $n_1$ first rows
of  $B(\xi,u)$ are null and the  rank is equal to $n_2.$
\smallbreak
 A typical illustration is  gas dynamics in $\R^d$: then, the first component of the system
 is   the density, a conserved quantity, while the other $d+1$ components
 (velocity field and temperature) are subject to diffusion (see Section \ref{sec:appli:CFNS} for more details).
 \medbreak
According to \cite [Theorem 1.1]{Serr10}, the fact that System \eqref{eq:brut} is entropy dissipative in the sense of \eqref{dfn:entro}, satisfies Assumption A and that the 
 $n_1$ first rows are first-order conservation laws
entails that  the map
$$ u= \begin{pmatrix}
	v\\   w
	\end{pmatrix} 
	\leftrightarrows U\defn  	\begin{pmatrix}
	v\\	z	\end{pmatrix}, \quad  v=(u_,\cdots,u_{n_1})^T $$
	    is a global diffeomorphism from $\mathbb{U}$ onto its image $\mathcal{U}.$ 
Furthermore, the viscous flux $B(u)\nabla_xu$ rewrites $Z(U)\nabla_x z$ and 
the operator $Z(U) \nabla_x$ is strongly elliptic:   there exists a continuous and positive function  $c_1$ such that: 
     \begin{equation}
	    \label{strong_elli}
	    \sum_{\alpha,\beta=1}^{d} \sum_{i,j>n_1}{\xi_\alpha\lambda_i\xi_\beta\lambda_jZ_{ij}^{\alpha\beta}  (U) }  \ge c_1(U) |\xi|^2|\lambda|^2, \ \forall \xi \in \mathbb{R}^d, \ \forall\lambda\in \mathbb{R}^{n_2}, \ \forall U\in\cU.
	    \end{equation}  	   
	
	 Assumption {\bf A} and the above change of variables  
	 ensure that  if we fix some reference state $\overline U$ of $\cU$ 
	 and set $V\defn U-\overline U,$ then  System \eqref{eq:brut} may be rewritten: 
	  \begin{equation}
	\label{Eq_b}
	 S^0(U)\pt V \!+\!\sum_{\alpha}{S^\alpha (U) \pal V} = \sum_{\alpha,\beta}{ \pal (Y^{\alpha\beta}(U) \pbe V)}+f(U,\n U).
	\end{equation}
In what follows, we shall write the matrices $S^\alpha (U)$ by block as follows for $\alpha=0,\cdots, d:$
$$ S^\alpha(U)=      \begin{pmatrix}
S^{\alpha}_{11}(U) & S^{\alpha}_{12}(U) \\ S^{\alpha}_{21}(U) &  S^\alpha_{22}(U)
\end{pmatrix}
\with S^\alpha_{11}(U)\in \cM_{n_1}(\mathbb{R})
\andf  S^\alpha_{22}(U)\in \cM_{n_2}(\mathbb{R}).$$
 We assume that the coefficients of the system \eqref{Eq_b} satisfy the 
 following\footnote{As observed by D. Serre in \cite{Serre09}, the block-diagonal structure of $S^0$ 
and of the dissipation tensor follow from Assumption \textbf{A} and the fact that the entropy $\eta$ is dissipative.}:
  \paragraph{\bf Assumption B} 
 \begin{enumerate}
 \label{cond:DD}
     \item The matrix $S^0(U)$ is  block diagonal and inversible on $\cU,$ 
     and $S^0_{22}(U)\in\cM_{n_2}(\R)$ is   symmetric  positive definite.
     \item \label{cond:DD:2} 
Either all the matrices $S_{11}^\alpha(U)\in\cM_{n_1}(\R)$ are symmetric with, in addition,   $S_{11}^0(U)$
     symmetric positive definite,      or all the matrices $(S^0_{11}(U))^{-1}S_{11}^\alpha(U)$ are symmetric.
     \item    The matrices $Y^{\alpha\beta}(U)$ have the following form: 
      \begin{equation}\label{def:Y:al:be}
       Y^{\alpha\beta}(U)=\begin{pmatrix}0_{n_1} & 0\\0 &  Z^{\alpha\beta}(U)
\end{pmatrix}\quad\hbox{with }\ Z^{\alpha\beta}(U)\in \cM_{n_2}(\mathbb{R})  \end{equation}
 and   Inequality \eqref{strong_elli} holds true. 
     \item The  function $f$ 
     satisfies  $f(\overline U,\cdot)=0$ 
     and may be written 
      \begin{equation}
\label{f=f(f1,f2)}
      f(U,\n U)= \begin{pmatrix}
 f^1(U)\\ f^2(U,\n U)
\end{pmatrix} \with f^2(U,\n U)=f^{21}(U)+ f^{22}(U,\n U^1)+ f^{23}(U,\n U^2),
  \end{equation}
  where $f^1$, $f^{21},$ $f^{22}$ and $f^{23}$ are  smooth functions satisfying 
  $$f^{1}(\overline{U})=0,\quad f^{21}(\overline{U})= f^{22}(\overline{U},0)=f^{23}(\overline{U},0)=0,$$ 
  and such that   $f^{23}$ is  at most quadratic with respect to  $\n U^2.$ 
 \end{enumerate}
 \begin{remark} 
    \label{rmq:cnd:source:term}
 Compared to the work by D. Serre, we consider a slightly more general class of systems: 
 only the submatrices $S^\alpha_{11}$ have to be symmetric,  and we handle  lower order terms.  
  In fact, the type of  nonlinearities in $f$ that can be treated  depends 
  on the regularity framework.  For high enough regularity like in  \cite{Kawashima83}, 
  one can consider general terms of the form $f^1=f^1(U,\nabla U^2)$ and $f^2=f^2(U,\nabla U).$ 
    In our first result  (Theorem  \ref{Thm:loc:scri}) where regularity is lower than 
    in Kawashima's work, but still subcritical, one can take $f$ of the form \eqref{f=f(f1,f2)}.
    In our second theorem (pertaining to critical regularity)
    more restrictive assumptions will be made both on  $f$ and on the matrices of the system.     
      The general principle is that in order to be able to close the estimates
    on some nontrivial time interval $[0,T],$ 
    we need to be in a functional framework where all the coefficients of the system are controlled 
    in $L^\infty([0,T]\times\Rd).$
    \end{remark}
  In what follows, we set   $V=(V^1,V^2)$ so that System \eqref{Eq_b} may be rewritten: 
 \begin{equation}
     \label{Eq_b:V1:V2}
     \begin{cases}
     \displaystyle S^0_{11}(U) \pt V^1 \!+\!\sum_{\alpha=1}^d\left(S^\alpha_{11} (U) \pal V^1\!+\!S^\alpha_{12} (U) \pal V^2\right) =  f^1(U)   \\[0,2cm]
     \displaystyle   S^0_{22}(U) \pt V^2 \!+\!\sum_{\alpha=1}^d\left(S^\alpha_{21} (U) \pal V^1\!+\!S^\alpha_{22} (U) \pal V^2\right)=  f^2(U,\n U)\\\hspace{8cm}+\Sum_{ \substack{\alpha,\beta=1}}^d\!{\pal (Z^{\alpha\beta}(U) \pbe V^2)  }.
     \end{cases}
 \end{equation}
Before stating our first local existence result, let us motivate our functional framework. Since our
general approach is based on energy estimates,
 we shall consider spaces  built on $L^2.$ In order to handle some limit cases and to be able to  gain two full 
 derivatives with respect to the regularity of the initial data in the parabolic part of the system, 
 it is suitable to use Besov spaces  of type $B^s_{2,1}$  (see the definition in Appendix \ref{appendix:LP}) rather than the usual  Sobolev spaces $H^s=B^s_{2,2}.$
 Another fundamental point is to have a functional framework that guarantees control in $L^\infty([0,T]\times\R^d)$
 for the coefficients of the system. Unless very particular  assumptions are made on the dependency
 of  $S^0, S^\alpha,Y^{\alpha\beta}$ with respect to $V,$ this leads us to 
 assume that the initial data $V_0$ belongs to $\B{\theta}{2}{1}$ with $\theta\ge \frac{d}{2},$ 
 to ensure the aforementioned $L^\infty$ control of the coefficients of the system.
 Finally,  we have to keep in mind 
  that $V^1$ is governed by a hyperbolic equation (hence no gain of regularity for $V^1$), 
  while, for given $V^1,$ the function   $V^2$ satisfies  a parabolic equation.
  Hence,  starting from $V^2_0\in \B{\theta}{2}{1}$  we expect  $V^2$ to be in  ${C}(0,T;\B{\theta}{2}{1})\cap L^1(0,T; \B{\theta+2}{2}{1})$, provided one can control the source term in $ L^1(0,T; \B{\theta}{2}{1})$, in particular   $ S^\alpha_{21}(U)\pal V^1$. 
  Owing to  product laws, this means that we  need $\n V^1$ to be in $L^1(0,T; \B{\theta}{2}{1}),$ which leads us to considering
     $V^1_0$ in $\B{\theta+1}{2}{1}.$
\smallbreak
   This motivates our  first result, that  can be stated as follows:
\begin{theorem}\label{Thm:loc:scri}
  Let $d\ge 1$ and  $s\ge d/2.$ Under   assumption \textbf{B}, if the initial data satisfies $( V_0^1,V^2_0)\in \B{s+1}{2}{1}\times\B{s}{2}{1} $ and $U_0\defn V_0+\overline{U}$ takes values in a  bounded open subset $\mathcal{O}_0$ of $ \mathcal{U}$
such that $\overline{\mathcal{O}_0}\subset \mathcal{U},$    then there
exists a time $T> 0$ depending only on suitable norms of the data 
and on ${\rm dist}({\cO_0},\partial\cU)$
and such that the following results hold true:
\begin{description}
    \item[Existence] System $\eqref{Eq_b}$ with $U\defn V+\overline U$ supplemented with the initial data $V_0$ has a unique solution $V=(V^1,V^2)$ in the class $E_{T}^s$ defined by
\[V^1\in  {\cC}([0,T];\B{s+1}{2}{1}),\quad
 V^2\in  {\cC}([0,T];\B{s}{2}{1})\cap  L^1_{T}(\B{s+2}{2}{1})\andf  \partial_tV\in L^1_T(\B{s}{2}{1}), \]
and  $U$ belongs to a $d-$neighborhood of  $\overline{\mathcal{O}_0}$ with $d< {\rm dist}( \overline{\mathcal{O}_0}, \partial\mathcal{U}) $.
 \item[Continuation criterion] If $V$ is defined on $[0,T_1[\times\Rd,$  belongs to $E_T^s$ for all $T<T_1,$ and satisfies: \begin{enumerate}
 \item $U([0,T_1[\times\Rd)$ is a compact subset of $\cU,$ \smallbreak
    \item \label{eq:2}  $\displaystyle\int^{T_1}_0\Bigl(\normeinf{\n V}^2+\Big\|\partial_t(S^0_{11}(U))\!+\!\sum_\alpha \partial_\alpha
    (S^\alpha_{11}(U))\Big\|_{L^\infty}
    \!+\!\normeinf{\partial_t (S^0_{22}(U))}\Bigr)<\infty $, \smallbreak
    \item \label{eq:3} 
    $\lVert \n V^1 \rVert_{L^\infty( [0, T_1[\times \Rd)}<\infty$,
\end{enumerate}
then  $V$ may be continued 
on $[0, T^*]\times \Rd$ for some $T^*>T_1$ in a solution of \eqref{Eq_b} which belongs to $E^s_{T^*}$.
\end{description}
\end{theorem}
\begin{remark}
Since the norms that come into play in the continuation criterion are controlled by the regularity 
in the space $E_T^{\frac d2},$ one may deduce that, in the case of smooth data,  the time of existence
is independent of the space $E_T^s$ that is considered.\end{remark}
\begin{remark}\label{rmq:improved}
Condition \ref{eq:3} is not needed, if 
$ f^{22}(U,\n V^1)$ is at most quadratic in $\n V^1.$
Furthermore, if  all the functions $S_{22}^0,$ $Z^{\alpha\beta}$ and $(S_{11}^0)^{-1}S^\alpha_{12}$  only depend on 
$U^1$ and $f^{23}(U,\nabla U^2)$ is affine\footnote{We shall say that a function $K=K(X,Y)$ 
is \emph{affine} in $Y$ if  it is of the form    \begin{align}
    \label{def:f:lin:v}    K(X,Y)=K_1(X)Y+K_2(X).\end{align}} in $\nabla U^2,$ then   Condition \ref{eq:2} reduces to
$$\Int_0^{T_1}\bigl(\|\nabla V^1\|_{L^\infty}^2+\|\nabla V^2\|_{L^\infty}\bigr)
<\infty.$$
\end{remark}
\begin{remark}\label{rmq:comp:serre}
Compared to the results of D. Serre \cite{Serr10} and S. Kawashima \cite{KawaSui88}, 
we here use \emph{different and smaller} regularity indices for $ V^1$ and $V^2$: one may take data in $\B{\cd+1}{2}{1}\times\B{\cd}{2}{1}$ instead of $H^s$ for $s>\cd+1$ in Serre's work and $ s>\cd+2$ in Kawashima's work.  
In fact,  the component $V_0^2$ can be taken in any space $H^s$ with $s>\cd$ and does not need to be Lipschitz.
Finally, although it has been omitted for simplicity, we can prove exactly the same statement 
if we put a  source term in $L^1([0,T];B^{s+1}_{2,1}\times B^s_{2,1})$ in the right-hand side of \eqref{Eq_b}.
\end{remark}

One may wonder  whether  System \eqref{Eq_b} is solvable in a `critical regularity setting' 
as in the Navier-Stokes case. In fact,
since the work of the second author in \cite{Danchin05,Danchin07}, it is  known that the barotropic compressible  Navier-Stokes equations
are well-posed if the initial density and velocity   belong to  $\BH{\frac{d}{2}}{2}{1}$ and $ \BH{\frac{d}{2}-1}{2}{1}$, respectively. Since in the setting of System \eqref{Eq_b}, the density  and 
velocity play the role of $V^1$ and $V^2,$ it is tempting to study whether regularity 
  $\BH{\frac{d}{2}}{2}{1}\times  \BH{\frac{d}{2}-1}{2}{1}$ is enough for $(V_0^1,V_0^2).$ 
  An obvious drawback of this framework is that, since  $\dot B^{\frac d2-1}_{2,1}$ does not control the $L^\infty$ norm, 
    the coefficients of the system cannot be too dependent on $U^2$ and $\nabla U^1$
    (see more explanations below  \eqref{cri:eq_V1_2}).   This  motivates   the following:  
   \paragraph{\textbf{Assumption  C}}   On  $\mathcal{U}$, we have
    \begin{enumerate}
\label{cond_system}
\item   The matrix $S^0(U)$ is  block diagonal and inversible on $\cU,$ 
      $S^0_{22}(U)$ is   symmetric  positive definite
and $S^0_{22}$   depends only on $U^1$.
    \item The matrices $S^\alpha_{21}(U)$, $S^\alpha_{22}(U)$ 
    are  affine  with respect to  $U^2.$
    \item  The matrices $ \wt S^\alpha_{12}\defn
    (S^0_{11})^{-1}S^\alpha_{12}$  depend only on $U^1$ while the matrices $ \wt S^\alpha_{11}\defn(S^0_{11})^{-1}S^\alpha_{11}$ are symmetric,  
    are affine with respect to $U^2,$ and independent of $U^1.$
        \item The functions $Z^{\alpha\beta}$  for $\alpha,\beta=1,\cdots,d$ depend only on $U^1$.
     \item $f^1$ and $f^2$ are  functions of $U$ only, and satisfy $f^1(\overline{U})=0$ and $f^2(\overline{U})=0$.
\end{enumerate}
\smallbreak
 Since we do not have any control on the $L^\infty$ norm of  $U^2_0,$ 
 the phase space $\cU$ cannot be supposed bounded in the $n_2$ last directions. 
 This leads us  to introduce the following set:
\begin{align}
 \label{def_U_1}
      \mathcal{U}^1=\{ U^1\in \mathbb{R}^{n_1}\, /\, \exists U^2\in \mathbb{R}^{n_2}; \: U=(U^1,U^2)\in \mathcal{U} \}.
    \end{align}
\begin{theorem}\label{thm:loc:cri}
    Let the structure assumptions \textbf{C} be in force and let $\cO_0^1$ be a bounded open subset such that $\overline{\cO_0^1}\subset \mathcal{U}^1.$  
     Let $U_0$ be  such that $U^1\in \cO_0^1$, $V^1_0\in \BH{\cd}{2}{1}$ and $V^2_0\in \BH{\cd-1}{2}{1}$ with $V_0=U_0-\overline{U}$.  Then, there exists a positive time $T$ such that System \eqref{Eq_b}  has a
unique solution $V$ with $ U= V+ \overline{U} $ and $ U^1\in \cO^1$, where $\cO^1 $ is a  $d_1-$neighborhood of  $ \cO_0^1$ with $d_1< {\rm dist}({\cO_0^1}, \partial\mathcal{U}^1) $. Moreover $V$ belongs to the space $\mathcal E_T$ defined by 
\[  V^1\in {\cC}([0,T];\BH{\cd}{2}{1}),\;\;  V^2\in  {\cC}([0,T];\BH{\cd-1}{2}{1}) \cap L^1_T(\BH{\cd+1}{2}{1}) \esp{and} \pt V\in L^1_T(\BH{\cd-1}{2}{1}).  \]
\end{theorem}
\begin{remark} It goes without saying that a similar result holds true in the \emph{nonhomogeneous} 
critical space $B^{\frac d2}_{2,1}\times B^{\frac d2-1}_{2,1}.$ We here chose the homogeneous setting to prepare
the ground for our companion paper   \cite{DanADOglob}  dedicated to global well-posedness. 
It is also possible to get a local well-posedness statement in intermediate spaces $B^{s+1}_{2,1}\times B^s_{2,1}$
with $s\in[d/2-1,d/2].$ 
\end{remark}

The rest of this paper unfolds as follows.  In Sect. \ref{sec:proof:loc:sur}, we establish the local existence and continuation criterion for System \eqref{Eq_b} under Assumption {\textbf B}. Sect. \ref{sec:proof:loc:cri} is devoted to the proof of our critical local well-posedness result (Theorem \ref{thm:loc:cri}).  In Appendix \ref{appendix:LP} we briefly 
recall the definition of the Littlewood-Paley decomposition and  review some useful properties of Besov spaces.
 In appendix \ref{appendix:resu}, 
we set out some key results that are of constant use in this article: maximal regularity of the linear parabolic equation, G\r arding inequality, etc.
\medbreak\noindent
{\bf Notation.} In all the paper,  $(c_j)_{j\in \mathbb{Z} }$ stands for a positive  sequence such that  $\left\Vert (c_j)\right\Vert_{l^1(\mathbb{Z})}=1$. Also, $C$ designates a generic  constant, the value of which depends on the context. 
If $X$ is a Banach space, then we denote by  $L^p(0,T;X)$ or $L^p_T(X)$ the Bochner space of   measurable functions $\phi:[0,T]\to X$ such that  $t\mapsto \|\phi(t)\|_X$  lies in the Lebesgue space $L_p(0,T).$  The corresponding norm is denoted by  $\Vert \cdotp\Vert_{L^\ro_T(X)}.$

For  $\Sigma:(0,T)\times\R^d\to \R\times\R^d$  a differentiable function, we set
$\DIV \Sigma\defn\partial_t\Sigma^0+\Sum_{\alpha=1}^d\partial_\alpha\Sigma^\alpha.$


\section{Proof of Theorem \ref{Thm:loc:scri}}\label{sec:proof:loc:sur}
In this section, we prove the local existence of solutions for System \eqref{Eq_b}   
under Assumption~\textbf{B}. 
To simplify the presentation, we assume that, on $\cU,$ the matrix $S^0(U)$ is symmetric definite positive and 
 the matrices $S^\alpha_{11}(U)$ are symmetric. 
To treat the case where  just  $S^0_{22}(U)$  is  symmetric definite positive
and the  matrices $(S^0_{11}(U))^{-1}S^\alpha_{11}(U)$ are symmetric, it 
is only a matter of following the proof of Proposition \ref{energy_V^1} below, instead of using  Proposition \ref{energ:est:V1:lem}. 
\smallbreak
The first step is to establish a priori estimates for the following  linearization of \eqref{Eq_b:V1:V2}:
\begin{equation}
  \label{lin:gen:eq_b}
	\left\{\begin{aligned}
	    &S^0_{11}(U)\pt \tV^1 + \suma S^\alpha_{11}(U) \pal \tV^1 =\Theta^1,  \\[-1ex]
	     &S^0_{22}(U)\pt \tV^2 - \sumab\pal(Z^{\alpha \beta}(U)\pbe \tV^2)=\Theta^2.
	\end{aligned}\right.
	 \end{equation}
We assume that the given function $U:[0,T]\times\Rd\to{\mathbb R}^n$ is smooth, sufficiently decaying at infinity
and  that  there exists  a bounded open   set  $\mathcal{O}$ satisfying 
 $ \overline{\mathcal{O}} \subset \mathcal{U}$ such that:
 \begin{align}
    \label{assuption_U_O_1}
       U(t,x) \in \mathcal{O} \esp{for all} t\in [0,T],\  x\in \Rd.
    \end{align}
Consequently, one can assume that  there exists a constant $C=C(\mathcal{O}, S)$  such that
\begin{align}
\label{norinf:S(U)}
    &\underset{\alpha\in\{0,\cdots,d\}}\sup \Vert S^\alpha(U)\Vert_{L^\infty([0,T]\times \Rd)}\le C, 
   \\ \label{equi_S0}
      &C^{-1}I_n\le S^{0}(U) \le CI_n\quad\hbox{on }\ [0,T]\times\Rd.
    \end{align}
The (given) source terms $\Theta^1$ and $\Theta^2$ are smooth, and we supplement 
the system with  a smooth initial data $\wt V_0.$
\medbreak
Since  there is no coupling between the two equations of \eqref{lin:gen:eq_b}, they  will be  considered
separately:  the equation for $\wt V^1$ will be seen as a hyperbolic symmetric system, while
that for $\wt V^2,$ as a parabolic system. 

The proof of the local existence result then follows from  an iterative scheme where $V_{p+1}$
is  the solution of the linear hyperbolic/parabolic system \eqref{lin:gen:eq_b} with 
 $f,$ and matrices  $S^\alpha_{jk}$ and $Z^{\alpha \beta}$ computed at $\overline U+ V_p.$
 The main difficulty is to exhibit a positive time $T$ such that  the sequence $(V_p)_{p\in\N}$ is bounded
 in the space $E_T^s$.
 Then, as for hyperbolic systems, we will be able to prove convergence only for 
  a weaker norm corresponding to \textit{a loss of one
derivative}. The same restriction occurs as regards the uniqueness issue and, in a last step, 
we will have to take advantage of functional analysis arguments to establish that, indeed, 
the limit satisfies the nonlinear system and belongs to the space $E_T^s.$ 
Since we have relatively high regularity, this loss of derivative is  harmless, 
except in the case $d=1$ and $s=1/2$ that will be briefly discussed  at the end of this section.


\subsection{A priori estimates for a  linear hyperbolic system}

Here we concentrate on the first equation  of \eqref{lin:gen:eq_b}. Before starting the proof, let us fix the following notation:
     \begin{equation}     \label{def:alpha*:**}
     \theta^{**}\defn \max\Bigl(\frac{d}{2},  \theta-1\Bigr)
    \esp{and} \theta^* =\max\Bigl(\frac d2,\theta\Bigr)\cdotp  \end{equation}
Besides, for any tempered distribution $W$ and $j\geq-1,$ we shall denote $W_j\defn\Dj W. $

\begin{proposition}\label{energ:est:V1:lem}
   Let $\sigma>-d/2.$ There exists  a constant $C_0$ depending only on $\overline {S^0_{11}}\defn S^0_{11}(\bar U)$ and a constant 
   $C$ depending on $\sigma,$ $\cO$ and on all the coefficients of the system such that 
     for all $t\in [0,T]$, the following inequality holds: 
    \begin{equation}
        \label{est:tV1:lin}
       \lVert \tV^1 \rVert_{{L}^\infty_t(\B{\sigma}{2}{1})} \leq C_0 
       \biggl(\NB{\sigma}{2}{1}{\tV_0^1}+\int^t_0  \NB{\sigma}{2}{1}{\Theta^1}+
       C\int_0^t \Phi_1 \NB{\sigma}{2}{1}{\tV^1}\biggr),
    \end{equation}
    where $\Phi_1(t)\defn \normeinf{\DIV(S_{11}(U))}+\NB{\sigma^{**}+1}{2}{1}{V}.$
    
    Furthermore, we have 
    \begin{equation}\label{eq:dttV1}
    \int_0^t\|\partial_t\wt V^1\|_{B^{\sigma-1}_{2,1}} \leq C
  \int_0^t \bigl(1+\|V\|_{B^{\sigma^{**}}_{2,1}}\bigr) \bigl(\|\wt V^1\|_{B^{\sigma}_{2,1}}+\|\Theta^1\|_{B^{\sigma-1}_{2,1}}\bigr)\cdotp\end{equation}    
\end{proposition}
\begin{proof}
Applying the non-homogeneous dyadic block $\Delta_j$ to the first  equation of \eqref{lin:gen:eq_b}$_1$ yields
	    \[S^{0}_{11}(U)\pt  \tV^1_j + \suma S^\alpha_{11}(U)\pal \tV^1_j= R^{11}_j +\Theta^1_j,\]
	    where we define 
	    $$\begin{aligned}
	    R_j^{11} \defn S^0_{11}(U)&\suma[\wt S_{11}^\alpha(U),\Dj]\pal \tV^1\with \wt S^\alpha_{11}\defn
	    (S^0_{11})^{-1}\circ S^\alpha_{11},\\\andf
     \Theta^1_j \defn S^0_{11}(U)&\suma \Dj\left( ( S^0_{11}(U))^{-1} \Theta^1\right)\cdotp
     \end{aligned}$$
     
	    Next, taking the scalar product in ${\mathbb R}^{n_1}$ of this equation with  $\wt V^1_j$, integrating on $\Rd$  along with integration by parts and using the symmetry properties of $S^\alpha_{11}(U)$ gives:
     \begin{align*}
       \frac{1}{2} \frac{d}{dt}\intd S^0_{11}(U) \tV^1_j\cdot\tV^1_j =  \frac{1}{2}\intd \left(\DIV(S_{11}(U))\right)\tV^1_j\cdot \tV^1_j +\intd(R^{11}_j +\Theta^1_j)\cdot\tV^1_j.
     \end{align*}
Cauchy-Schwarz inequality, inequalities \eqref{norinf:S(U)} and \eqref{equi_S0} 
lead   for some $C=C( \mathcal{O})$ to
     \begin{multline*} 
      \frac{d}{dt}\intd S^0_{11}(U) \tV^1_j\cdot\tV^1_j \le \normeinf{\DIV(S_{11}(U))}  \intd  S^0_{11}(U) \tV^1_j\cdot\tV^1_j \\
          + C_0\normede{( R^{11}_j , \Theta^1_j)}\sqrt{\intd  S^0_{11}(U) \tV^1_j\cdot \tV^1_j}\,\cdotp 
            \end{multline*}
     Then, from Lemma \ref{lem_der_int} with $\displaystyle X=  \intd S^0_{11}(U) \tV^1_j\cdot\tV^1_j$ and \eqref{equi_S0}, one gets that for all $t\in[0,T],$
     \begin{equation}
         \label{est:V1:Li:L2}
         \Vert \tV^1_j(t)\Vert_{L^2} \le C_0\normede{\tV^1_{0,j}}+ C\int^t_0 \Bigl(\normeinf{\DIV(S_{11}(U))} \normede{\tV_j^1}+ \Vert ( R^{11}_j ,\Theta^1_j )\Vert_{L^2}\Bigr)\cdotp
     \end{equation}
     To  bound the terms $  R^{11}_j$ in $L^2,$ we put
 Inequality \eqref{comm:est:a:b:sig>0} and Proposition \ref{propo_produc_BH} together, 
and obtain if $\sigma\ge \frac{d}{2}+1$, 
\begin{multline}
\label{est:R11:gene}
    \normede{R^{11}_j}  \le Cc_j2^{-j\sigma} \Bigl(\normeinf{ \nabla \left( \wt S_{11}^\alpha(U) -\wt S_{11}^\alpha(\overline{U}) \right) }\NB{\sigma}{2}{1}{\tV^1}\\
   + \normeinf{\n \tV^1}\NB{\sigma-1}{2}{1}{  \nabla \left(\wt S_{11}^\alpha(U)-\wt S_{11}^\alpha(\overline{U}) \right)}\Bigr)\cdotp
\end{multline}
Taking advantage of the embedding $ \B{\sigma-1}{2}{1} \hookrightarrow  L^\infty$, the previous inequality may be simplified as follows, for all $\sigma\ge \cd+1$ for some $ C=C(\mathcal{O})$:
\begin{equation}\label{esti_R_11}
   \normede{R^{11}_j}  
    \le Cc_j2^{-j\sigma}  \NB{\sigma}{2}{1}{ V}\NB{\sigma}{2}{1}{\tV^1}.
\end{equation}
 For $-\cd<\sigma\le \cd+1$, we  combine \eqref{comm:est:a:b},  Proposition \ref{propo_compo_BH} and the embedding $ \B{\cd}{2}{1} \hookrightarrow  L^\infty\cap \B{\cd}{2}{\infty}$ to get
\begin{align}
    \label{esti_R_11:sigma:cri}
    \normede{R^{11}_j}  
    &\le Cc_j2^{-j\sigma}  \NB{\cd+1}{2}{1}{ V}\NB{\sigma}{2}{1}{\tV^1}.
\end{align}
Plugging \eqref{esti_R_11} (or \eqref{esti_R_11:sigma:cri}) into \eqref{est:V1:Li:L2} yields for all $t\in[0,T],$
\begin{multline*}
2^{j\sigma}\Vert \tV^1_j(t)\Vert_{L^2} \le C_02^{j\sigma}\normede{\tV^1_{0,j}}+ C2^{j\sigma} \int^t_0\Vert \tV^1_j\Vert_{L^2}\normeinf{\DIV(S_{11}(U))}\\
    +Cc_j\int^t_0\NB{\sigma^{**}+1}{2}{1}{ V} \NB{\sigma}{2}{1}{\tV^1}+C_02^{j\sigma}\int^t_0\Vert\Theta^1_j \Vert_{L^2}.
\end{multline*}
 Then, summing over $j\ge -1$ gives  Inequality \eqref{est:tV1:lin}.
 \smallbreak
 In order to prove \eqref{eq:dttV1},  it suffices to use the relation
 \begin{equation}  \label{eq:pt:tV1}
  \pt \tV^1 =- \suma \wt S^\alpha_{11}(U) \pal \tV^1 + (S^0_{11}(U))^{-1}\Theta^1.
	 \end{equation}
	 Then, the result follows from Propositions \ref{propo_produc_BH} and  \ref{propo_compo_BH}.
 \end{proof}


\subsection{A priori estimates for a  linear parabolic system}

\begin{proposition}\label{energ:est:V2:lem}
     Let  $s>-d/2.$ There exists a constant $C_0$ depending only on $\overline{S^0_{22}}\defn S^0_{22}(\bar U),$
 a constant $c$ depending only on the  ellipticity constant in \eqref{strong_elli}, 
 and a constant $C$ depending on $\cO,$ $s$ and on the coefficients of the system 
  such that for all $t\in[0,T],$ 
    \begin{equation}
     \label{est:tV2:lin}
     \lVert \tV^2 \rVert_{\widetilde{L}^\infty_t(\B{s}{2}{1})}+ c\lVert \tV^2 \rVert_{L^1_t(\B{s+2}{2}{1})} \le C_0\biggl(\NB{s}{2}{1}{\tV^2_0}+\int^t_0 \NB{s}{2}{1}{\Theta^2} +C\int_0^t\Phi_2 \NB{s}{2}{1}{\tV^2}\biggr),
\end{equation}
where $    \Phi_2\defn  1+\normeinf{\pt(S^0_{22}(U))}+ (1+\NB{s^*}{2}{1}{V})^2
     \lVert V\Vert_{\B{s^{*}+1}{2}{1}}^2.$
     \medbreak
     Furthermore, we have for all $t\in[0,T],$ 
     \begin{multline}\label{eq:dttV2}
\int_0^t\|\pt\wt V^2\|_{B^s_{2,1}}\leq C
\int_0^t\bigl(1+\|V\|_{B^{s^*}_{2,1}}\bigr)\Bigl(\|\wt V^2\|_{B^{s+2}_{2,1}}+\|\nabla\wt V^2\|_{B^{s^*}_{2,1}}\|V\|_{B^{s+1}_{2,1}}
+\|\Theta_2\|_{B^s_{2,1}} \Bigr)\cdotp\end{multline}    \end{proposition}    
\begin{proof}
Let $\wt Z^{\alpha\beta}\defn (S^0_{22})^{-1}\circ Z^{\alpha\beta}.$
Applying $S^0_{22}(U)\Dj (S^0_{22}(U))^{-1}$ to 
 \eqref{lin:gen:eq_b}$_2$ gives
\begin{align*}    S^0_{22}(U)\pt \tV^2_j  - \sumab Z^{\alpha \beta}(U)\pal\pbe \tV^2_j=\Theta^2_j+R_j^2,\end{align*}
 with
\begin{align*}
 \Theta^2_j &\defn  S^0_{22}(U) \Dj((S^0_{22})^{-1}(U) \Theta^2),\\
     R_j^{2} &\defn S^0_{22}(U) \sumab \biggl(\left[\Dj,\wt Z^{\alpha\beta}(U)\right]\pal\pbe \tV^2 
     +\Dj\left((S^0_{22})^{-1}(U)\pal(Z^{\alpha\beta}(U))\pbe \tV^2\right)\biggr)\cdotp
\end{align*}
Taking the $\Lde(\Rd;\R^{n_2})$ inner product of the above equation with $\tV^2_j$ yields  for $j\ge -1$
	 \begin{multline}
  \label{esti_v_2-2}
	    \frac{1}{2} \frac{d}{dt}\intd{ S^0_{22}(U) \tV^2_j\cdot \tV^2_j}   -\sumab\intd{ Z^{\alpha \beta}(U)\pal\pbe \tV_j^2\cdot \tV_j^2 } =\frac{1}{2}\intd{ (\pt(S^0_{22}(U)) )\tV_j^2\cdot\tV_j^2 }\\
+\intd{(R^2_j+\Theta^2_j)\cdot\tV_j^2 }.
	\end{multline}
Under Condition \eqref{strong_elli}, we have by making use of   Lemma \ref{gard_lem}, for all  $j\ge 0,$
\begin{multline*}
    -\sumab\intd{ Z^{\alpha \beta}(U)\pal\pbe \tV_j^2\cdot \tV_j^2 } \ge c\normede{\n \tV_j^2 }^2-\varepsilon\normede{\n^2 \tV_j^2}\normede{ \tV_j^2}-C({\varepsilon}, \mathcal{O} )\Vert  \tV_j^2\Vert^2_{L^2} 
\end{multline*}
where $c$ is positive constant depending on $\mathcal{O},$ and $\varepsilon>0.$
 Owing to  Bernstein inequality and for $\varepsilon$ small enough, we deduce that
 for some constant $C$ depending only on $Z$ and on $\cO,$ 
 \begin{eqnarray}
    \label{est:j=:sur:VS}
\qquad    -\sumab\intd{ Z^{\alpha \beta}(U)\pal\pbe \tV_j^2\cdot \tV_j^2 } \ge 2^{2j}\frac{c}{2}\Vert \tV_j^2 \Vert_{L^2(\mathbb{R})}^2-C \Vert  \tV_j^2\Vert^2_{L^2(\mathbb{R}^d)} \esp{for} j\ge 0,\\\label{est:j=-1:sur:VS}
\qquad \esp{and}     \sumab\intd{ Z^{\alpha \beta}(U)\pal\pbe  \tV_{-1}^2\cdot  \tV_{-1}^2 } \le C \normede{ \n^2 \tV_{-1}^2 }\normede{ \tV_{-1}^2 } \le C  \normede{ \tV_{-1}^2 }^2.
\end{eqnarray}
  Hence, from \eqref{norinf:S(U)}, \eqref{equi_S0} and using \eqref{est:j=:sur:VS}, \eqref{est:j=-1:sur:VS}, Inequality \eqref{esti_v_2-2} becomes for all $j\ge -1$:
  $$\displaylines{
    \frac{d}{dt}\intd{ S^0_{22}(U) \tV^2_j\cdot \tV^2_j} + c2^{2j}\intd S^0_{22}(U) \tV^2_j\cdot \tV^2_j \le C_0(1+\normeinf{\pt(S_{22}^0(U))})\intd S^0_{22}(U) \tV^2_j\cdot \tV^2_j\hfill\cr\hfill
       +C_0\normede{(R^2_j,\Theta^2_j)}  \sqrt{\intd S^0_{22}(U) \tV^2_j\cdot \tV^2_j}.}$$
	 So, using Lemma \ref{lem_der_int} and, again, \eqref{equi_S0},
  one gets for all $t\in [0,T]$ and $j\ge -1$,
	  \begin{multline}
 \label{esti_v_2-2'}
      \normede{\tV^2_j(t)}+c2^{2j}\int^t_0\normede{\tV_j^2} \le C_0\biggl( \normede{\tV^2_{0,j}}\\+  \Vert  \int^t_0\left(1+ \normeinf{\pt(S_{22}^0(U))}\right)\Vert\tV^2_j\Vert_{L^2}
        +\int^t_0\normede{(R^2_j,\Theta^2_j)}\biggr)\cdotp
	       	\end{multline}
 Owing to   \eqref{norinf:S(U)}, we have for some $C=C(\mathcal{O})$,
\begin{align*}
& \left\Vert R^{2}_j \right\Vert_{L^2} 
    \le C \sumab \biggl(\left\Vert\left[\Dj,\wt Z^{\alpha\beta}(U) \right]\pal\pbe \tV^2 \right\Vert_{L^2}
    + \left\Vert \Dj\bigl((S^0_{22})^{-1}(U)\pal(Z^{\alpha\beta}(U))\pbe \tV^2\bigr) \right\Vert_{L^2}\biggr)\cdotp
\end{align*}
 Taking  $\theta=s>-d/2$ 
in Proposition \ref{propo_produc_BH} and combining with Proposition \ref{propo_compo_BH} and suitable 
 embedding, we discover that
\begin{align*}
      \NB{s}{2}{1}{(S^0_{22}(U))^{-1}\pal(Z^{\alpha\beta}(U))\pbe \tV^2 }\le C (1+\NB{s^*}{2}{1}{V}) \NB{s^{*}+1}{2}{1}{V}\NB{s+1}{2}{1}{\tV^2}.
\end{align*} 
Next, with the aid of inequality \eqref{comm:est:a:b:sig>0} and Proposition \ref{propo_compo_BH} one obtains for $s\ge \cd+1,$
\begin{multline*}
 \sum_{j\geq-1}    \sumab \!2^{js}   \big\|{\big[\Dj,\wt Z^{\alpha\beta}(U) \big]\pal\pbe \tV^2}\big\|_{L^2}
 \le C \Bigl(\normeinf{\n V}\!\NB{s}{2}{1}{\n \tV^2} \!+ \normeinf{\n ^2 \tV^2}\!\NB{s}{2}{1}{ V}\Bigr)\cdotp
\end{multline*}
The previous inequality may be  simplified by using Besov embedding. We have
\begin{align*}
   \sum_{j\geq-1}     \sumab2^{js}
   \big\|{\big[\Dj,\wt Z^{\alpha\beta}(U) \big]\pal\pbe \tV^2}\big\|_{L^2}
\le C\NB{s-1}{2}{1}{\n V}\NB{s}{2}{1}{\n \tV^2}.
\end{align*}
If  $-\cd<s\le \cd+1$, then Inequality \eqref{comm:est:a:b} combined with  Proposition \ref{propo_compo_BH}
and the embedding $  \B{\cd}{2}{1} \hookrightarrow  L^\infty\cap \B{\cd}{2}{\infty}$ give
\begin{align}
\label{est:comm:tV2:gen:sig:s}
   \sum_{j\geq-1}     \sumab2^{js}
   \lVert{\left[\Dj,\wt Z^{\alpha\beta}(U) \right]\pal\pbe \tV^2}\rVert_{L^2}
\le C\NB{\cd+1}{2}{1}{ V}\NB{s}{2}{1}{\n \tV^2}.
\end{align}
Reverting to \eqref{esti_v_2-2'} then  integrating on $[0,t]$ and summing over $j\ge -1$ 
implies:
\begin{multline}    \label{esti_gene}
    \LptNB{s}{2}{1}{\tV^2}{\infty}\!+\!c\LptNB{s+2}{2}{1}{\tV^2}{1} \le 
    C_0\biggl(\NB{s}{2}{1}{\tV^2_0}\!\!+\!\int^t_0 \Bigl(\left(1\!+\! \normeinf{\pt(S_{22}^0(U))}\right)\NB{s}{2}{1}{\tV^2}\\+\NB{s}{2}{1}{\Theta^2}\Bigr)\biggr)
+ C\int^t_0\Bigl((1+\NB{s^*}{2}{1}{V})\NB{s^* +1}{2}{1}{V}\NB{s+1}{2}{1}{\tV^2}+\NB{s^{**}+1}{2}{1}{V}\NB{s+1}{2}{1}{\tV^2}\Bigr)\cdotp
    \end{multline}
Using interpolation  and Young's inequality yields 
$$(1\!+\!\NB{s^*}{2}{1}{V}) \NB{s^*+1}{2}{1}{V}\NB{s+1}{2}{1}{\tV^2}
    \le \frac{c}{4} \NB{s+2}{2}{1}{\tV^2} 
   \! +\! C (1\!+\!\NB{s^*}{2}{1}{V})^2\NB{s^*+1}{2}{1}{V}^2\NB{s}{2}{1}{\tV^2}.$$
Similarly, we have
$$ \NB{s^{**}+1}{2}{1}{V}\NB{s+1}{2}{1}{\tV^2}\le \frac{c}{4} \NB{s+2}{2}{1}{\tV^2} + C\NB{s^{**}+1}{2}{1}{V}^2\NB{s}{2}{1}{\tV^2}.$$
    Plugging these inequalities  into \eqref{esti_gene}  and observing that 
 $s^*\geq s^{**},$  we get \eqref{est:tV2:lin}.
\smallbreak
Finally, to bound $\pt \tV^2$, we use the relation 
\begin{align}
      \label{eq:pt:tV2}
       \pt \tV^2  = (S^0_{22}(U))^{-1}\sumab \pal(Z^{\alpha \beta}(U)\pbe \tV^2) + (S^0_{22}(U))^{-1}\Theta^2.
\end{align}
Hence, using again  Propositions \ref{propo_produc_BH} and Proposition \ref{propo_compo_BH}, we discover that
$$\begin{aligned}
\|\pt\wt V^2\|_{B^s_{2,1}}
\leq C\bigl(1+\|V\|_{B^{s^*}_{2,1}}\bigr)
\biggl(\|\wt V^2\|_{B^{s+2}_{2,1}}+\|\nabla\wt V^2\|_{L^\infty}\|V\|_{B^{s+1}_{2,1}}
+\|\Theta_2\|_{B^s_{2,1}} \biggr)\cdotp
\end{aligned}$$ 
Using suitable embedding, we get Inequality \eqref{eq:dttV2}. 
\end{proof}


\subsection{Estimates for the linearized coupled system}\label{ss:coupled} 

For  given smooth functions $U$ with range in $\mathcal{U}$ 
 we  consider
 the following linear system with variable coefficients:
 \begin{equation}
  \label{lin:eq_b}
	\begin{cases}
	    &S^0_{11}(U)\pt \tV^1 + \suma S^\alpha_{11}(U) \pal \tV^1 =\Theta^1(U) \\[0.2cm]
	     &S^0_{22}(U)\pt \tV^2 - \sumab\pal(Z^{\alpha \beta}(U)\pbe \tV^2)=\Theta^2(U)
	\end{cases}
	 \end{equation}
supplemented with initial data
\begin{equation}
    \label{data:lin}
   \tV_{|t=0}=\tV_0=(\tV_0^1,\tV_0^2)\in B^{s+1}_{2,1}\times B^s_{2,1}\with s\geq\cd\cdotp
\end{equation}
The functions $\Theta^1$ and $\Theta^2$ are given by (see Assumption \textbf{B} for the conditions on $f$): 
\begin{align}
    \label{def:Theta1:2:invari}
    \begin{split}
     \Theta^1(U)&\defn f^1(U) -\suma S^\alpha_{12}(U) \pal V^2,\\[-2ex]
    \Theta^2(U) &\defn f^2(U,\n U)-\suma\left(S^\alpha_{21}(U) \pal V^1 + S^\alpha_{22}(U) \pal V^2\right)\cdotp
      \end{split}
\end{align}
Our aim is to prove that if, for some given $R\geq1$  
$$\max\Bigl(\|\wt V_0^1\|_{B^{s+1}_{2,1}}, \|\wt V_0^2\|_{B^{s}_{2,1}}, 
\|V^1\|_{L^\infty_T(B^{s+1}_{2,1})} + \|V^2\|_{L^\infty_T(B^{s+1}_{2,1})}\Bigr)\leq R,$$
then the same property holds for $(\wt V^1,\wt V^2)$ provided  $T$ is  small enough. 
\smallbreak
We plan to bound $\wt V^1$ and $\wt V^2$ by means of Propositions  \ref{energ:est:V1:lem} and \ref{energ:est:V2:lem}. 
In the case of large $R$ however, a difficulty arises in some terms of $\Theta^1.$ For example, 
we have  for all $\alpha=1,\cdots, d,$
\begin{align*}
  \LpNB{s+1}{2}{1}{S^\alpha_{12}(U)\pal V^2}{1}
     &\le C \int^T_0 \left(\NB{s+1}{2}{1}{V}\NB{\frac{d}{2}+1}{2}{1}{V^2}+ ( 1+\NB{\frac{d}{2}}{2}{1}{V})\NB{s+2}{2}{1}{V^2}\right)\\
    & \le C (1+R) \LpNB{s+2}{2}{1}{V^2}{1}.
\end{align*}
In order to ensure that the contribution of this term in the estimate of $\wt V^1$
is smaller than $R,$ we need to know that  $\LpNB{s+2}{2}{1}{V^2}{1}$ is very small. 
Although Lebesgue dominated convergence theorem guarantees that this is true  when $T$ goes to zero,
we need a more precise information for constructing the solutions. 
To achieve it, we decompose  $V^2$ and $\wt V^2$ as follows:
$$V^2=V^2_L+V_S\andf  \wt V^2=V^2_L+\wt V_S,$$ where 
 $V^2_L$ is the solution  the following linear parabolic system 
 with constant coefficients:
	    \begin{align}\label{eq_lin-dif}	 \begin{cases}
	    \overline{S^0_{22}}\pt V^2_L  - \overline{Z}^{\alpha \beta}\pal\pbe V^2_L=0  \\[1ex]
	      V^2_L(0)= \tV^2_0.	 \end{cases}	\end{align}
The new unknown  $\wt V_S$ thus  satisfies 
  \begin{equation}
	 \label{eq_small-V_S}
	  	 S^0_{22}(U) \pt \tV_S  - \sumab\pal(Z^{\alpha\beta}(U) \pbe \tV_S)=\Theta_S
	 \end{equation}
	    where 
	    \begin{multline}\label{eq:Theta_S}
	    \Theta_S\defn  (\overline{S^0_{22}}-S^0_{22}(U))\pt V_L^2+\sumab\pal(Z^{\alpha \beta}(U)- \overline{Z}^{\alpha \beta})\pbe V_L^2    \\+ f^2(U,\n U)-\suma\left(S^\alpha_{21}(U) \pal V^1 + S^\alpha_{22}(U) \pal V^2\right)\cdotp\end{multline}
  Let us fix some $R\geq1$ and   make the following assumptions:   
    \begin{enumerate}
    \item [$(\cH_1)$] $\max\bigl(\LpNB{s+1}{2}{1}{V^1}{\infty}, \LpNB{s}{2}{1}{V^2}{\infty}\bigr)  \le R$,
  \item [$(\cH_2)$]    $\|\pt V_L^2\|_{L^1_T(B^s_{2,1})}+ \|V_L^2\|_{L^1_T(B^{s+2}_{2,1})}\leq\eta^2$,  
    \item [$(\cH_3)$] $\LpNB{s}{2}{1}{V_S}{\infty} + \LpNB{s+2}{2}{1}{V_S}{1} 
    + \LpNB{s}{2}{1}{\pt V_S}{1}\leq\eta $,
     \item [$(\cH_4)$]  $\LpNB{s}{2}{1}{\pt V}{1}\leq\sqrt\eta $,
     \item [$(\cH_5)$] $V([0,T]\times\R^d)\subset\cO\subset\!\subset\cU.$
\end{enumerate}
We claim that if the initial data satisfy 
     \begin{equation}\label{eq:wtdata}2C_0\max\bigl(\NB{s+1}{2}{1}{\wt V^1_0},\NB{s}{2}{1}{\wt V^2_0}\bigr)\leq R\end{equation}
     where $C_0$ has been  defined in 
     Propositions \ref{energ:est:V1:lem} and \ref{energ:est:V2:lem}, and if $\eta\in(0,1)$ and $T\in(0,1)$ are small enough then the following set 
 \begin{equation}     \label{def:esp:ex:loc}
    E^s_{T,R,\eta}\defn \Big\{ Z\in E_T^s :\;  \;\text{ Conditions $ (\mathcal{H}_1)-(\mathcal{H}_5)$ are satisfied}  \Big\}
 \end{equation} 
is invariant under the mapping  $ V \mapsto \tV$ with $\tV$ satisfying \eqref{lin:eq_b} and \eqref{data:lin}. 
\smallbreak
As a first, note that Condition $(\cH_2)$ only depends on $\tV_0^2.$  In fact, putting Inequalities \eqref{L-1_j}
and \eqref{L-1_las} together yields for all $t,h\geq0,$
\begin{equation}\label{eq:unifVL}
\int_{t}^{t+h}\Bigl(\|\partial_tV_L^2\|_{B^s_{2,1}}+\|\partial_tV_L^2\|_{B^{s+2}_{2,1}}\Bigr)
\leq C\biggl(h+\sum_{j\geq0} e^{-c2^{2j}t}\bigl(1-e^{-c2^{2j}h}\bigr)2^{js}\|\Dj V_0^2\|_{L^2}\biggr)\cdotp
\end{equation}
Hence, one gets $(\cH_2)$ whenever\footnote{Lebesgue dominated convergence theorem ensures that $T_0$ is indeed positive.}
\begin{equation}\label{eq:condH2}
T\leq  T_0\defn \sup\Bigl\{h>0\,/\, h+\sum_{j\geq0} \bigl(1-e^{-c2^{2j}h}\bigr)2^{js}\|\Dj \wt V_0^2\|_{L^2}\leq\eta^2/C\Bigr\}\cdotp
\end{equation} 
\smallbreak
In order to  verify $(\cH_1)$ and $(\cH_4)$  for  $\wt V^1,$ 
let us apply  Proposition \ref{energ:est:V1:lem} to the first equation of \eqref{lin:eq_b}. 
Since, by the chain rule and embeddings, 
$$
\|\DIV(S_{11}(U))\|_{L^\infty}\leq C_\cO \bigl(1+\|V\|_{L^\infty}\bigr)\|\nabla_{t,x}V\|_{L^\infty}
\leq C_\cO \bigl(1+\|V\|_{L^\infty}\bigr)\bigl(\|\pt V\|_{B^s_{2,1}}+\|V\|_{B^{s+1}_{2,1}}\bigr),$$
we readily have
\begin{multline}\label{eq:tV1}   \lVert \tV^1 \rVert_{\widetilde{L}^\infty_t(\B{s+1}{2}{1})} \leq C_0 
       \biggl(\NB{s+1}{2}{1}{\tV_0^1}+\int^t_0  \NB{s+1}{2}{1}{\Theta^1}+
       C\int_0^t \Psi_1  
        \NB{s+1}{2}{1}{\tV^1}\biggr)\\\with  \Psi_1\defn 
         \bigl(1+\|V\|_{B^s_{2,1}}\bigr)\bigl(\|\pt V\|_{B^s_{2,1}}+\|V\|_{B^{s+1}_{2,1}}\bigr)\cdotp\end{multline} 
    Bounding $\Theta^1$ just follows from  Propositions \ref{propo_produc_BH} and \ref{propo_compo_BH}: we have
    $$\begin{aligned}
    \|f^1(U)\|_{B^{s+1}_{2,1}}&\leq C_\cO\bigl(\|V^1\|_{B^{s+1}_{2,1}}+\|V^2\|_{B^{s+1}_{2,1}}\bigr),\\
      \NB{s+1}{2}{1}{S^\alpha_{12}(U)\pal V^2}&\leq C_\cO
     \bigl(\NB{s+1}{2}{1}{V}\|{\nabla V^2}\|_{L^\infty}+ \NB{s+2}{2}{1}{V^2}\bigr)\cdotp
    \end{aligned}$$
    Hence, using embedding and H\"older inequality, 
    \begin{multline}\label{eq:Theta1}\|\Theta^1\|_{L^1_T(B^{s+1}_{2,1})}\leq C_\cO\Bigl(T\|V^1\|_{L^1_T(B^{s+1}_{2,1})}
    +\sqrt T\|V^2\|_{L^2_T(B^{s+1}_{2,1})} + \|V^2\|_{L^2_T(B^{s+1}_{2,1})}^2   \\+ \bigl(1+\|V^1\|_{L^\infty_T(B^{s+1}_{2,1})}\bigr)
    \bigl(\|V_S\|_{L^1_T(B^{s+2}_{2,1})}+ \|V^2_L\|_{L^1_T(B^{s+2}_{2,1})}\bigr)\biggr)\cdotp\end{multline}
    In what follows, we shall often use the fact that, owing to an interpolation inequality, 
    $$  \|V^2\|_{L^2_T(B^{s+1}_{2,1})}\leq \|V^2\|_{L^\infty_T(B^{s}_{2,1})}^{1/2}\|V^2\|_{L^1_T(B^{s+2}_{2,1})}^{1/2}. $$
Since Proposition \ref{lem_lin_est} guarantees that 
\begin{equation}\label{eq:V2Linfty}
\|V^2_L\|_{L^\infty_T(B^s_{2,1})} \leq C_0\|V^2\|_{B^s_{2,1}}\leq R/2,\end{equation} 
using also $(\cH_2)$ and $(\cH_3)$ yields  
    \begin{equation}\label{eq:V2L2}     \|V^2\|_{L^2_T(B^{s+1}_{2,1})}\leq R^{1/2}\eta. \end{equation}
    Hence, reverting to \eqref{eq:Theta1} and using also $(\cH_1)-(\cH_2),$ we conclude that
    $$\|\Theta^1\|_{L^1_T(B^{s+1}_{2,1})}\leq C_\cO\bigl(TR+\sqrt{RT}\eta
    +R\eta\bigr)\cdotp$$
        We observe that (use \eqref{eq:V2Linfty}, \eqref{eq:V2L2}, $(\cH_1)$ and $(\cH_4)$):
      \begin{equation}  \label{eq:Psi1} \int_0^T\Psi_1\leq R\bigl(\sqrt\eta + TR+\sqrt{RT}\eta\bigr)\cdotp
      \end{equation}  
   Hence, assuming that $T$ and $\eta$ have been chosen so that
     $$R\bigl(\sqrt\eta + TR+\sqrt{RT}\eta\bigr)\ll1$$
     and using Gronwall lemma in \eqref{eq:tV1}, we end up with 
    \begin{equation}\label{eq:V1fin}
        \lVert \tV^1 \rVert_{\widetilde{L}^\infty_T(\B{s+1}{2}{1})} 
        \leq \frac34 R+   C_\cO\bigl(TR+\sqrt{RT}\eta
    +R\eta\bigr)\cdotp\end{equation}
    Therefore, the first part of $(\cH_1)$ is satisfied whenever $\eta$ and $T$ are chosen so that 
    $$     C_\cO\bigl(T\sqrt R+\sqrt{T}\eta +\sqrt R\eta\bigr)\leq \sqrt R/4.$$
    Next, let us check that $\pt \wt V^1$ satisfies $(\cH_4).$ 
    We know from Inequality \eqref{eq:dttV1} that 
    $$
    \|\pt\wt V^1\|_{L_T^1(B^s_{2,1})}\leq C \bigl(1+\|V\|_{L^\infty_T(B^{s}_{2,1})}\bigr)
 \bigl(\|\wt V^1\|_{L_T^1(B^{s+1}_{2,1})}+\|\Theta^1\|_{L^1_T(B^{s}_{2,1})}\bigr)\cdotp$$
 Hence using $(\cH_1)$ and the inequalities we have just proved for $\wt V^1$ and $\Theta^1,$ we get
  $$    \|\pt\wt V^1\|_{L_T^1(B^s_{2,1})}\leq CR\bigl(TR + \eta(R+\sqrt{RT})\bigr)\cdotp$$ 
    It is clear that if one chooses $\eta$ and $T$ small enough, then one can ensures 
    $(\cH_4)$ for $\wt V^1.$
    \medbreak
    Next, let us prove   $(\cH_3).$  To start with, applying 
     Proposition \ref{energ:est:V2:lem} to \eqref{eq_small-V_S} and remembering that $\tV_S|_{t=0}=0$ immediately gives
     \begin{equation}\label{eq:tVS}   \lVert \tV_S \rVert_{\widetilde{L}^\infty_T(\B{s}{2}{1})}+ c\lVert \tV_S \rVert_{L^1_T(\B{s+2}{2}{1})} \le C_0\biggl(\int^T_0 \NB{s}{2}{1}{\Theta_S} +C\int_0^T\Psi_2 \NB{s}{2}{1}{\tV_S}\biggr),
\end{equation}
with $\Theta_S$ defined in \eqref{eq:Theta_S} and 
\begin{equation}\label{eq:Psi2}
 \Psi_2\defn  1+\normeinf{\pt(S^0_{22}(U))}+ (1+\NB{s}{2}{1}{V})^2
     \lVert V\Vert_{\B{s+1}{2}{1}}^2.\end{equation}
     From the chain rule, we have 
$$\normeinf{\pt(S^0_{22}(U))}\leq C_\cO\normeinf{\pt V}.$$
     Consequently, using the embedding $B^s_{2,1}\hookrightarrow L^\infty,$ hypotheses $(\cH_1)-(\cH_5)$  and Inequality \eqref{eq:V2L2} yields
     \begin{align}\label{eq:Phi2}
   \qquad\quad  \int_0^T\!\Psi_2&\leq T+C_\cO\biggl(\|\partial_tV\|_{L^1_T(B^s_{2,1})}\nonumber\\
 &\hspace{3cm}    +\bigl(1+\|V\|_{L^\infty_T(B^s_{2,1})}\bigr)^2\biggl(T\|V^1\|_{L^\infty_T(B^{s+1}_{2,1})}^2
     +\|V^2\|_{L^2_T(B^{s+1}_{2,1})}^2\biggr)\biggr)\nonumber\\
 &\leq  T+ C_\cO \bigl(\sqrt\eta+ R^4T+R^3\eta^2)\bigr)\cdotp     
\end{align}
All the terms of $\Theta_S$ may be bounded by taking advantage of Propositions \ref{propo_produc_BH} and 
\ref{propo_compo_BH}, hypotheses  $(\cH_1)-(\cH_5)$  and Inequality \eqref{eq:V2L2}. We get
$$
 \LpNBH{s}{2}{1}{
(\overline{S}^0_{22}-S^0_{22}(U))\pt V^2_L}{1} \le C_\cO\LpNB{s}{2}{1}{V}{\infty}\LpNB{s}{2}{1}{\pt V^2_L}{1}\leq CR\eta^2,
$$
\begin{align*}
\LpNB{s}{2}{1}{ \pal(Z^{\alpha \beta}(U)- \overline{Z^{\alpha \beta}}\pbe V^2_L)}{1}& \leq C_\cO\bigl(\LpNB{s}{2}{1}{V}{\infty}\LpNB{s+2}{2}{1}{V^2_L}{1}\\
&\quad+\bigl(\sqrt{T}\LpNB{s+1}{2}{1}{V^1}{\infty}\!+\!\LpNB{s+1}{2}{1}{V^2}{2}\bigr)\LpNB{s+1}{2}{1}{V^2_L}{2}\bigr)\\
&\leq C_\cO\bigl(R\eta^2+(R\sqrt T+\sqrt R\,\eta)\sqrt R\,\eta\bigr), 
\end{align*} 
\begin{align*}
    \LpNB{s}{2}{1}{S^\alpha_{21}(U)\pal V^1}{1}&\le C_\cO T(1+\LpNB{s}{2}{1}{V}{\infty}) \LpNB{s+1}{2}{1}{V^1}{\infty}\le C_\cO TR^2,\\
    \LpNB{s}{2}{1}{S^\alpha_{22}(U)\pal V^2}{1}& \le C_\cO \sqrt{T}(1+\LpNB{s}{2}{1}{V}{\infty}) \LpNB{s+1}{2}{1}{V^2}{2}\\
    & \le C_\cO \sqrt{T} R^{3/2}\eta,
\end{align*}
Likewise, we have
\begin{align*}
     \LpNB{s}{2}{1}{f^{21}(U)}{1} & \le C_\cO T\LpNB{s}{2}{1}{V}{\infty}\le   C_\cO T R\\
       \LpNB{s}{2}{1}{f^{22}(U,\nabla U^1)}{1} & \le C_{\cO,R}  T\LpNB{s}{2}{1}{(V, \n V^1)}{\infty}\le C_{\cO,R} R T,
    \end{align*}
    and, since $f^{23}$ is almost quadratic with respect to $\nabla U^2,$    
       \begin{align*}
        \LpNB{s}{2}{1}{f^{23}(U,\nabla U^2)}{1} &\le C_\cO  ( 1+\LpNB{s}{2}{1}{V}\infty)
        \bigl(\sqrt T+ \LpNB{s+1}{2}{1}{V^2}{2}\bigr)\LpNB{s+1}{2}{1}{V^2}{2}\\
    & \le  C_\cO R (\sqrt T + \sqrt R \eta)\sqrt R \eta.
\end{align*}
In the end, using \eqref{eq:Phi2}, assuming that $T+C_\cO(\sqrt \eta + R^4T+R^3\eta^2)$ is small enough
 and plugging all the above inequalities in \eqref{eq:tVS},  we conclude that 
\begin{equation*}
  \lVert \tV_S \rVert_{\widetilde{L}^\infty_T(\B{s}{2}{1})}+ \lVert \tV_S \rVert_{L^1_T(\B{s+2}{2}{1})} \le 
  C_\cO\bigl(R^2\eta^2  + TR^2 +  C_RTR\bigr)\cdotp
  \end{equation*}
Hence, if  $R^2\eta^2 +C_RTR +TR^2$ is small enough with respect to $\eta,$ then one can ensure 
that  $(\cH_3)$ is satisfied by $\tV_S,$ and thus also $(\cH_1),$ 
due to $\|V_L^2\|_{L^\infty_T(B^s_{2,1})}\leq C_0\|\wt V_0^2\|_{B^s_{2,1}}$ and \eqref{eq:wtdata}. 
To complete the proof of $\tV\in E^s_{T,R,\eta},$ it is only a matter of establishing that
\begin{equation}\label{eq:tVSdt}
\LpNB{s}{2}{1}{\partial_t\tV_S}{1}\leq\sqrt \eta.\end{equation}
To do this, we use \eqref{eq:dttV2} for $\tV_S$ which, in our context implies that
\begin{multline*}\|\pt\wt V_S\|_{L^1_T(B^s_{2,1})}\leq C\bigl(1+\|V\|_{L^\infty_T(B^{s}_{2,1})}\bigr)
\Bigl(\|\wt V_S\|_{L_T^1(B^{s+2}_{2,1})}
\\+\|\nabla\wt V_S\|_{L^2_T(B^{s}_{2,1})}\bigl(\sqrt T\|V^1\|_{L^\infty_T(B^{s+1}_{2,1})}+\|V^2\|_{L_T^2(B^{s+1}_{2,1})}\bigr) 
+\|\Theta_S\|_{L_T^1(B^s_{2,1})} \Bigr)\cdotp\end{multline*}
Taking advantage of \eqref{eq:tVSdt} and on the fact that we have just proved 
that $\|\Theta_S\|_{L_T^1(B^s_{2,1})}$ is $\cO(\eta),$  the above inequality implies that 
$$\|\pt\wt V_S\|_{L^1_T(B^s_{2,1})}\leq C_{\cO} R\eta \bigl(1 + R\sqrt T+\sqrt{R\eta}\bigr)\cdotp$$
If $\eta$ and $T$ has been chosen  sufficiently small, we thus have   \eqref{eq:tVSdt}, which completes the proof of 
$\tV\in E^s_{T,R,\eta}.$


\subsection{The proof of the local existence}

One can construct a sequence of approximate solutions by solving iteratively linear Systems of type \eqref{lin:eq_b}: we define the first term of the sequence to be
$V_0\defn (0,V_L^2)$ then, once $V_p$ is known, we set $U_p\defn \bar U+V_p$
and define $V_{p+1}$ to be the solution of 
\begin{equation} \begin{aligned}
  \label{def:V_p+1}
              &S^0_{11}(U_p)\pt V^1_{p+1} +\suma S^\alpha_{11}(U_p) \pal V^1_{p+1} =\Theta^1(U_p),  \\
	    &S^0_{22}(U_p)\pt V^2_{p+1} - \sumab\pal(Z^{\alpha \beta}(U_p)\pbe V^2_{p+1})=\Theta^2(U_p)     
	     \end{aligned}\end{equation}
where the right-hand sides are given by \eqref{def:Theta1:2:invari}, 
supplemented with  the initial data
\begin{align}    \label{def:data:Vp+1}
    V_{p+1|t=0}\defn S_{p+1} V_0, \esp{ where} S_p \hbox{ is the cut-off operator defined in \eqref{def:dot:S:j} }.
\end{align}
Since the initial data belong to all the Sobolev spaces (owing to the spectral cut-off), the classical theory 
for linear hyperbolic or parabolic systems in Sobolev spaces guarantees that, at each step, 
the above system has a global solution that belongs to all Sobolev spaces 
(see  \cite{GavSerre07,HajDanChe11} for the hyperbolic part of the system, and 
 \cite{Angel23,Kawashima83,Serr10} for the parabolic part). 
 Furthermore, in light of the previous subsection, since  for all $p\in\N,$
 $$\|S_p V_0^1\|_{B^{s+1}_{2,1}}\leq \| V_0^1\|_{B^{s+1}_{2,1}}\andf \|S_p V_0^2\|_{B^{s}_{2,1}}\leq \| V_0^2\|_{B^{s}_{2,1}},$$
 if taking $R\geq 2C_0( \| V_0^1\|_{B^{s+1}_{2,1}}+ \| V_0^2\|_{B^{s}_{2,1}}),$
 then one can find  positive real numbers $\eta,\,R$ so that for all $p\in\N,$
 $V_p\in E^s_{T,R,\eta}$ implies $V_{p+1}\in  E^s_{T,R,\eta}.$
 Hence, all terms of $(V_p)_{p\in\N}$ belong to $E^s_{T,R,\eta}.$ 
 
 In order to prove the convergence of this sequence,  it will be shown that it is a Cauchy sequence 
  in  the space 
  \begin{equation}\label{eq:FT}
  F_{T}^s\defn\Bigl\{V\defn(V^1,V^2):V\in {\cC}([0,T];\B{s}{2}{1}\times \B{s-1}{2}{1}),\; V^2\in L^1_T({\B{s+1}{2}{1}}) \Bigr\}\cdotp\end{equation}
  The reason for lowering regularity  is the usual loss of one derivative when proving stability estimates for quasilinear hyperbolic  systems. Here it is  harmless except  for $d=1$ and $s=1/2$  (see  the end of this section). 
  
  To simplify the presentation, we only consider the case where the  lower order
  terms $f^1$ and $f^2$ are identically zero. 
  Now, put $\dV_p \defn V_{p+1} -V_p$ and take the difference between the equation \eqref{def:V_p+1} for the $(p\! +\!1)$-th step and the
$p$-th step. We  get
\begin{equation}\begin{aligned}\label{eq:tV_p}
         S^0_{11}(U_p)\pt \dV^1_{p}& + \suma S^\alpha_{11}(U_p) \pal \dV^1_{p} =h_p,\\
           S^0_{22}(U_p)\pt \dV^2_{p}& - \sumab\pal(Z^{\alpha \beta}(U_p)\pbe \dV^2_{p})=g_p,
\end{aligned}\end{equation}
with $h_p=h_p^1+h_p^2$ , $g_p=g_p^1+g_p^2+g_p^3+g_p^4+ g_p^5+g_p^6$  and
\begin{align*}
    h_p^1&\defn -S^0_{11}(U_p)\suma \left( \wt S^\alpha_{11}(U_p)- \wt S^\alpha_{11}(U_{p-1})\right)\pal V^1_{p-1},\\
     h_p^2 &\defn - S^0_{11}(U_p)\suma \left(\wt S^\alpha_{12}(U_{p}) -  \wt S^\alpha_{12}(U_{p-1})\right) \pal V^2_{p-1}
     -\suma S^\alpha_{12}(U_{p})\pal \dV_{p-1}^2,
      \end{align*}
\begin{align*}
    g_p^1&\defn (S^0_{22}(U_{p-1})- S^0_{22}(U_{p}))\pt V_p^2,\ 
    &&g_p^2\defn \suma \left( S^\alpha_{22}(U_{p-1})- S^\alpha_{22}(U_{p})\right)\pal V^2_{p-1},\\
     g_p^3&\defn \sumab\!\pal\!\left(\bigl(Z^{\alpha \beta}(U_{p})-Z^{\alpha \beta}(U_{p-1})\bigr)\pbe V^2_{p}\right),\ 
     &&g_p^4\defn \suma\left(S^\alpha_{21}(U_{p-1}) - S^\alpha_{21}(U_{p})\right) \pal V^1_{p-1}, \\
    g_p^5&\defn  - \suma S^\alpha_{21}(U_{p}) \pal \dV^1_{p-1},\ 
    &&g_p^6\defn  - \suma S^\alpha_{22}(U_{p}) \pal \dV^2_{p-1}\cdotp
\end{align*}
 All the estimates established in Subsection \ref{ss:coupled} are valid for $V_p.$ In particular, 
defining $\Psi_1$ and $\Psi_2$ according to \eqref{eq:tV1} and \eqref{eq:Psi2}, 
we have 
$$\int_0^T\Psi_1\leq \log2\andf \int_0^T\Psi_2\leq \log2.$$
 Hence, applying  \eqref{est:tV1:lin} and  \eqref{est:tV2:lin}  with exponents $s$ and $s-1,$ respectively, 
 to the two equations of  \eqref{eq:tV_p}, we  get\footnote{Here we need $s-1$  to be 
larger than $-d/2,$ whence the restriction on the regularity exponent if $d=1.$}
\begin{align}
         \label{est:tVp1:1}
       &\LpNB{s}{2}{1}{\dV^1_p}{\infty} \le 2C_0 \biggl(\NB{s}{2}{1}{\dV^1_{0,p}}+\int^{T}_0 \NB{s}{2}{1}{h_p}\biggr),
        \\          \label{est:tVp2:1}
         & \LpNB{s-1}{2}{1}{\dV^2_p}{\infty}+\frac{c}{2}  \LpNB{s+1}{2}{1}{\dV^2_p}{1} \le 2C_0 \biggl(\NB{s-1}{2}{1}{\dV^2_{0,p}}+\int^{T}_0 \NB{s-1}{2}{1}{g_p}\biggr)\cdotp\end{align}
    In order to estimate the terms on the right-hand side of \eqref{est:tVp1:1} and \eqref{est:tVp2:1},
    we shall remember all the time that the terms $V_p$ are in the set $E^s_{T,R,\eta}.$ 
     Now, leveraging the product and composition laws recalled in Appendix, we  get
  \begin{multline*}
        \LpNB{s}{2}{1}{h_p^1}{1}\le C(1+\LpNB{s}{2}{1}{V_p}{\infty}) \Bigl( T \LpNB{s}{2}{1}{\dV_{p-1}^1}{\infty}\\
      +\sqrt{T}\LpNB{s}{2}{1}{\dV_{p-1}^2}{2}\Bigr)\LpNB{s+1}{2}{1}{V_p^1}{\infty},
\end{multline*}
      \begin{multline*}
         \LpNB{s}{2}{1}{h_p^2}{1}\le C(1+\TLpNB{s}{2}{1}{V_p}{\infty}) \left( \sqrt{T}\LpNB{s+1}{2}{1}{\dV^1_{p-1}}{\infty}\LpNB{s+1}{2}{1}{V_p^2}{2}\right.\\
        \left.+ \LpNB{s}{2}{1}{\dV_{p-1}^2}{2}  \LpNB{s+1}{2}{1}{V_p^2}{2}+ \LpNB{s+1}{2}{1}{\dV_{p-1}^2}{1}\right)\cdotp
\end{multline*}
Next, Inequality \eqref{product_propo4} combined with suitable embedding 
and Proposition \ref{propo_compo_BH} gives
\begin{align*}
    \LpNB{s-1}{2}{1}{g_p^1}{1}&\le C\Bigl(\|S^0_{22}(U_{p-1})-S^0_{22}(U_p)\|_{L^\infty_T(B^{-1}_{\infty,\infty})}\LpNB{s}{2}{1}{\pt V_p^2}{1}\\&\hspace{4cm}
 +\|\pt V_p^2\|_{L_T^1(L^\infty)}\|S^0_{22}(U_{p-1})-S^0_{22}(U_p)\|_{L^\infty_T(B^{s-1}_{2,1})}\Bigr)\\
&\leq C\Bigl(1+\|(V_{p-1},V_p)\|_{L^\infty_T(B^s_{2,1})}\Bigr)
\LpNB{s-1}{2}{1}{\dV_{p-1}}{\infty} \LpNB{s}{2}{1}{\pt V_p^2}{1}.\end{align*}
We have  by Propositions  \ref{propo_produc_BH}  and  \ref{propo_compo_BH}, 
\begin{align*}
        &\LpNB{s-1}{2}{1}{g_p^3}{1} \le C  \bigl(\sqrt{T}\LpNB{s}{2}{1}{\dV^1_{p-1}}{\infty}+ \LpNB{s}{2}{1}{\dV_{p-1}^2}{2}  \bigr)\LpNB{s+1}{2}{1}{V_p^2}{2},\\
      &\LpNB{s-1}{2}{1}{(g_p^2,g_p^4)}{1}\le C \LpNB{s-1}{2}{1}{\dV_{p-1}}{\infty}
      \bigl(T\LpNB{s+1}{2}{1}{ V_p^1}{\infty}+\sqrt{T}\LpNB{s+1}{2}{1}{ V_p^2}{2} \bigr),\\
      &\Vert g_p^5\Vert_{L^1_{T}(\B{s-1}{2}{1})}\le  C \left( T+T\LpNB{s}{2}{1}{V_p}{\infty}\right)\LpNB{s}{2}{1}{\dV^1_{p-1}}{\infty}, \\
     & \Vert g_p^6\Vert_{L^1_{T}(\B{s-1}{2}{1})}\le  C \left( \sqrt{T}+\sqrt{T}\LpNB{s}{2}{1}{V_p}{\infty}\right)\LpNB{s}{2}{1}{\dV^2_{p-1}}{2}. \end{align*}
     Plugging the above inequalities in \eqref{est:tVp1:1} and \eqref{est:tVp2:1}, 
     taking advantage of the  boundedness of $(V_p)$ in $E^s_{T,R,\eta},$ 
  and    using \eqref{eq:V2L2},  we obtain that 
\begin{multline*}
    \LpNB{s}{2}{1}{\dV^1_p}{\infty}\le 
    C\Bigl(\NB{s}{2}{1}{\dV^1_{0,p}}+ \bigl(R^2T+R^{3/2}T^{1/2}\eta\bigr)\LpNB{s}{2}{1}{\dV^1_{p-1}}{\infty}
    \\+\bigl(R^2T^{1/2}+R^{3/2}\eta\bigr)\LpNB{s}{2}{1}{\dV^2_{p-1}}{2}+R\LpNB{s+1}{2}{1}{\dV^2_{p-1}}{1}\Bigr),
     \end{multline*}
     \begin{multline*}
   \LpNB{s-1}{2}{1}{\dV^2_p}{\infty}+\LpNB{s+1}{2}{1}{\dV^2_p}{1}\le C\Bigl(\NB{s-1}{2}{1}{\dV^2_{0,p}}+\bigl(\sqrt{\eta}R+\eta\sqrt{RT} +TR\bigr) \LpNB{s}{2}{1}{\dV^1_{p-1}}{\infty}\\
 +\bigl(R\sqrt T +\eta\sqrt R\bigr)\|\dV^2_{p-1}\|_{L^2_T(B^{s}_{2,1})}
 +\bigl(R\sqrt \eta +TR\bigr)\|\dV^2_{p-1}\|_{L^\infty_T(B^{s-1}_{2,1})}\Bigr)\cdotp
 \end{multline*}
 Let us set for some small enough parameter $\varepsilon$: 
 $$
 X_p(T)\defn \varepsilon R^{-1}    \LpNB{s}{2}{1}{\dV^1_p}{\infty}+   \LpNB{s-1}{2}{1}{\dV^2_p}{\infty}+\LpNB{s+1}{2}{1}{\dV^2_p}{1}.
 $$
 as well as  the following interpolation inequality: 
\begin{align*}
 \LpNB{s}{2}{1}{\dV_{p-1}^2}{2}\le \sqrt{ \LpNB{s-1}{2}{1}{\dV_{p-1}^2}{\infty} \LpNB{s+1}{2}{1}{\dV_{p-1}^2}{1} }\le X_{p-1}(T)
 \end{align*}

 From the above two inequalities, we deduce after some simplification that for all $p\in\N,$
 \begin{multline*}
 X_p(T) \leq C\Bigl( \varepsilon R^{-1}\NB{s}{2}{1}{\dV^1_{0,p}}+\NB{s-1}{2}{1}{\dV^2_{0,p}}\\
+ \Bigl(\varepsilon^{-1}(R^2\sqrt{\eta} + \eta R^{3/2}T^{1/2}+R^2T)+\varepsilon +\sqrt T R+R\sqrt\eta\Bigr)
X_{p-1}(T) \Bigr)\cdotp\end{multline*}
Therefore, choosing $\varepsilon =1/(4C)$ then reducing $T$ and $\eta$ if needed, we end up with 
\begin{equation}\label{eq:Xp}
X_p(T)\leq  C\Bigl( \varepsilon R^{-1}\NB{s}{2}{1}{\dV^1_{0,p}}+\NB{s-1}{2}{1}{\dV^2_{0,p}}\Bigr) +\frac12 X_{p-1}(T).
\end{equation}
We observe that 
$$\NB{s}{2}{1}{\dV^1_{0,p}}\simeq  2^{-p} \bigl(2^{p(s+1)}\|\Delta_p V^1_0\|_{L^2}\bigr)\andf
\NB{s-1}{2}{1}{\dV^2_{0,p}}\simeq 2^{-p} \bigl(2^{ps}\|\Delta_p V^2_0\|_{L^2}\bigr)\cdotp $$
Hence, summing up Inequality \eqref{eq:Xp} from $p=1$ to $p=\infty,$ we discover that
$\sum X_p(T)$ is a convergent series. 
Therefore  $(V_{p})_{p\in\N}$ is a Cauchy sequence in $F_{T}^s$.  Hence, there
exists a function $V \in F_{T}^s$ such that $V_p \longrightarrow V$ strongly  in $F_{T}^s$
 as $p \longrightarrow \infty$. Now,  from a functional analysis argument totally similar to that of  \cite[Chap. 10]{HajDanChe11}, one can prove that 
 $V$ satisfies  $(\mathcal{H}_1)- (\mathcal{H}_5$).
 Furthermore  the  strong convergence in $F_{T}^s$ combined with the uniform bounds in $E^s_{T,R,\eta}$ 
enable us to pass to the limit in the approximate system, and to conclude that $V$ is a
solution to \eqref{Eq_b} supplemented with initial data $V_0.$ 
There only remains to check the time continuity of $V$ with values in $B^{s+1}_{2,1}\times B^s_{2,1},$
and the fact that $V^2$ belongs to $L^1(0,T;B^{s+2}_{2,1}).$ These two properties may be proved by following 
the method of \cite[Chap. 10]{HajDanChe11}. 
\medbreak
As for  the proof of uniqueness, we set $\dV \defn V_2-V_1$, where $V_1$ and $V_2$ are two
solutions to the system $\eqref{Eq_b}$  subject to the same initial data.
Then the error solution $\dV$ satisfies the equation \eqref{eq:tV_p} where instead of $V_p,V_{p-1},\dV_p,\dV_{p-1}$ we have $ V, \dV $ respectively. In the same way,  one can find some $T_1\leq T$ such that  \eqref{eq:Xp} is satisfied with 
 $ \dV_p=\dV_{p-1}=\dV$ and  no first term in the right-hand side.  
 This easily implies that   $\dV=0$ in $F_{T_1}^s,$ whence uniqueness  on $[0,T_1]\times\Rd.$ Using a continuity argument, one can then get  uniqueness on the whole interval of existence.


\subsection{A continuation criterion}

This section is devoted to proving the last part of   Theorem \ref{Thm:loc:scri}.
Let us explain how to proceed in the general case of Conditions {\bf B},  leaving at the end the particular cases mentioned in 
Remark \ref{rmq:improved}. 

Let  $V$ be a solution of \eqref{Eq_b} on $[0,T^*[\times\R^d$ that belongs to $E_T^s$ for all $T<T^*.$ 
We shall prove that $\|V\|_{E_{T^*}}^s$ is finite, which combined with classical arguments will entail
that the solution may be continued beyond $T^*.$ The starting point is Inequality 
\eqref{est:V1:Li:L2} with $\Theta^1=\Theta^1(U)$ defined in \eqref{def:Theta1:2:invari}. 
Compared to the proof of Proposition \ref{energ:est:V1:lem}, the difference is that the commutators 
$R_j^{11}$ and the terms in $\Theta^1$ are going to be bounded according to \eqref{comm:est:a:b:sig>0} 
and \eqref{product_propo1}, respectively. 
In this way, for some constant depending only on $s$ and on the range of $[0,T^*[\times\R^d$ by $V$ (which is bounded), 
we get, denoting $\wt S_{11}^\alpha\defn (S_{11}^0)^{-1}S_{11}^\alpha,$ 
$$\begin{aligned}
\|R_j^{11}\|_{L^2}&\leq C c_j2^{-j(s+1)} \sum_\alpha\Bigl(\|\nabla(\wt S_{11}^\alpha(U))\|_{L^\infty}\|\nabla V^1\|_{B^s_{2,1}}
+\|\nabla V^1\|_{L^\infty}\|\nabla(\wt S_{11}^\alpha(U))\|_{B^s_{2,1}}\Bigl)\\
&\leq  C c_j2^{-j(s+1)}\bigl(\|\nabla V\|_{L^\infty}\|\nabla V\|_{B^s_{2,1}}
+\|\nabla V^1\|_{L^\infty}\|V^1\|_{B^{s+1}_{2,1}}\bigr)\cdotp\end{aligned}$$
As for $\Theta^1,$ using the decomposition 
$$\Theta^1=f^1(U)-\sum_\alpha \overline{S_{12}^\alpha}\partial_\alpha V^2-
\sum_\alpha \bigl(S_{12}^\alpha(U)-\overline{S_{12}^\alpha}\bigr)\partial_\alpha V^2$$
and remembering that $f^1(\bar U)=0,$  we get
$$\|\Theta^1\|_{B^{s+1}_{2,1}}\leq C\Bigl(\|V\|_{B^{s+1}_{2,1}}+\|\nabla V^2\|_{B^{s+1}_{2,1}}
+\|\nabla V^2\|_{L^\infty}\|V\|_{B^{s+1}_{2,1}}\Bigr)\cdotp$$
Hence, plugging these two inequalities in \eqref{est:V1:Li:L2}, multiplying by $2^{j(s+1)}$
and summing up on $j\ge -1,$ we get for all $t\in[0,T^*[,$
\begin{multline}\label{eq:V1}\|V^1(t)\|_{B^{s+1}_{2,1}} \leq C_0\|V^1_0\|_{B^{s+1}_{2,1}} \\+
C\int_0^t\Bigl(\|\DIV(S_{11}(U))\|_{L^\infty}\|V^1\|_{B^{s+1}_{2,1}}+
(1+\|\nabla V\|_{L^\infty})\|V\|_{B^{s+1}_{2,1}}\Bigr)+C\int_0^t\|V^2\|_{B^{s+2}_{2,1}}.\end{multline}
In order to bound $V^2,$ we start from Inequality \eqref{esti_v_2-2'} and use
the commutator estimate \eqref{comm:est:a:b:sig>0b} and product estimate  \eqref{product_propo1}.
Denoting $\wt Z^{\alpha\beta}\defn (S^0_{22})^{-1}Z^{\alpha\beta},$ we get
  $$\begin{aligned}
\|R_j^{2}\|_{L^2}&\leq C\sum_{\alpha,\beta}\bigl(\|[\Dj,\wt Z^{\alpha\beta}(U)]\partial_\alpha\partial_\beta V^2\|_{L^2}
+\|\Dj ((S_{22}^0)^{-1}\partial_\alpha(Z^{\alpha\beta}(U))\partial_\beta V^2)\|_{L^2}\bigr)\\
&\leq Cc_j2^{-js}\sum_{\alpha,\beta} \Bigl(\|\nabla(\wt Z^{\alpha\beta}(U))\|_{L^\infty}\|\nabla V^2\|_{B^s_{2,1}}
+\|\nabla V^2\|_{L^\infty}\|\nabla(\wt Z^{\alpha\beta}(U))\|_{B^s_{2,1}}\\
&\hspace{3cm}+\|\partial_\alpha(Z^{\alpha\beta}(U))\|_{L^\infty}\|\nabla V^2\|_{B^s_{2,1}}
+\|\nabla V^2\|_{L^\infty}\|\nabla (Z^{\alpha\beta}(U))\|_{B^s_{2,1}}\Bigr)\\
&\leq Cc_j2^{-js}\Bigl(\|\nabla V\|_{L^\infty}\|\nabla V^2\|_{B^s_{2,1}}+\|\nabla V^2\|_{L^\infty}\|\nabla V\|_{B^s_{2,1}}\Bigr)\cdotp
\end{aligned}$$
Next, recall that
$$
\Theta^2=f^{21}(U)+f^{22}(U,\nabla U^1)+f^{23}(U,\nabla U^2)-
\sum_\alpha\bigl(S^\alpha_{21}(U)\partial_\alpha V^1+S^\alpha_{22}(U)\partial_\alpha V^2\bigr)\cdotp
$$
Let us denote $M\defn \|\nabla V^1\|_{L^\infty([0,T^*)\times\R^d)}$ (a finite real number by assumption).
Leveraging the usual product and composition laws, and our specific assumptions on $f^2,$ we easily get:
\begin{multline*}
\|\Theta^2\|_{B^s_{2,1}}\leq C_M\bigl(\|V^1\|_{B^{s+1}_{2,1}}+\|V^2\|_{B^s_{2,1}}\bigr)+C\|\nabla V^2\|_{L^\infty}^2\|V\|_{B^s_{2,1}}\\+ 
C\Bigl( \|V\|_{B^s_{2,1}}
+\bigl(1+\|\nabla V^2\|_{L^\infty}\bigr)\|\nabla V^2\|_{B^s_{2,1}}
+\|\nabla V\|_{L^\infty}\|V\|_{B^s_{2,1}}+\|V\|_{B^{s+1}_{2,1}}\Bigr)\cdotp\end{multline*}
Inserting the above inequalities in \eqref{esti_v_2-2'}, we conclude that for all $t\in[0,T^*[,$ 
\begin{multline}\label{eq:V2}\|V^2(t)\|_{B^{s}_{2,1}} + \int_0^t\|V^2\|_{B^{s+2}_{2,1}} \leq C_0\|V^2_0\|_{B^{s}_{2,1}} 
+C_M\int_0^t\|(V^1,V^2)\|_{B^{s+1}_{2,1}\times B^s_{2,1}}\\
+C\int_0^t\bigl(\|\partial_tS^0_{22}(U)\|_{L^\infty}\|V^2\|_{B^s_{2,1}}+(1+\|\nabla V\|_{L^\infty}^2)\|V\|_{B^s_{2,1}}\bigr)\\
+C\int_0^t\Bigl((1+\|\nabla V\|_{L^\infty})\|\nabla V^2\|_{B^s_{2,1}}
+\|\nabla V^2\|_{L^\infty}\|V\|_{B^{s+1}_{2,1}}
+\|V\|_{B^{s+1}_{2,1}}\Bigr)\cdotp\end{multline}
Combining Inequalities \eqref{eq:V1} and \eqref{eq:V2}   and using repeatedly the fact that for all $\varepsilon>0,$
$$\|V^2\|_{B^{s+1}_{2,1}}\leq \varepsilon \|V^2\|_{B^{s+2}_{2,1}}   +C\varepsilon^{-1}\|V^2\|_{B^{s}_{2,1}}, $$
yields  for all $t\in[0,T^*[,$ 
\begin{multline}\label{eq:boundfst}
\|V\|_{F^s_t}\leq C_0\|V_0\|_{B^{s+1}_{2,1}\cap B^{s}_{2,1}} 
+C_M\int_0^t\|V\|_{B^{s+1}_{2,1}\times B^s_{2,1}}\\
+C\int_0^t \bigl(\|\DIV(S_{11}(U))\|_{L^\infty}+\|\partial_tS^0_{22}(U)\|_{L^\infty}
+1+\|\nabla V\|_{L^\infty}^2\bigr)\|V\|_{B^{s+1}_{2,1}\times B^s_{2,1}}.\end{multline}
Using Gronwall inequality and the assumptions of the last part of Theorem \ref{Thm:loc:scri}, one can conclude 
that $\|V\|_{E^s_{T^*}}<\infty.$

Note that if $f^{22}$ depends at most quadratically of $\nabla U^1,$ then we have:
$$\|f^{22}(U,\nabla U^1)\|_{B^s_{2,1}} \!\leq \!C\Bigl(\|\nabla V^1\|_{B^s_{2,1}}
+\|\nabla V^1\|_{L^\infty}\|\nabla V^1\|_{B^s_{2,1}}\!
+\bigl(\|\nabla V^1\|_{L^\infty}+\|\nabla V^1\|_{L^\infty}^2\bigr) \|V\|_{B^s_{2,1}}\Bigr)\cdotp$$
Hence the second term in the right-hand side of \eqref{eq:boundfst} is not needed.

Further note that if $f^{23}$ is affine in $\nabla U^2,$   and $S_{22}^0,$ $Z^{\alpha\beta}$ and $\wt S_{11}^\alpha
\defn(S_{11}^0)^{-1}S^\alpha_{12}$  only depend on $U^1,$ 
 then there are no terms  
$\|\nabla V^2\|_{L^\infty}\|V^2\|_{B^{s+1}_{2,1}}$ or 
 $\|\nabla V^2\|_{L^\infty}^2\|V^2\|_{B^{s}_{2,1}}$   in the right-hand sides of \eqref{eq:V1} 
and \eqref{eq:V2}, which allows to replace $\|\nabla V\|_{L^\infty}^2$ in \eqref{eq:boundfst}
by just  $\|\nabla V^1\|_{L^\infty}^2.$ In fact, in order to prove the counterpart of \eqref{eq:V1}, one has to start with the identity
$$\partial_t V^1+\sum_\alpha \wt  S_{11}^\alpha \partial_\alpha V^1=-\sum_\alpha \wt S^\alpha_{12}(U^1)\partial_\alpha V^2.$$

Let us finally explain why $\|V\|_{E^s_{T^*}}<\infty$ allows to continue the solution beyond $T^*.$ 
On the one hand, the fact that $\partial_tV\in L^1(0,T^*;B^s_{2,1})$ and $V\in \cC_b([0,T^*);B^s_{2,1})$ 
ensures that $V$ has a limit $V_{T^*}$ in $B^s_{2,1}$ when $t\to (T^*)^-.$ 
On the other hand the boundedness of $V^1$ in $B^{s+1}_{2,1}$ ensures that
this limit actually belongs to $B^{s+1}_{2,1}\times B^s_{2,1}.$
Now, solving the Cauchy problem for \eqref{Eq_b} with initial data $V_{T^*}$  gives a continuation of the solution 
in the desired space.

  \subsection{The endpoint $d=1$ and $s=1/2$}
  
  The only parts that failed are  the convergence of $(V_p)_{p\in\N}$ and uniqueness in the space $F_T^{\frac12}$
  defined in \eqref{eq:FT} since the product maps $B^{\frac12}_{2,1}\times B^{-\frac12}_{2,1}$ 
  in   $B^{-\frac12}_{2,\infty}$ rather than in  $B^{-\frac12}_{2,1}.$
 A way to overcome the difficulty is to first smooth out the data  and to produce a sequence $(V_p)_{p\in\N}$
 of true solutions of \eqref{Eq_b} corresponding to these smooth data, that will be 
 bounded in $E^{\frac12}_{T,R,\eta}$ for some suitable $T$ and $\eta$ \emph{independent of $p$}
 (at this point we have to take advantage of the continuation criterion), 
 then to use compactness arguments to pass to the limit up to subsequence. 
 The reader is referred to the forthcoming section  where a similar method is used
 to  prove existence in the critical regularity framework.
 
 As for the uniqueness, it can  be treated by following the ideas of Subsection \ref{sec:uni:crit}, based
 on a suitable logarithmic interpolation inequality, and Osgood lemma.


\section{Local existence in critical spaces}
\label{sec:proof:loc:cri}

Here we want to establish the local well-posedness  for System \eqref{Eq_b} under 
Assumption \textbf{C} in what we  called \emph{a critical functional framework}
 by analogy with  the compressible Navier-Stokes system studied in \cite{Dan01loc}.
More precisely, we consider initial data  $(U^1_0,U^2_0)=(\bar U^1,\bar U^2)+ (V^1_0,V^2_0)$ 
such that   $V^1_0\in\dot B^{\frac d2}_{2,1}$ and  $V^2_0\in\dot B^{\frac d2-1}_{2,1}.$

 \subsection{A priori estimates}
As a first, we aim at establishing a priori estimates for System  \eqref{Eq_b}. For simplicity,  we omit the lower order source term $f$
satisfying Assumption \textbf{C}. 
Then, we consider  a  smooth solution  $V=(V^1,V^2)$ on $[0,T]\times \Rd$ 
 and fix some real numbers $M_1\geq1,$ $M_2>0$ and $M_3>0$ such that
\begin{align}    \label{def_M}
      \|V^1\|_{L_T^\infty(B^{\cd}_{2,1})}\leq M_1,\quad   
     \LpNBH{\cd+1}{2}{1}{V^2}{1}\leq M_2\andf \LpNBH{\cd-1}{2}{1}{\pt V^2}{1}\leq M_3.
    \end{align}
We assume   that there exists a bounded open 
   subset  $\cO^1$ of $\mathcal{U}^1$ satisfying $ \overline{\cO}^1\subset \mathcal{U}^1$ and 
 \begin{align}    \label{def_U_O_1}
       U^1(t,x) \in \cO^1 \esp{for all} t\in [0,T]\andf x\in \Rd. \end{align}
Hence, there exists a constant $C$ such that 
\begin{align}    \label{equi_S02}
      C^{-1}I_{n_1}\le S^{0}_{22} (U^1(t,x))\le CI_{n_1} \esp{for all} t\in [0,T]\andf x\in \Rd.    \end{align}
   
   \paragraph{\bf Step $1$: Estimates for $V^1$.}    The first step is to prove:
\begin{proposition}\label{energy_V^1}
  Let Assumption \textbf{C} be in force and $d\ge 2$. 
    Then, there exists a constant $C=C(\cO^1)$ such that for all $m\in \mathbb{Z}$, the following inequality holds true.
    \begin{equation}
	  \label{Esti_V_1}
     \sum_{j\ge m}2^{j\frac d2}\Vert V^1_j\Vert_{L^\infty_T(L^2)}
       \le \sum_{j\ge m}2^{j\frac d2}\normede{V^1_{0,j}}+C\bigl(1+\|V^1\|_{L^\infty_T(\dot B^{\frac d2}_{2,1})}\bigr)M_2
	   \end{equation}
     with $ V_j^1 \defn \DDj V^1$ and   $ V_{0,j}^1 \defn \DDj V_0^1$
     (where $\ddj$ stands for   the homogeneous Littlewood-Paley  operator defined in Appendix).\end{proposition}
    \begin{proof}
  Let $A^\alpha_{11} \defn (S^0_{11})^{-1} S^\alpha_{11} $ and $A^\alpha_{12} \defn (S^0_{11})^{-1} S^\alpha_{12}.$
  According to Assumption \textbf{C},  System \eqref{Eq_b} may be written
 \begin{align*}
      \pt V^1  +\suma \left( A^\alpha_{11} (U^2) \pal V^1 + A^\alpha_{12}(U^1) \pal V^2\right)=0.
 \end{align*}
	  	    Applying $\ddj$  to the above  equation yields for all $j\in\Z,$ 
     \begin{align*}
         \pt V^1_j + \suma A^\alpha_{11}(U^2)\pal V^1_j= -\suma\DDj(A^\alpha_{12}(U^1) \pal V^2)+ R_j^{11}\text{ with } R_j^{11} \defn [A_{11}^\alpha(U^2),\DDj]\pal V^1. 
     \end{align*}
Now, arguing  as in  the proof of Proposition \ref{energ:est:V1:lem}, we arrive for all $t\in[0,T]$ at 
    \begin{multline}
               \label{cri:eq_V1_1}
          2^{j\frac d2}\Vert V^1_j\Vert_{L^\infty_t(L^2)}  \le  2^{j\frac d2}\normede{V^1_j(0)}+C  2^{j\frac d2}\suma\int^t_0\normeinf{\nabla A^\alpha_{11}(U^2)} \normede{V_j^1}\\
        + C  2^{j\frac d2}\suma\int^t_0\Bigl(\normede{ \DDj(A^\alpha_{12}(U^1) \pal V^2)} + \normede{ R_j^{11}}\Bigr)\cdotp     \end{multline}
         At this point,  two comments are in order.  First, to bound the right-hand side of \eqref{cri:eq_V1_1}, we need   $ \n A^\alpha_{11}(U)$ to be in  $L^1_T(L^\infty(\Rd)),$  which is not ensured by our critical functional framework, 
      unless  the matrices 
     $ A^\alpha_{11}$ are independent of $U^1$.  Moreover, as $ U^2$  is not necessarily  bounded in $[0,T]\times\Rd$, 
          $  A^\alpha_{11} $ has to be  affine  with respect to $U^2$.
          Second,  we need  $A^\alpha_{12}(U) $  to be bounded.  
         Again, since we do not have a control on $\|U^2\|_{L^\infty([0,T]\times\Rd)},$ we can only handle the case  where the matrices
          $A^\alpha_{12}(U)$ are independent of $U^2$.
          
          This being said, the terms	$     \Vert \DDj(A^\alpha_{12}(U^1) \pal V^2)\Vert_{L^2}$  may be bounded according to  the stability of the space $ \BH{\frac{d}{2}}{2}{1} $ by  product and to  Proposition  \ref{propo_compo_BH} as follows:
 \begin{align}	     \label{cri:eq_V1_2}
	   \Vert \DDj(A^\alpha_{12}(U^1) \pal V^2)\Vert_{L^2}\le Cc_j2^{-j\frac d2}
	   (1+\NBH{\cd}{2}{1}{V^1})  \NBH{\cd+1}{2}{1}{ V^2}.
    	    \end{align}
Next, thanks to  Proposition  \ref{propo_commutator-BH}, we have the following bound for $R_{j}^{11}$: 
    \begin{align*}
    \Vert  R_{j}^{11} \Vert_{L^2} \le C c_j2^{-j\frac{d}{2}} \Vert \nabla A^\alpha_{11}(U^2) \Vert_{\BH{\cd}{2}{1}} 
    \Vert V^1\Vert_{\BH{\cd}{2}{1}}.\end{align*}
Using the fact that $  A^\alpha_{11}(U^2)$ is at most linear, we finally get
\begin{align}	     \label{cri:est:Rj}
\Vert R_{j}^{11} \Vert_{L^2} \le C c_j2^{-j\frac{d}{2}}  \NBH{\cd+1}{2}{1}{V^2} \NBH{\cd}{2}{1}{V^1}.  \end{align}
Plugging \eqref{cri:est:Rj} and \eqref{cri:eq_V1_2} into \eqref{cri:eq_V1_1}, using $ \BH{\frac{d}{2}}{2}{1} \hookrightarrow  L^\infty$ and summing on $j\ge m$ yields \eqref{Esti_V_1}.
	\end{proof}

\paragraph{\bf Step $2$: Estimates for $V^2$ and $\pt V$.} For any integer $m,$ 
  the second equation of \eqref{Eq_b} (under conditions \textbf{C}) may be rewritten
       \begin{multline}       \label{equ_V_2_modif}
         S^0_{22}(U_m^1)\pt V^2 -\sumab \pal(Z^{\alpha \beta}(U^1_m)\pbe V^2)=-\suma\left(S^\alpha_{21}(U) \pal V^1 + S^\alpha_{22}(U) \pal V^2\right)\\
         +(S^0_{22}(U_m^1)-S^0_{22}(U^1))\pt V^2+ \sumab \pal((Z^{\alpha \beta}(U^1)-Z^{\alpha \beta}(U^1_m))\pbe V^2),
     \end{multline}
         where we denoted 
     \begin{equation}
	 \label{def_U_m}
	    U_m^1\defn\overline{U}^1+V_m^1\defn\overline{U}^1+\Dot{S}_{m+1}V^1=\overline{U}^1+\sum_{j\le m}\DDj V^1.
	 \end{equation}
	 Thanks to  the spectral localization of $U^1,$  
	 the left-hand side of  \eqref{equ_V_2_modif} may be seen as a parabolic system  with \emph{smooth} coefficients.
	 We expect the  error induced by  these localizations to tend to $0$ when $m$ goes to $\infty.$  
As  $V^1\in \cC([0,T]; \dot B^{\frac d2}_{2,1})$ and $B^{\frac d2}_{2,1}\hookrightarrow L^\infty,$
	we have  \eqref{def_U_O_1} (for, possibly, a slightly larger set $\cO^1$ compactly embedded in $\cU$)	
	for $m $ large enough. Note also that
Bernstein's inequality ensures that  there exists a constant $C>0$ independent of  $m$ so that for all  $\gamma\geq d/2,$ we have
     \begin{equation}
	 \label{injection_S_m}
	    \NBH{\gamma}{2}{1}{V_m^1}\le  C 2^{m(\gamma-\cd)} \NBH{\cd}{2}{1}{V^1}.
	 \end{equation}
  We aim at getting uniform estimates on $ V^2$ in suitable Besov spaces. For that, 
  as in the previous section, we set   $V_S \defn V^2-V_L^2$
where $V^2_L$ stands for the solution of \eqref{eq_lin-dif} with initial data $V_0^2.$
This function satisfies the following parabolic system:
	  \begin{multline}    \label{eq_V_S}
   S^0_{22}(U_m^1) \pt V_S- \sumab\pal(Z^{\alpha\beta}(U^1_m) \pbe V_S)\\
       = \sumab\pal\left((Z^{\alpha\beta}(U^1)-Z^{\alpha\beta}(U^1_m)) \pbe V_S\right)
   + R^t +R^{21}+R^{22}+R^L,\end{multline}
	where 
    \begin{align*}
                   R^t\defn  (S^0_{22}(U_m^1)-S^0_{22}(U^1))\pt V^2,\quad  R^{21} \defn -\suma S^\alpha_{21}(U) \pal V^1,
                  \quad R^{22}\defn -\suma   S^\alpha_{22}(U) \pal V^2\\
          \andf R^L \defn (\overline{S}^0_{22}-S^0_{22}(U^1))\pt V_L^2   +\sumab \pal( (Z^{\alpha \beta}(U^1)-\overline{Z^{\alpha \beta}}) \pbe V_L^2).           \hspace{2cm}                \end{align*}
   \begin{proposition}
     \label{energy_V^S}
  Under the hypotheses of Proposition \ref{energy_V^1} there exists a constant $C$ depending on $\cO^1$ and 
  on the matrices of the system such that, setting
 \begin{align*}
         &\mathfrak{A}^S(T) \defn \TLpNBH{\cd-1}{2}{1}{V_S}{\infty}+ \LpNBH{\cd+1}{2}{1}{V_S}{1};\;\;  \mathfrak{A}^L(T) \defn  \LpNBH{\cd-1}{2}{1}{\pt V^2_L}{1}+ \LpNBH{\cd+1}{2}{1}{V^2_L}{1},
      \end{align*}
    we have 
    \begin{multline}\label{est:vs:1:c}
      \bigl( 1-C(T+2^m   \LpNBH{\frac{d}{2}-1}{2}{1}{\pt V^1}{1}) \bigr)\mathfrak{A}^S(T) 
         \le  CM_1^3\Bigl(\bigl(1+\|V_0^2\|_{\dot B^{\frac d2-1}_{2,1}}\bigr)\mathfrak{A}^L(T)\\
          +\Bigl(\LpNBH{\frac{d}{2}}{2}{1}{V^1-V^1_m}{\infty}+2^m \sqrt T+\mathfrak{A}^S(T)\Bigr)\mathfrak{A}^S(T)
        +  \LpNBH{\frac{d}{2}}{2}{1}{V^1- V^1_m}{\infty}M_3     +T\Bigr)\cdotp 
\end{multline}
Moreover, 
\begin{align}
\label{esti:pt:v1}
   \LpNBH{\cd-1}{2}{1}{\pt V^1}{1} &\le CM_1\sqrt T \Bigl(\sqrt T + \mathfrak{A}^S(T) +\sqrt{ \|V_0^2\|_{\dot B^{\frac d2-1}_{2,1}} \mathfrak{A}^L(T)}\Bigr)
\andf\\\label{esti:pt:v2}
     \LpNBH{\cd-1}{2}{1}{\pt V^2}{1} &\le CM_1^3\Bigl(T+   \|V_0^2\|_{\dot B^{\frac d2-1}_{2,1}} \mathfrak{A}^L(T) + M_2+     
      (\mathfrak{A}^S(T))^2\Bigr)\cdotp\end{align}
  \end{proposition}
\begin{proof}
Applying $S^0_{22}(U_m^1)\DDj(S^0_{22}(U_m^1))^{-1}$ to \eqref{eq_V_S} gives
  \begin{align*}
     S^0_{22}(U_m^1) \pt \DDj V_{S}- Z^{\alpha \beta} (U^1_m)\pal\pbe \DDj V_{S}= R_j^{L}+R_j^{21}+R_j^{22}+R_j^{S}+R^t_j+E^S_j,
     \end{align*}
where we denote:  
\begin{align*}
    R_j^{L} &\defn   S^0_{22}(U_m^1)\DDj \biggl( ( S^0_{22}(U_m^1))^{-1} R^{L}\biggr);\;  R_j^{t} \defn  S^0_{22}(U_m^1)\DDj \biggl( ( S^0_{22}(U_m^1))^{-1}  R^{t}\biggr),\\
    R_j^{21}& \defn  S^0_{22}(U_m^1)\DDj \biggl( ( S^0_{22}(U_m^1))^{-1}  R^{21}\biggr);\;  R_j^{22} \defn  S^0_{22}(U_m^1)\DDj \biggl( ( S^0_{22}(U_m^1))^{-1}  R^{22}\biggr);\;\\
    E_j^S &\defn \sumab  S^0_{22}(U_m^1)\DDj \biggl( ( S^0_{22}(U_m^1))^{-1}\pal\left(Z^{\alpha\beta}(U^1)-Z^{\alpha\beta}(U^1_m) \pbe V_S\right)\biggr),\\
      R_j^{S} &\defn  S^0_{22}(U_m^1) \sumab\left[\DDj, \wt Z^{\alpha\beta}(U^1_m) \right]\pal\pbe V_S\\
      &\hspace{3cm}+\sumab  S^0_{22}(U_m^1)\DDj \biggl( ( S^0_{22}(U_m^1))^{-1} \left(\pal(Z^{\alpha\beta}(U^1_m))\pbe V_S\right)\biggr)\cdotp
     \end{align*}

Perform the same energy method  as in the proof of Proposition \ref{energ:est:V2:lem}. Remembering that \eqref{def_U_O_1} holds, we get
 \begin{multline}
\label{esti_v_2-2'''}
     	       \Vert \DDj V_{S} \Vert_{L^\infty_T(L^2)} +c2^{2j}\int^t_0\normede{\DDj V_S}\\ \le  C\biggl(\int^T_0\normede{\DDj V_{S}} (1+\normeinf{\pt(S^0_{22}(U_m^1))} )
	        +\int^T_0\normede{(R_j^{L},R_j^{21}, R_j^{22},R_j^{S},E_j^S,R^t_j)}\biggr)\cdotp
	   	\end{multline}
Note that \eqref{def_U_O_1}, the embedding  $ \BH{\frac{d}{2}}{2}{1} \hookrightarrow  L^\infty$ and  \eqref{injection_S_m}  lead to :
\begin{align}
\label{esti:v1:m:inf}
    \normeinf{\pt(S^0_{22}(U_m^1))} &\le C \normeinf{\pt V_m^1} \le C \NBH{\frac{d}{2}}{2}{1}{\pt V_m^1}\le C2^m \NBH{\frac{d}{2}-1}{2}{1}{\pt V^1} .
\end{align}
Owing to Proposition \ref{propo_produc_BH} and \ref{propo_compo_BH}, we have for $\alpha,\beta=1\cdots, d,$
\begin{align*}
 &\Vert  (\overline{S}^0_{22}-S^0_{22}(U_m^1))\pt V^2_L \Vert_{L^1_T(\BH{\cd-1}{2}{1})}\le C \LpNBH{\frac{d}{2}}{2}{1}{V_m^1}{\infty}  \LpNBH{\frac{d}{2}-1}{2}{1}{\pt V^2_L}{1},\\
 &  \Vert \pal(Z^{\alpha \beta}(U^1)-\overline{Z^{\alpha\beta}})\pbe V^2_L  )\Vert_{L^1_T(\BH{\cd-1}{2}{1})}\le  C \LpNBH{\frac{d}{2}}{2}{1}{V^1}{\infty} \LpNBH{\frac{d}{2}+1}{2}{1}{V^2_L}{1},
\end{align*}
whence, using also  \eqref{injection_S_m} yields the following bound on $  R^{L}$,
\begin{align*}
    \LpNBH{\cd-1}{2}{1}{R^L}{1}\le  \LpNBH{\frac{d}{2}}{2}{1}{V^1}{\infty} \Bigl(\LpNBH{\frac{d}{2}-1}{2}{1}{\pt V^2_L}{1}+ \LpNBH{\frac{d}{2}+1}{2}{1}{V^2_L}{1} \Bigr)\cdotp\end{align*}
Remembering the product law  \eqref{product_propo2}, we conclude that
\begin{align}
  \label{esti_v:L}
  \sum_j 2^{j(\cd-1)} \Vert  R_j^{L}  \Vert_{L^1_T(L^2)}\le  C (1+ \LpNBH{\frac{d}{2}}{2}{1}{ V^1_m}{\infty}) \LpNBH{\frac{d}{2}}{2}{1}{V^1}{\infty} \cU^L(T).\end{align}
The next step is to bound $  R_j^{21} $ and $  R_j^{22} $ in $ L^1_T(L^2) $. The term $  R^{21} $ can be decomposed as
\begin{align*}
     R^{21}= -\suma \overline{S^\alpha_{21}}\pal V^1+ \suma  \Bigl(\overline{S^\alpha_{21}}- S^\alpha_{21}(U)\Bigr)\pal V^1.
\end{align*}
Bearing in mind  that  $ S^\alpha_{21}$  is affine with respect to  $V^2$ and using   Propositions \ref{propo_Compo_u_v_BH} (especially \eqref{comp:u_v:BH:ine:prop:1} with $s=\cd$) and \ref{propo_produc_BH} to bound the second term in the previous identity yields  
\begin{multline}
\label{esti_R_21}
 \sum_j  2^{j(\cd-1)} \Vert  R_j^{21}  \Vert_{L^1_T(L^2)}\\ \le C \bigl(1+ \LpNBH{\frac{d}{2}}{2}{1}{ V^1_m}{\infty}\bigr)
 \bigl(1+\LpNBH{\frac{d}{2}}{2}{1}{V^1}{\infty}\bigr)\Bigl(T
    +\sqrt{T}\LpNBH{\frac{d}{2}}{2}{1}{V^2}{2}\Bigr) \LpNBH{\cd}{2}{1}{V^1}{\infty}.
    \end{multline}
Since the matrices  $S^\alpha_{22}$ have the same structure as $S^\alpha_{21} $, the term $  R_j^{22} $ may be bounded  as $R_j^{21},$ 
which gives
\begin{multline}\label{esti_R_22}
 \sum_j 2^{j(\cd-1)}\Vert  R_j^{22}  \Vert_{L^1_T(L^2)} 
   \\ \le C (1+ \LpNBH{\frac{d}{2}}{2}{1}{ V^1_m}{\infty})
     \bigl(1+\LpNBH{\frac{d}{2}}{2}{1}{V^1}{\infty}\bigr)\Bigl(\sqrt T
    +\LpNBH{\frac{d}{2}}{2}{1}{V^2}{2}\Bigr) \LpNBH{\cd}{2}{1}{V^2}{2}.\end{multline}
To bound the term $ R^t_j ,$ we take advantage of Propositions \ref{propo_produc_BH} and \ref{propo_compo_BH} and get
\begin{align}\label{esti_R_2}
    \sum_j  2^{j(\cd-1)}\Vert R^t_j   \Vert_{L^1_T(L^2)} &\le C(1+ \LpNBH{\frac{d}{2}}{2}{1}{ V^1_m}{\infty}) \LpNBH{\frac{d}{2}}{2}{1}{V^1-V^1_m}{\infty} \LpNBH{\frac{d}{2}-1}{2}{1}{\pt V^2}{1}.
\end{align}
Bounding $ R^S_j$ and $ E^S_j$ involves Propositions \ref{propo_commutator-BH} (with $\sigma=\cd-1$), \ref{propo_compo_BH} (with $ s=\cd $) and \ref{propo_produc_BH}  combined with \eqref{def_U_O_1}: we get
\begin{align*}
\begin{split}
    & \sum_j 2^{j(\cd-1)} \left\Vert   \left[\DDj, \wt Z^{\alpha\beta}(U^1_m) \right]\pal\pbe V_S \right\Vert_{L^1_T(L^2)}  \le  C\sqrt{T}  \LpNBH{\frac{d}{2}}{2}{1}{\n V^1_m}{\infty} \LpNBH{\frac{d}{2}-1}{2}{1}{\n  V_S}{2} ,\\
     &  \sum_j 2^{j(\cd-1)} \left\Vert \DDj\Bigl( (S^0_{22}(U^1_m))^{-1}\pal(Z^{\alpha\beta}(U^1_m))\pbe V_S   \Bigr)\right\Vert_{L^1_T(L^2)} \\
    &\hspace{4cm}\le  C\sqrt{T} (1+ \LpNBH{\frac{d}{2}}{2}{1}{ V^1_m}{\infty})  \LpNBH{\frac{d}{2}}{2}{1}{\n V^1_m}{\infty} \LpNBH{\frac{d}{2}-1}{2}{1}{\n V_S}{2},\\ 
    &   \sum_j 2^{j(\cd-1)} \left\Vert  \DDj\biggl((S^0_{22}(U^1_m))^{-1}\pal\Bigl((Z^{\alpha\beta}(U^1)\!-\!Z^{\alpha\beta}(U^1_m)) \pbe V_S\Bigr) \biggr)  \right\Vert_{L^1_T(L^2)} \\
    &\hspace{4cm}\le  C (1+ \LpNBH{\frac{d}{2}}{2}{1}{ V^1_m}{\infty}) \LpNBH{\frac{d}{2}}{2}{1}{V^1- V^1_m}{\infty} \LpNBH{\frac{d}{2}}{2}{1}{\n  V_S}{1}.
    \end{split}
\end{align*}
Hence, owing to \eqref{injection_S_m} and to the fact that
$$
 \LpNBH{\frac{d}{2}}{2}{1}{\n  V_S}{1}\leq C \LpNBH{\frac{d}{2}-1}{2}{1}{V_S}{\infty}^{1/2} \LpNBH{\frac{d}{2}+1}{2}{1}{V_S}{1}^{1/2}
 \leq C \mathfrak{A}^S(T),
$$ we deduce that 
\begin{align}    \label{est:E_m}
   \sum_j 2^{j(\cd-1)}\left\Vert E^S_j   \right\Vert_{L^1_T(L^2)}     &\le  C (1+ \LpNBH{\frac{d}{2}}{2}{1}{ V^1_m}{\infty})   \LpNBH{\frac{d}{2}}{2}{1}{V^1- V^1_m}{\infty} \mathfrak{A}^S(T),\\  \label{est:RS}
   \sum_j 2^{j(\cd-1)} \left\Vert R^S_j  \right\Vert_{L^1_T(L^2)} &\le  C 2^m\sqrt{T} (1+ \TLpNBH{\frac{d}{2}}{2}{1}{ V^1_m}{\infty})  \LpNBH{\frac{d}{2}}{2}{1}{ V^1}{\infty}  \mathfrak{A}^S(T).
\end{align}
Inserting  Inequalities \eqref{esti:v1:m:inf} to \eqref{est:RS}  into \eqref{esti_v_2-2'''}, then   summing  over $ j\in\Z$, using the definition of $M_1, M_3$ (note that $M_1\geq1$)  and remarking that $ \LpNBH{\frac{d}{2}}{2}{1}{ V^1_m}{\infty}\le M_1 $ yields
$$\displaylines{
	         \bigl( 1-C(T+2^m   \LpNBH{\frac{d}{2}-1}{2}{1}{\pt V^1}{1}) \bigr)\mathfrak{A}^S(T) 
         \le  CM_1^3\Bigl(\mathfrak{A}^L(T) +\bigl(\LpNBH{\frac{d}{2}}{2}{1}{V^1-V^1_m}{\infty}+2^m \sqrt T\bigr)\mathfrak{A}^S(T)
     \hfill\cr\hfill    +  \LpNBH{\frac{d}{2}}{2}{1}{V- V^1_m}{\infty}M_3
         +T+\sqrt T \LpNBH{\cd}{2}{1}{V^2}{2}+\LpNBH{\cd}{2}{1}{V^2}{2}^2\Bigr)\cdotp}$$
Finally, since $V^2=V^2_L+V_S,$ using interpolation and the fact that 
$$\LpNBH{\cd-1}{2}{1}{V^2_L}{\infty}\leq C_0\NBH{\cd-1}{2}{1}{V_0^2}$$
gives
\begin{align}        \label{intp:to:V2:res}
         \LpNBH{\cd}{2}{1}{V^2}{2}\le \mathfrak{A}^S +C_0\sqrt{\NBH{\cd-1}{2}{1}{V_0^2}\mathfrak{A}^L},   \end{align}
from which we get \eqref{est:vs:1:c}. 
\medbreak
To prove   \eqref{esti:pt:v1} and   \eqref{esti:pt:v2},  we recall that
 \begin{align*}
      &\pt V^1 =-\suma \left( A^\alpha_{11}(U) \pal V^1 + A^\alpha_{12}(U) \pal V^2\right),\\
	     &\pt V^2 =(S^0_{22}(U^1))^{-1}\biggl(\sumab\pal(Z^{\alpha \beta}(U^1)\pbe V^2)-\suma\left(S^\alpha_{21}(U) \pal V^1 + S^\alpha_{22}(U) \pal V^2\right)\biggr)\cdotp
 \end{align*}
Then, thanks to  Propositions \ref{propo_produc_BH} and \ref{propo_Compo_u_v_BH} and remembering  Assumption \textbf{C},\begin{align*}
  &  \LpNBH{\cd-1}{2}{1}{ A^\alpha_{11}(U) \pal V^1}{1}\le C(T+ \sqrt{T}\LpNBH{\cd}{2}{1}{V^2}{2}  )\LpNBH{\cd}{2}{1}{V^1}{\infty},\\
&  \LpNBH{\cd-1}{2}{1}{ A^\alpha_{12}(U) \pal V^2}{1} \le C\sqrt{T} ( 1 +
     \LpNBH{\cd}{2}{1}{V^1}{\infty} )\LpNBH{\cd}{2}{1}{V^2}{2},\\
      &  \LpNBH{\cd-1}{2}{1}{ (S^0_{22}(U^1))^{-1} S^\alpha_{21}(U) \pal V^1 }{1} \le C\bigl(1+ \LpNBH{\cd}{2}{1}{V^1}{\infty}\bigr)\bigl( T(1+
     \TLpNBH{\cd}{2}{1}{V^1}{\infty})\\ &\hspace{7cm}+
     \sqrt{T} \LpNBH{\cd}{2}{1}{V^1}{\infty}\LpNBH{\cd}{2}{1}{V^2}{2} \bigr)\LpNBH{\cd}{2}{1}{V^1}{\infty},\\
      &  \LpNBH{\cd-1}{2}{1}{ (S^0_{22}(U^1))^{-1} S^\alpha_{22}(U) \pal V^2 }{1} \le C\bigl(1+ \LpNBH{\cd}{2}{1}{V^1}{\infty}\bigr)\bigl( \sqrt{T}(1+
     \LpNBH{\cd}{2}{1}{V^1}{\infty}) \\&\hspace{7cm}+
   \LpNBH{\cd}{2}{1}{V^1}{\infty}  \LpNBH{\cd}{2}{1}{V^2}{2} \bigr)\TLpNBH{\cd}{2}{1}{V^2}{2},\\
      &  \LpNBH{\cd-1}{2}{1}{ (S^0_{22}(U^1))^{-1} \pal(Z^{\alpha \beta}(U^1)\pbe V^2) }{1} \le C(1+ \LpNBH{\cd}{2}{1}{V^1}{\infty})^2 \LpNBH{\cd+1}{2}{1}{V^2}{1}.
    \end{align*}
    Taking advantage of  Inequality \eqref{intp:to:V2:res} completes the proof of  \eqref{esti:pt:v1} and  \eqref{esti:pt:v2}.
       \end{proof}


\paragraph{\bf Step $3$: Closing the estimates} \label{stability:crit}

Let us set $d_1\defn \frac12 d(\cO_0^1,\partial\cU)$ and define $\cO^1$ to be a $d_1$-neighborhood of $\cO_0^1.$
Let 
\begin{equation}\label{eq:M1}M_1\defn 1+ 2\|V_0^1\|_{\dot B^{\frac d2}_{2,1}}.\end{equation}
In this part, we are going to prove that if $m\in\N$ is chosen sufficiently large, then $\eta\in(0,1)$ sufficiently small, 
 one can find some   $T\in(0,1)$ depending only on $m,$ $\eta$ and on the initial data (and of the matrices of the system) so that the solution $V$ satisfies the following:
\begin{enumerate}
    \item [$(\mathbb{C}1)$] $\TLpNBH{\cd}{2}{1}{V^1}{\infty}
    \le M_1,$
     \item [$(\mathbb{C}2)$]     $\TLpNBH{\cd}{2}{1}{V^1-V^1_m}{\infty}\le \sqrt\eta$,
    \item [$(\mathbb{C}3)$]  $ |V^1(t,x)-V^1_0(x)|\le d_1 \esp{for any} (t,x)\in [0,T]\times\Rd $,\smallbreak
      \item [$(\mathbb{C}4)$]   $ \TLpNBH{\cd-1}{2}{1}{V_S}{\infty}   +\LpNBH{\cd+1}{2}{1}{V_S}1 \le \eta $ ,
      \item [$(\mathbb{C}5)$] $ \mathfrak A^L(T) \le \eta^2.$
\end{enumerate}
Note that Property $(\mathbb{C}3)$ readily ensures that   $U^1(t,x) \in \cO^1$ for any $(t,x)\in [0,T]\times\Rd,$ 
so that we will be able to use the estimates of the previous subsections.
\smallbreak
Before starting the proof, we fix $m\in\N$ large enough so that 
\begin{equation}\label{eq:fixm}
\sum_{j\geq m}2^{j\cd}\|V^1_{0,j}\|_{L^2}\leq\frac12\sqrt \eta,
\end{equation}
and $T$   small enough so that 
\begin{equation}\label{eq:fixT}
T\leq \sup\Bigl\{h>0,\: C_0\sum_{j\in\Z} \biggl(1-e^{-c2^{2j}h}\Bigr)2^{j(\cd-1)}\|\DDj V_0^2\|_{L^2}\leq\eta^2\biggr\},\end{equation}
where $C_0$ is the constant in Inequality \eqref{L-1_j} (adapted to the homogeneous setting).
Note that this readily ensures $(\mathbb{C}5).$
\medbreak
\subsubsection*{Substep 1. Proving $(\mathbb{C}1)$} Taking advantage of  Inequality \eqref{Esti_V_1} with $m=+\infty,$ then using 
Properties $(\mathbb{C}1)$ to $(\mathbb{C}5)$  and the definition of $M_1$ in \eqref{eq:M1} gives
\begin{align*}
   \TLpNBH{\cd}{2}{1}{V^1}{\infty}&\leq  \NBH{\cd}{2}{1}{V_0^1} + 
   C\bigl(  \LpNBH{\cd+1}{2}{1}{V_S}1 +  \LpNBH{\cd+1}{2}{1}{V_L^2}1\bigr)\bigl(1+  \TLpNBH{\cd}{2}{1}{V^1}{\infty}\bigr)\\
   &\leq C\eta -\frac12 + \biggl(\frac12+C\eta\biggr) M_1.
   \end{align*}
Hence  $(\mathbb{C}1)$ holds true with strict inequality provided $\eta$ has been chosen  so that $C\eta<1/2.$ 

\subsubsection*{Substep 2. Proving $(\mathbb{C}2)$} 
From  Inequality \eqref{Esti_V_1} with $m$ given by \eqref{eq:fixm}, we gather
\begin{align*} 
\TLpNBH{\cd}{2}{1}{V^1-V^1_m}{\infty}&\!\le \sum_{j\geq m}\! 2^{j\cd} \|V_j^1\|_{L_T^\infty(L^2)}\\
&\leq \sum_{j\geq m}\! 2^{j\cd} \|V_{0,j}^1\|_{L^2} + C\bigl(  \LpNBH{\cd\!+\!1}{2}{1}{V_S}1\! +\!  \LpNBH{\cd\!+\!1}{2}{1}{V_L^2}1\bigr)
 \bigl(1\!+\!  \TLpNBH{\cd}{2}{1}{V^1}{\infty}\bigr)\\
 &\leq \sqrt\eta/2 + C\eta(1+M_1).\end{align*}
 Hence we have $(\mathbb{C}2)$  with strict inequality if  $\eta$ has been chosen  so that $C\sqrt\eta(1+M_1)<1/2.$

\subsubsection*{Susbstep 3. Proving $(\mathbb{C}3)$} 

We use the fact that:
\[ V^1-V^1_0= \Dot{S}_{m+1}(V^1-V^1_0)+({\rm Id}-\Dot{S}_{m+1})(V^1-V^1_0). \]
Then, the embedding  $ \BH{\frac{d}{2}}{2}{1} \hookrightarrow L^\infty ,$ the fact that $ \DDj({\rm Id}- \Dot{S}_{m+1})=0$ if $j< m$ 
and Inequality \eqref{injection_S_m} ensure that
for all $t\in[0,T],$
$$\begin{aligned} \| V^1(t)-V^1_0\|_{L^\infty} 
&\le C\biggl(\int^T_0 \NBH{\cd}{2}{1}{\Dot{S}_{m+1}\pt V^1}+\sum_{j\ge m}2^{j\cd}\normede{\DDj (V^1(t)-V^1_0)}\biggr)\\
&\le C\biggl(2^m \LpNBH{\cd-1}{2}{1}{\pt V^1}{1}+\TLpNBH{\cd}{2}{1}{V^1-V^1_m}{\infty}+
\sum_{j\geq m} 2^{j\cd} \|V_{0,j}\|_{L^2}\biggr)\cdotp\end{aligned}$$
The function $\partial_t V^1$ may be bounded from \eqref{esti:pt:v1} and  $(\mathbb{C}3)-(\mathbb{C}4)$ as follows:
\begin{equation}\label{eq:boundptv1}
 \LpNBH{\cd-1}{2}{1}{\pt V^1}{1}\leq CM_1\sqrt T\Bigl(\sqrt T + \eta\bigl(1+\|V_0^2\|_{\dot B^{\cd-1}_{2,1}}\bigr)\Bigr)\cdotp
 \end{equation}
Remembering \eqref{eq:fixm} and  $(\mathbb{C}2),$ we thus get 
$$ \| V^1(t)-V^1_0\|_{L^\infty} \le CM_12^m \sqrt T\Bigl(\sqrt T + \eta\bigl(1+\|V_0^2\|_{\dot B^{\cd-1}_{2,1}}\bigr)\Bigr)
+C\sqrt \eta.$$
Hence $(\mathbb{C}3)$ is satisfied for sufficiently small $T$ if $\eta$ has been chosen so that $C\sqrt\eta\leq d_1/2.$

\subsubsection*{Susbstep 4. Proving $(\mathbb{C}4)$} 

Owing to \eqref{eq:boundptv1},   if $T$ has been chosen small enough then for any $\eta\in(0,1),$
the negative part of the prefactor of $\partial_tV^1$ in \eqref{est:vs:1:c} may be omitted, and we get
(up to a change of $C$): 
$$\mathfrak{A}^S(T) 
         \le  CM_1^3\Bigl(\bigl(1+\|V_0^2\|_{\dot B^{\frac d2-1}_{2,1}}\bigr)\eta^2          +\bigl(2^m \sqrt T+\sqrt\eta\bigr)\mathfrak{A}^S(T)
        +  \eta^{3/2}    +T\Bigr)\cdotp$$
If $\eta$ and $T$ are such that $C(2^m \sqrt T+\sqrt\eta)\leq1/2,$ this gives
$$\mathfrak{A}^S(T) 
         \le  2CM_1^3\Bigl(\bigl(1+\|V_0^2\|_{\dot B^{\frac d2-1}_{2,1}}\bigr)\eta^2              +  \eta^{3/2}    +T\Bigr)\cdotp$$
From it, we get  $(\mathbb{C}4)$ with a strict inequality if, say, 
$$
2CM_1^3\Bigl(\bigl(1+\|V_0^2\|_{\dot B^{\frac d2-1}_{2,1}}\bigr)\eta +  \sqrt\eta\Bigr)<1/2
\andf 2CM_1^3T<1/2.$$

\subsubsection*{Susbstep 5. Bootstrap}  
Since all the quantities coming into play in  $(\mathbb{C}1)- (\mathbb{C}5)$ are continuous in time
and since the desired properties are true for $T=0,$ we are guaranteed 
that they are also true on a small enough time interval $[0,T_0].$  In the previous
computations, we pointed out some $T>0$ \emph{depending only on the initial data} 
such that if  $(\mathbb{C}1)- (\mathbb{C}5)$ are satisfied, then they actually hold \emph{with strict inequality}. 
The usual connectivity argument thus ensures that we do have  $(\mathbb{C}1)- (\mathbb{C}5)$ on $[0,T],$
which completes the proof.



\subsection{The proof of existence} 

Let us smooth out the initial data $V_0$ as follows: 
$$
V_{0,p}:=\sum_{|j|\leq p} \DDj V_0,\qquad p\in\N.$$
Then, we see that \eqref{eq:fixm} can be ensured independently of $p,$ and we have 
\begin{equation}\label{eq:datasmooth}
V_{0,p}\to V_0 \hbox{ in } \dot B^{\frac d2}_{2,1}\times \dot B^{\frac d2-1}_{2,1}\andf
\underset{p\in\N}\sup \|V_{0,p}\|_{\dot B^{\frac d2}_{2,1}\times \dot B^{\frac d2-1}_{2,1}}
\leq \|V_0\|_{\dot B^{\frac d2}_{2,1}\times \dot B^{\frac d2-1}_{2,1}}.\end{equation}
The fact that $\dot B^{\frac d2}_{2,1}\hookrightarrow L^\infty$ guarantees that \eqref{def_U_O_1}
is satisfied for large enough $p.$ 
Now, since $V_{0,p}$ belongs to all spaces $B^s_{2,1},$ 
applying Theorem \ref{Thm:loc:scri} gives us a sequence of smooth local solutions $(V_p)_{p\in\N}$ 
on some maximal time interval $[0,T_p).$
Since the solutions are smooth,  the computations that have been performed in the previous section hold true;
 keeping in mind our definition of smoothed out data and \eqref{eq:datasmooth}, 
properties $({\mathbb C}_1)$ to $({\mathbb C}_5)$ are  satisfied on $[0,T_p^*[$ with $T^*_p:=\min(T_p,T)$ and $T$ given therein.
The important point is that these conditions  (and embedding) ensure  
that $\nabla V^1_p\in L^2(0,T_p^*;L^\infty)$ and $\nabla V^2_p\in L^1(0,T_p^*;L^\infty).$
Consequently, applying the continuation criterion pointed out in Remark \ref{rmq:improved}  ensures that $T_p>T.$
As a conclusion, we proved that  the lifespan of each term of the sequence is greater than $T,$ 
and that $(V_p)_{p\in\N}$ is bounded in the space $\cE_T.$ 

The rest of the proof is standard. The boundedness in $\cE_T$ guarantees that
$V_p$ converges weakly  to some limit $V$ that belongs to 
$L^\infty(0,T;B^{\frac d2}_{2,1}\times B^{\frac d2-1}_{2,1}).$ 
Then, the boundedness of the time derivatives combined with a Lions-Aubin type argument and, finally, interpolation
gives some strong convergence (locally in space and time) that is enough to pass to the limit 
in \eqref{Eq_b}. 
 To recover the time continuity with values in  $B^{\frac d2}_{2,1}\times B^{\frac d2-1}_{2,1}$ and the $L^1_T$
 properties coming into play in the definition of $\cE_T,$ one may argue as for the compressible 
 Navier-Stokes equations (see e.g. \cite[Chap.10]{HajDanChe11}).


\subsection{The proof of uniqueness}\label{sec:uni:crit}

To simplify the presentation, we here assume that $S^0={\rm Id}.$ 
Let $V_1$ and $V_2$ be two solutions of \eqref{Eq_b}  on $[0,T]\times\Rd$ given by  Theorem \ref{thm:loc:cri}
and corresponding to the same initial data. Let $\dV\defn V_2-V_1.$ 
The proof of uniqueness  consists in obtaining suitable a priori estimates for the following system satisfied by $\dV$:
 \begin{align}
 \label{eq:tv}
     \begin{cases}
    \pt \dV^1+\suma A^\alpha_{11}(U^2_2) \dV^1=h, \\[1ex]
  \pt\dV^2 -\sumab  Z^{\alpha\beta}(U^1_{2,m})\pal \pbe \dV^2=g\defn \sum_{k=1}^7 g_k,
\end{cases}
 \end{align} 
with
\begin{align*}
    &  h= \suma \left( A^\alpha_{11}(U^2_2)-A^\alpha_{11}(U^2_1)\right)\pal V^1_1-\suma A^\alpha_{12}(U_2^1)\pal \dV^2 -\suma \left( A^\alpha_{12}(U^1_2)-A^\alpha_{12}(U^1_1)\right)\pal V^2_1,\\ 
      &g_1=- \suma\left(  S^\alpha_{22}(U_2)-S^\alpha_{22}(U_1)\right)\pal V^2_1,\quad
      g_2= -\suma S^\alpha_{22}(U_2)\pal \dV^2 \\
    & g_3 = - \suma S^{\alpha}_{21}(U_2)\pal \dV^1,\quad g_4=
    -  \suma\left(  S^\alpha_{21}(U_2)-S^\alpha_{21}(U_1)\right)\pal V^1_1,\\
    & g_5 = \sumab \pal \left(Z^{\alpha\beta}(U^1_{2,m})\right)\pbe \dV^2, \quad
     g_6 = \sumab \pal\left((  Z^{\alpha\beta}(U^1_{2}) -  Z^{\alpha\beta}(U^1_{2,m})) \pbe \dV^2\right), \\& g_7= \sumab \pal\left((  Z^{\alpha\beta}(U^1_{2}) - Z^{\alpha\beta}(U^1_{1})) \pbe V^2_1 \right).
\end{align*}
Like in Section \ref{sec:proof:loc:sur}, uniqueness has to be proved in a space with one less derivative, namely
\begin{align*}
   \widetilde{L}^\infty_T(\BH{\cd-1}{2}{1})\times \biggl( \widetilde{L}^\infty_T(\BH{\cd-2}{2}{1})\cap L^1_T(\BH{\cd+1}{2}{1})\biggr)\cdotp
\end{align*}
This would indeed work in dimension $d\geq3.$ In dimension $d=2$ however,  
this would lead us to estimating  the right-hand side of  \eqref{eq:tv}$_2$  in $L^1_T(\dot B^{-1}_{2,1}).$
Terms like $g_1$ or $g_4$ are not tractable in this low regularity framework since, typically, 
  the product of functions  maps $ \BH{0}{2}{1}(\R^2)\times\BH{0}{2}{1}(\R^2) $ in 
the larger space $ \BH{-1}{2}{\infty}(\R^2),$ rather than in $ \BH{-1}{2}{1}(\R^2).$ 
Bounding  $\dV^2$ in  ${L}^\infty_T(\BH{-1}{2}{\infty})\cap L^1_T(\BH{1}{2}{\infty})$ is not good either since the fact that  $\BH{1}{2}{\infty} \centernot\hookrightarrow  L^\infty$ 
causes 
some problem when estimating $h.$ 
Following \cite{Danchin05}, we shall  bypass this difficulty  leveraging the following logarithmic interpolation inequality:
 \begin{equation}
        \label{ine:log:inte:NHB}
        \LpNBH{\cd}{2}{1}{a}{1}\le C \TLpNBH{\cd}{2}{\infty}{a}{1} 
        \log\left( e+ 
        \frac{\TLpNBH{\cd-1}{2}{\infty}{a}{1}+\TLpNBH{\cd+1}{2}{\infty}{a}{1}}{ \TLpNBH{\cd}{2}{\infty}{a}{1}}
        \right)\cdotp
    \end{equation}
Consequently, in what follows, we shall  estimate $(\dV^1,\dV^2)$  in
 \begin{align*}
{\mathcal F}_T\defn    {L}^\infty_T(\BH{\cd-1}{2}{\infty})\times {L}^\infty_T(\BH{\cd-2}{2}{\infty})\cap \widetilde{L}^1_T(\BH{\cd}{2}{\infty}).
\end{align*}
Although $V_1$ and $V_2$ need not be in ${\mathcal F}_T,$ their difference is, 
as a consequence of the following computations (see e.g. \cite{Danchin05} for more explanations).
Now, apply  operator $\DDj$ to \eqref{eq:tv}$_1$ to get 
\[    \pt \dV^1_j+\suma A^\alpha_{11}(U^2_2) \dV^1_j= \DDj h + \widetilde{R}_j \ \text{ with }  \ \widetilde{R}_j =-\suma[\DDj,A^\alpha_{11}(U^2_2)]\pal \dV^1.\]
As $\widetilde V^1(0)=0$ and $ A^\alpha_{11}(U^2_2) $ is affine, using the energy method gives
\begin{align*}
    \Vert \dV^1_j\Vert_{L^\infty_T(L^2)} \le  C\Vert \dV^1_j\Vert_{L^\infty_T(L^2)} \Vert \n V^2_2\Vert_{L^1_T(L^\infty)}+   \Vert(\widetilde R_j , \DDj h)\Vert_{L^1_T(L^2)}.
\end{align*}
Bounding $ \Vert( \wt R_j , \DDj h)\Vert_{L^\infty_T(L^2)}$ is achieved by combining Propositions \ref{propo_commutator-BH} and \ref{propo_produc_BH}. We have
\begin{align*}
  & 2^{j(\cd-1)} \Vert \widetilde R_j\Vert_{L^\infty_T(L^2)} \le C \LpNBH{\cd-1}{2}{\infty}{\dV^1}{\infty}\Vert \n V^2_2 \Vert_{ \widetilde{L}^1_T(\BH{\cd}{2}{\infty})\cap L^1_T(L^\infty)},\\
    &2^{j(\cd-1)} \Vert \DDj h  \Vert_{L^1_T(L^2)} \le C \int^T_0\biggl(\NBH{\cd}{2}{1}{\dV^2}\NBH{\cd-1}{2}{\infty}{\n V^1_1}\\
    &\hspace{5cm}+ (1+\NBH{\cd}{2}{1}{V^1_2})\NBH{\cd-1}{2}{\infty}{\n \dV^2}
    + \NBH{\cd-1}{2}{\infty}{\dV^1}\NBH{\cd}{2}{1}{\n V^2_1}\biggr)\cdotp
\end{align*}
Finally using the embedding ${L}_T^1(\BH{\cd}{2}{1}) \hookrightarrow \widetilde{L}^1_T(\BH{\cd}{2}{\infty})\cap L^1_T(L^\infty),$ we arrive at 
\begin{multline*}
    \LpNBH{\cd-1}{2}{\infty}{\dV^1}{\infty}\le C   \LpNBH{\cd+1}{2}{1}{(V^2_2,V^2_1)}{1}  \LpNBH{\cd-1}{2}{\infty}{\dV^1}{\infty}\\
    +C\int^T_0\biggl( 1+ \NBH{\cd}{2}{1}{V^1_1}+ \NBH{\cd}{2}{1}{V^1_2} \biggr) \NBH{\cd}{2}{1}{\dV^2}.
\end{multline*}
By virtue of the Lebesgue dominated convergence theorem, $   \LpNBH{\cd+1}{2}{1}{(V^2_2,V^2_1)}{1}  $
 tends to $0$ when $T$ goes to $0.$ Hence there exists a positive time (still denoted by $T$) 
such that the first term on the right-hand side may  be absorbed by the left-hand side. Then, making use of inequality \eqref{ine:log:inte:NHB} and setting 
\begin{align}
    \label{def:M1:M2:tV}
    \begin{split}
   & M_1(T)= \TLpNBH{\cd}{2}{1}{(V^1_1,V^1_2)}{\infty}+ \TLpNBH{\cd-1}{2}{1}{(V^2_1,V^2_2)}{\infty} \esp{and}\\
   &M_2(T)= \LpNBH{\cd-1}{2}{1}{(\pt V_1,\pt V_2)}{1}+ \LpNBH{\cd+1}{2}{1}{( V^2_1, V^2_2)}{1}
     \end{split}\end{align}
yields: 
\begin{equation}\label{est:tv1:uni}
      \LpNBH{\cd-1}{2}{\infty}{\dV^1}{\infty} \le C\left( 1\!+\! M_1(T) \right) \TLpNBH{\cd}{2}{\infty}{\dV^2}{1}\log\left( e + \frac{TM_1(T)+M_2(T)}{\TLpNBH{\cd}{2}{\infty}{\dV^2}{1}}\right)\cdotp
\end{equation}
To  bound $ \dV^2,$ apply the operator $\DDj$ to  \eqref{eq:tv}$_2$ to get 
$$\pt\dV^2_j -\sumab  Z^{\alpha\beta}(U^1_{2,m}) \pal\pbe \dV^2_j=\DDj g+ k_j
\with
   k_j\defn  \sumab \left[ \DDj, Z^{\alpha\beta}(U^1_{2,m})\right] \pal\pbe \dV^2_j.$$
Arguing as for proving \eqref{esti_v_2-2'} gives
\begin{align*}\Vert \dV^2_j\Vert_{L^\infty_T(L^2)} + c2^{2j} \Vert \dV^2_j\Vert_{L^1_T(L^2)} \le   \Vert(  \DDj g, k_j)\Vert_{L^\infty_T(L^2)}
+C\Vert \dV^2_j\Vert_{L^1_T(L^2)}.
\end{align*}
Since $d\ge 2$, taking advantage of the commutator estimates of Proposition \ref{propo_commutator-BH} gives  
\begin{align}    \label{est:frc:Gj}
    \Vert k_j \Vert_{L^1_T(L^2)}\le C2^{-j(\cd-2)}\sqrt{T}\TLpNBH{\cd}{2}{1}{\n U_{2,m}^1}{\infty}\TLpNBH{\cd-2}{2}{\infty}{\n\dV^2}{2} ,
     \end{align}
     whence, if $T$ is small enough, using \eqref{injection_S_m},
     \begin{equation}\label{est;tv2:uniq}
     \Vert \dV^2_j\Vert_{L^\infty_T(L^2)} + 2^{2j} \Vert \dV^2_j\Vert_{L^1_T(L^2)} \le  C2^{-j(\cd-2)} 2^m\sqrt{T}M_1 (T)\delta{\mathfrak{U}}(T) +   \Vert  \DDj g\Vert_{L^\infty_T(L^2)},
\end{equation}
where hereafter we put 
\begin{align*}     {\delta\mathfrak{U}}(T)\defn \LpNBH{\cd-2}{2}{\infty}{\dV^2}{\infty}\!+\!\TLpNBH{\cd}{2}{\infty}{\dV^2}{1}.\end{align*}
To bound  the $g_j$'s, we use repeatedly Propositions \ref{propo_commutator-BH}, \ref{propo_produc_BH} and \ref{propo_compo_BH},
 Inequality \eqref{comp:u_v_1:2:BH:ine:prop:inf} (adapted to the spaces $ \widetilde{L}^\rho_T(\BH{s}{2}{r})$) and 
the following two product laws hold true:
\begin{equation}\label{eq:endpoint}
 \BH{\cd-1}{2}{\infty}\times\BH{\cd-1}{2}{1}\to  \BH{\cd-2}{2}{\infty}\andf \BH{\cd-2}{2}{\infty}\times\BH{\cd}{2}{1}\to  \BH{\cd-2}{2}{\infty},\qquad d\geq2.
\end{equation}
We find that
\begin{align*}
    \TLpNBH{\cd-2}{2}{\infty}{g_1}{1}&\le C \sqrt{T}\TLpNBH{\cd-1}{2}{\infty}{\dV^2}{2}(1+\TLpNBH{\cd}{2}{1}{V^1_2}{\infty})\TLpNBH{\cd-1}{2}{1}{ \n V^2_1}{2} \\&\hspace{3cm}+C\int^T_0\NBH{\cd-1}{2}{\infty}{\dV^1}(1+\NBH{\cd}{2}{1}{V^2_1})
     \NBH{\cd-1}{2}{1}{ \n V^2_1}, \\
        \TLpNBH{\cd-2}{2}{\infty}{g_2}{1} &\le  C (\sqrt{T}+\TLpNBH{\cd}{2}{1}{V^2_2}{2})\bigl(1+\TLpNBH{\cd}{2}{1}{V_2^1}{\infty}\bigr)   \TLpNBH{\cd-2}{2}{\infty}{\n \dV^2}{2}, \\
       \TLpNBH{\cd-2}{2}{\infty}{g_{3}}{1}&\le C \int^T_0(1+ \NBH{\cd}{2}{1}{V_2^2})(1+\NBH{\cd}{2}{1}{V_2^1}) \NBH{\cd-2}{2}{\infty}{\n \dV^1},\\      
     \TLpNBH{\cd-2}{2}{\infty}{g_{4}}{1} &\le C \sqrt{T} \TLpNBH{\cd-1}{2}{\infty}{\dV^2}{2}(1+\TLpNBH{\cd}{2}{1}{V^1_{2}}{\infty})\TLpNBH{\cd-1}{2}{1}{ \n V^1_1}{\infty}\\&\hspace{3cm}+C\int^T_0\bigl(1+\NBH{\cd}{2}{1}{V^2_1}\bigr)
     \NBH{\cd-1}{2}{1}{ \n V^1_1} \NBH{\cd-1}{2}{\infty}{\dV^1}.
 \end{align*}
Then,  thanks to \eqref{comp:uv:propo:1} and \eqref{eq:endpoint}, 
$$        \TLpNBH{\cd-2}{2}{\infty}{g_{5}}{1} \le C\sqrt{T}\TLpNBH{\cd}{2}{1}{\n V^1_{2,m}}{\infty}\TLpNBH{\cd-2}{2}{\infty}{\n   \dV^2}{2}.$$
 Finally, thanks to \eqref{eq:endpoint},   Proposition \ref{propo_compo_BH} (especially \eqref{comp:uv:propo:1}) and Bernstein inequality,
\begin{align*}
        \TLpNBH{\cd-2}{2}{\infty}{g_6}{1} &\le C\TLpNBH{\cd}{2}{1}{V^1_2-V^1_{2,m}}{\infty}(1+\TLpNBH{\cd}{2}{1}{V^1_{2}}{\infty})\TLpNBH{\cd-1}{2}{\infty}{\n  \dV^2}{1}\\
      \TLpNBH{\cd-2}{2}{\infty}{g_7}{1} &\le C(1+\TLpNBH{\cd}{2}{1}{(V^1_{1},V^1_{2})}{\infty})\int^T_0\NBH{\cd-1}{2}{\infty}{\dV^1}\NBH{\cd}{2}{1}{\n  V^2_1}.
 \end{align*}
Multiplying \eqref{est;tv2:uniq} by $2^{j(\cd-2)}$,
taking into account the above estimates, using  \eqref{injection_S_m} and interpolation inequalities,  one concludes, assuming with 
no loss of generality that    $M_1(T)\geq1,$
\begin{multline}    \label{est;tv2:uniq:1}
  \delta{\mathfrak{U}}(T) \le  C_{M_1}\biggl(\Bigl(2^{m}M_1(T)\sqrt{T} +\sqrt T\,M_1^2(T)+M_1^{3/2}(T)\sqrt{M_2(T)}\\+\TLpNBH{\cd}{2}{1}{V^1_2-V^1_{2,m}}{\infty} \Bigr) \delta{\mathfrak{U}}(T)
   + \int^T_0\Bigl(1+\NBH{\cd}{2}{1}{(V^1_1,V^2_1)}+\NBH{\cd+1}{2}{1}{V^2_1}\Bigr)\NBH{\cd-1}{2}{\infty}{\dV^1}\biggr)\cdotp
\end{multline}
As, by Lebesgue dominated convergence theorem,   $ \TLpNBH{\cd}{2}{1}{V^1_2-V^1_{2,m}}{\infty} $ and $M_2(T)$ tend to~$0$ when $m $ goes to $\infty,$ 
the first term on the right-hand side may be absorbed by the left-hand side if, 
 first, $m$ is taken large enough then, $T$ is sufficiently small. Inequality \eqref{est;tv2:uniq:1} thus reduces to 
 \begin{align*}
     \delta{\mathfrak{U}}(T) \le  
 C_{M_1}\int^T_0\Bigl(1+\NBH{\cd}{2}{1}{(V^1_1,V^2_1)}+\NBH{\cd+1}{2}{1}{V^2_1}\Bigr) \NBH{\cd-1}{2}{\infty}{\dV^1} .
\end{align*}
We plug \eqref{est:tv1:uni} into this inequality and we use the fact that the function $r\mapsto r\log(e+\frac{1}{r})$ is increasing, to eventually get
\begin{align*}
     \delta{\mathfrak{U}}(T) \le  
  C_{M_1}\int^T_0 \Bigl(1+\NBH{\cd}{2}{1}{(V^1_1,V^2_1)}+\NBH{\cd+1}{2}{1}{V^2_1}\Bigr)     \delta{\mathfrak{U}}\, \log\biggl( e + \frac{TM_1(T)+M_2(T)}{ {\delta{\mathfrak{U}}}}\biggr)\cdotp
\end{align*}
As
\begin{align*}
\Bigl(1+\NBH{\cd}{2}{1}{V^2_1}+\NBH{\cd+1}{2}{1}{(V^1_1,V^2_1)}\Bigr) \in L^1_T \ \text{     and    }\   \int^T_0 \dfrac{1}{r\log(e+\frac{1}{r})}dr=\infty,
\end{align*} 
Osgood's lemma entails that $\delta{\mathfrak{U}}(t)=0$ for all $0\le t\le T$  for small enough $T>0,$
and thus  $V_1\equiv V_2 $ on $[0,T]\times\Rd.$  Appealing to a connectivity argument  
yields uniqueness on the whole interval existence, which
completes the proof.

\section{Application to the compressible Navier-Stokes system}\label{sec:appli:CFNS}

We here consider the full Navier-Stokes system 
governing the evolution of a  \emph{Newtonian} compressible fluid in   $\Rd,$ with no external force. 
Denoting by  $u = u(t,x) \in \Rd$  its velocity field,   $\ro=\ro(t,x)\in \mathbb{R}_+,$ its density, 
 $p = p(t,x) \in \mathbb{R}$, its pressure, $\theta = \theta (t, x) \in \mathbb{R}_+,$ its absolute  temperature and 
  $e = e(t, x) \in \mathbb{R}$, its  internal energy by unit mass,  this system reads:
\begin{align}
     \label{CFNS}
     \begin{cases}
         &  \pt \ro+ \div(\ro u )=0,\\
          &   \ro \pt u + \ro u\cdot \n u- \div( 2\mu D(u)+ \lambda \div u\, I_d)+\n p=0,\\
          &\ro e_\theta(\pt \theta +u\cdot \n \theta)+\theta p_\theta \div u-\div(k \n \theta)= {\mathbb T},
     \end{cases}
 \end{align}
with \begin{equation} D(u)\defn\frac12(\nabla u+{}^t\nabla u)\andf
 {\mathbb T} \defn \label{def:vaphi:Dx}
   \frac{\mu}{2}\sum_{i,j=1}^d(\partial_{j} u^i+ \partial_{i} u^j)^2+ \lambda (\div u )^2.\end{equation}
To close the system, we make the following \textbf{Assumption D}:
\begin{itemize}
    \item The thermodynamic quantities $p$ and $e$ are smooth functions of $\ro,\theta>0$ 
   such that
    \begin{align}        \label{positive:Pr:eth}
        p_\ro \defn \frac{\partial p}{\partial\ro}>0 \esp{and} e_\theta\defn  \frac{\partial e}{\partial\theta}>0.
    \end{align}
    \item The viscosity coefficients $\lambda,\mu$ and the heat conductivity $k$ are  smooth functions of $\ro,\theta>0$
    that satisfy:
    \begin{equation}       \label{LAme:cnd}       \mu>0,\quad\nu\defn 2\mu+\lambda>0 \esp{and} k>0.
    \end{equation}
\end{itemize}
Denoting    $ U\defn(\ro,u,\theta),$ System  \eqref{CFNS} may be rewritten: 
\begin{align}
    \label{SHPDS:NS}
    S^0(U)\frac{d}{dt} U+\suma S^\alpha(U)\pal U-\sumab \pal\left( Y^{\alpha\beta}(U)\pbe U\right)=f(U)
\end{align}
where the matrices $S^\alpha(U)$ and $Y^{\alpha\beta}(U),$ and the function $f$ are defined on the phase space
$\mathcal{U}\defn \bigl\{(\ro,u,\theta)\in   \mathbb{R}^{d+2}/ \ro>0,\; \theta>0\bigr\}$ by
\begin{align*}
    S^0(U)\defn 
    \begin{pmatrix}
\frac{p_\ro}{\ro} & 0 & 0\\
0 & \ro I_d & 0\\
0&0 & \frac{\ro e_\theta}{\theta}
\end{pmatrix},\;\;\;
 f(U)\defn  \begin{pmatrix}
0\\
0 \\
 {\mathbb T}-k\n \theta\cdot \n (\frac{1}{\theta})
\end{pmatrix},
\end{align*}
\begin{align*}
   \suma S^\alpha(U)\xi_\alpha \defn
 \begin{pmatrix}
\frac{p_\ro}{\ro}u\cdot \xi & p_\ro \xi & 0\\
p_\ro ^T\xi & \ro(u\cdot \xi) I_d & p_\theta ^T\xi\\
0& p_\theta \xi & \frac{\ro e_\theta}{\theta}u\cdot \xi
\end{pmatrix},
\end{align*}
and
\begin{align*}
        Y^{\alpha\beta}\xi_\alpha\xi_\beta \defn
    \begin{pmatrix}
0& 0 \\
0 & Z^{\alpha\beta}\xi_\alpha\xi_\beta 
\end{pmatrix} 
\with  
Z^{\alpha\beta}\xi_\alpha\xi_\beta\defn 
\begin{pmatrix}
 \mu |\xi|^2+(\mu+\lambda)\xi\otimes \xi  {\rm Id} & 0\\
0 & \frac{k}{\theta} |\xi|^2
\end{pmatrix}\cdotp
\end{align*}
The matrix $S^0(U)$ is  diagonal positive  for all $U\in \mathcal{U}$, and  the matrices $S^\alpha(U)$ are real symmetric. Furthermore, a simple calculation reveals that\begin{align}
    \label{sca:Yal:be:V}
     \sumab \left\langle Z^{\alpha\beta}\xi_\alpha\xi_\beta A, A\right\rangle\ge \biggl(\min(\mu,\nu)|X|^2+\frac{k}{\theta}Y^2\biggr)|\xi|^2,
     \quad  X,\xi\in \Rd, \ Y\in \mathbb{R}, 
\end{align}
where $A\defn (X,Y)\in\Rd\times\R$ and $\left\langle \cdot ,\cdot \right\rangle$ denotes the canonic scalar product in $\mathbb{R}^{d}\times\R$. The right-hand side of \eqref{SHPDS:NS} is  a lower order quadratic term, that satisfies the 4-th condition of Assumption {\bf B}.
As a direct application of Theorem~\ref{Thm:loc:scri}, we get:
\begin{theorem} \label{thm:loc:sur:CFNS}  Let $d\geq1.$
   Let Assumption \textbf{D} be in force and let  $s\ge d/2.$  Assume that $\overline{\ro}>0$ and $\overline{\theta}>0,$ and  that  $(\ro_0,u_0,\theta_0)\in \mathcal{U}$ satisfies $\ro_0- \overline{\ro}\in \B{s+1}{2}{1}$, $u_0\in \B{s}{2}{1}$ and $\theta_0- \overline{\theta}\in \B{s}{2}{1}$.\\
   Then, there exists some $T>0$ such that  the problem \eqref{CFNS} supplemented with the initial data $ (\ro_0,u_0,\theta_0)$ has a unique solution $(\ro,u,\theta)\in \mathcal{U}$ on $[0,T]\times{\mathbb R}^d$  such that
      \[ \ro- \overline{\ro}  \in  {\cC}([0,T];\B{s+1}{2}{1}) \esp{and} (u,\theta-\overline{\theta})\in  {\cC}([0,T];\B{s}{2}{1})\cap  L^1_{T}(\B{s+2}{2}{1}). \]
\end{theorem}
Assuming that the viscosity coefficients $\mu$ and $\nu,$ and the pressure $p$ only depend  on $\rho,$ 
the first two equations of \eqref{CFNS} may be seen as the following closed system (the so-called barotropic compressible Navier-Stokes system):
 \begin{align}     \label{CNS:ro}
     \begin{cases}         &  \pt \ro+ u\cdot\n \ro+ \ro\div u=0\\
          &   \ro \pt u + \ro u\cdot \n u- \div( 2\mu(\ro) D u+ \lambda(\ro) \div u I_d)+\n p(\ro)=0.
     \end{cases} \end{align}
 If  assuming   \eqref{LAme:cnd},  
then Assumption  \textbf{C} is satisfied and Theorem \ref{thm:loc:cri} allows to recover  the following result 
that has  been proved by the second author  in \cite{Danchin07}:
\begin{theorem}    \label{thm:loc:cri:CFNS}
 Assume $d\geq2$. Let $\overline{\ro}>0$  and  suppose that the initial data $(\ro_0,u_0)$ satisfy $\ro_0- \overline{\ro}\in \BH{\cd}{2}{1}$, $u_0\in \BH{\cd-1}{2}{1}$ and $\ro_0$ bounded away from zero.
 
   Then, System \eqref{CNS:ro} supplemented with the initial data $ (\ro_0,u_0)$ has a unique solution $(\ro,u)$ on $[0,T]\times\Rd$ for some  $T>0,$
   with $\rho$ bounded away from zero, 
 \[  \ro-\overline{\ro}\in  {\cC}([0,T];\BH{\cd}{2}{1}) \esp{and} u\in {\cC}([0,T];\BH{\cd-1}{2}{1})\cap  L^1_{T}(\BH{\cd+1}{2}{1}). \]
\end{theorem}

\medbreak
\paragraph{\textbf{Acknowledgments}}\textit{The first author has received funding from the European Union's Horizon 2020 research and innovation programme under the Marie Sk\l{}odowska-Curie grant agreement \textnumero~945332.}



\appendix

\section{Littlewood-Paley decomposition and Besov spaces}
\label{appendix:LP}
Here  we briefly present some results on the Littlewood-Paley decomposition and Besov spaces. 
More details may be  found in \cite[Chap. 2]{HajDanChe11}. 
\smallbreak
To define the Littlewood-Paley decomposition, we  fix some smooth radial non increasing function $\chi$ with $ \mathrm{Supp}\chi \subset B(0,\frac{4}{3})$  and $\chi \equiv 1$  on $ B(0,\frac{3}{4})$, then set $\varphi(\xi)=\chi(\frac{\xi}{2})-\chi(\xi)$ so that
\begin{align*}
    \chi+\sum_{j\ge 0} \varphi(2^{-j}\cdot)=1 \ \text{ on } \ \Rd \esp{and} \sum_{j\in \mathbb{Z}} \varphi(2^{-j}\cdot) =1  \ \text{ on }\ \Rd\backslash\{0\}. 
\end{align*}
We introduce the following homogeneous  and nonhomogeneous spectral cut-off operators: 
\begin{align}
    \label{def:DDj:Dj}
    \begin{split}
    &\DDj \defn \varphi(2^{-j}D) \;\; \text{for all}\; j\in \mathbb{Z}, \\
    & \Dj=\DDj  \;\; \text{for all}\; j\ge 0,\;\; \Delta_{-1}= \Dot{S}_0 \esp{and} \Dj=0  \;\; \text{for all}\; j< -1,\\[-1ex]
    \end{split}
\end{align}
\begin{align}    \label{def:dot:S:j}
    \Dot{S}_j =\chi(2^{-j}D) \;\; \text{for all}\; j\in \mathbb{Z}\esp{and} S_j\defn\Dot{S}_j \;\; \text{for all}\; j\ge 0,\;\; S_{j}=0 \;\; \text{for all}\; j\le  -1.
    \end{align}
We denote by  $\mathcal{S}'_h$  the set of all tempered distributions $z$ such that 
\begin{align}
    \label{cnd:on S'h}
    \underset{j \rightarrow -\infty}{\lim} \Dot{S}_j z=0.
\end{align}
For $s\in\R$ and $p\in[1,\infty],$ we introduce the homogeneous Besov semi-norms (resp. nonhomogeneous Besov norms):
\begin{align}
    \label{def:NBH:NB}
    \NBH{s}{p}{r}{z} \defn \Vert 2^{js} \Vert\DDj  z\Vert_{L^p}\Vert_{\ell^r(\Z)} \quad (\text{resp. }   \NB{s}{p}{r}{z} \defn \Vert 2^{js} \Vert\Dj  z\Vert_{L^p}\Vert_{\ell^r(j\geq-1)}).
\end{align}
Then, for any $s\in \mathbb{R}$ and $r\in [1,\infty]$ we define the homogeneous Besov spaces  $\BH{s}{p}{r}$ (resp. nonhomogeneous Besov spaces $\B{s}{p}{r}$) to be the subset of those  $z$ in $\mathcal{S}'_h$ (resp.  those 
tempered distributions $z$) such that $\NBH{s}{p}{r}{z} $ (resp. $\NB{s}{p}{r}{z}$) is finite.

Although in the study of non-stationary PDEs spaces of type $L^\ro(0,T;X)$ for appropriate
Banach spaces $X$ and Lebesgue exponent $\rho$ come up naturally,
we sometimes needed  to use the Chemin-Lerner spaces \cite{Chemin1999}  that are defined below: 
\begin{definition}
    \label{def:Chemin lener}
    Let  $\ro$ in $[1,\infty]$ and time $T \in [0, \infty]$. We set
    \begin{align*}
        \TLpNB{s}{2}{r}{z}{\ro}\defn \Vert 2^{js}\Vert \Dj z\Vert_{L^\ro_T(L^2)}\Vert_{\ell^r(j\geq-1)}\esp{with}  \Vert z \Vert_{L^\ro_T(L^2)} \defn \Vert z\Vert_{L^\ro(0,T;L^2)}.
    \end{align*}
   Then, $\widetilde{L}^\ro_T(\B{s}{2}{r})$
 is the set of tempered distributions $z$ on $[0,T]\times \Rd$ such that $ \TLpNB{s}{2}{r}{z}{\ro}<\infty.$
    \end{definition}
    We also set $\widetilde{\cC}([0,T];\B{s}{2}{r})\defn \widetilde{L}^\infty_T(\B{s}{2}{r})\cap \cC([0,T];\B{s}{2}{r}) $
and define similarly spaces  $\widetilde{L}^\ro_T(\BH{s}{2}{r}).$
 \smallbreak
 Let us emphasize that, according
to the Minkowski inequality, we have:
\begin{align}
    \label{link:CL:Besov:space}
    \TLpNB{s}{2}{r}{z}{\ro}\le  \LpNB{s}{2}{r}{z}{\ro},\; \text{if}\; r\ge \ro \esp{and}
     \LpNB{s}{2}{r}{z}{\ro} \le  \TLpNB{s}{2}{r}{z}{\ro}, \; \text{if}\; r\le \ro.
\end{align}
We keep the same notation for Besov spaces pertaining to functions with several components. 
\medbreak
 In order to bound the commutator  terms, we used the following results:

\begin{proposition}
	 \label{propo_commutator-BH}
	 Let $\mathbb{B}_{2,r}^{s}$  designate $\B{s}{2}{r}$ and $\BH{s}{2}{r}$. The following inequalities hold true
 for all  $\sigma>0$:
     \begin{align}
       \label{comm:est:a:b:sig>0}
        \normede{[a,\Dj] b} &\le C2^{-j\sigma}c_j (\normeinf{\n a} \lVert b \rVert_{\mathbb{B}_{2,1}^{\sigma-1} }+ \normeinf{b} \lVert \n a \rVert_{\mathbb{B}_{2,1}^{\sigma-1} }  ) \esp{with} \sum_j c_j =1,\\
 \label{comm:est:a:b:sig>0b}
        \normede{[a,\Dj] b} &\le C2^{-j\sigma}c_j (\normeinf{\n a} \lVert b \rVert_{\mathbb{B}_{2,1}^{\sigma-1} }+ \|b\|_{\mathbb B^{-1}_{\infty,\infty}} \lVert \n a \rVert_{\mathbb{B}_{2,1}^{\sigma }}  ) \esp{with} \sum_j c_j =1.
   \end{align}
      If $\sigma\ge \cd+1$, there also holds:
    \begin{align}
       \label{comm:est:a:b:sig>cd}
        \normede{[a,\Dj] b} \le C2^{-j\sigma}c_j  \NB{\sigma-1}{2}{1}{\n  a}\NB{\sigma-1}{2}{1}{b} \esp{with} \sum_{\substack{j\ge -1}} c_j =1.
   \end{align}
	If  $-\frac d2 < \sigma \le \cd+1 $,    then
   \begin{align}
       \label{comm:est:a:b}
        \normede{[a,\Dj] b} \le C2^{-j\sigma}c_j \lVert \n a \rVert_{\mathbb{B}_{2,\infty}^{\cd} \cap L^\infty}  \lVert b \rVert_{\mathbb{B}_{2,1}^{\sigma-1} }  \esp{with} \sum_j c_j =1,   \end{align}
   and if  $-\frac d2 \leq \sigma < \cd+1,$ we also have:
   \begin{align}
       \label{comm:est:a:b:inf}
       \underset{j}{\sup} \normede{[a,\Dj] b} \le C2^{-j\sigma} \Vert\n a\Vert_{\mathbb{B}^{\cd}_{2,\infty}\cap L^\infty} \Vert b\Vert_{\mathbb{B}^{\sigma-1}_{2,\infty}}.   \end{align}
    Similar results hold  true  if we replace $L^2$ by $L^\ro_T(L^2)$ in the l.h.s and 
   use Chemin-Lerner norms  in the r.h.s.  For instance, 
   Inequality \eqref{comm:est:a:b:sig>cd} becomes for all $1\leq\rho,\rho_1,\rho_2\leq\infty$ s.t. $1/\rho=1/\rho_1+1/\rho_2,$
  $$     \|[a,\Dj] b\|_{L_T^\rho(L^2)} \le C2^{-j\sigma}c_j  \|\nabla a\|_{\wt L_T^{\rho_1}(B^{\sigma-1}_{2,1})}
  \|b\|_{\wt L_T^{\rho_2}(B^{\sigma-1}_{2,1})}
   \esp{with} \sum_{\substack{j\ge -1}} c_j =1.$$
  	 \end{proposition}
	 \begin{proof}
	 All the above estimates but the second one are stated almost as is in \cite[Lemma 2.100]{HajDanChe11}
	 and follow from the decomposition that is performed at the top of page 113 therein.
	 As for Inequality \eqref{comm:est:a:b:sig>0b}, it may be proved as \eqref{comm:est:a:b:sig>0}
	 except that the third term $R_j^3$ defined at page 113  is bounded by means of Theorem 2.82 as follows:
	 $$\|R^3_j\|_{L^2}\leq Cc_j2^{-j\sigma} \|b\|_{\mathbb B^{-1}_{\infty,\infty}} \lVert \n a \rVert_{\mathbb{B}_{2,1}^{\sigma }}.$$
	 This gives  the result.   
	 		 \end{proof}
	 
  The following product laws in Besov spaces have been used repeatedly.
	 \begin{proposition}
	 \label{propo_produc_BH}
	     Let $(s,r)\in ]0,\infty[\times[1,\infty]$. Then $\mathbb{B}^s_{2,r}\cap L^\infty$ is an algebra and we have 
	     \begin{equation}	         \label{product_propo1}
	      \lVert ab \rVert_{\mathbb{B}_{2,r}^{s}}  \le C\bigl(\normeinf{a}  \lVert b \rVert_{\mathbb{B}_{2,r}^{s}} 
	      + \normeinf{b}  \lVert a \rVert_{\mathbb{B}_{2,r}^{s}}\bigr)\cdotp
	     \end{equation}
	Moreover, if  $-d/2< s \le d/2$, then the following inequality holds:
	      \begin{equation}	       \label{product_propo2}
	     \lVert ab \rVert_{\mathbb{B}_{2,1}^{s}} \le C  \lVert a \rVert_{\mathbb{B}_{2,1}^{\cd}}  \lVert b \rVert_{\mathbb{B}_{2,1}^{s}}
	     \end{equation}
      and if $  -d/2\le s < d/2,$
       \begin{equation}	         \label{product:propo:uniq}
           \lVert ab \rVert_{\mathbb{B}_{2,\infty}^{s}} \le C  \lVert a \rVert_{\mathbb{B}_{2,\infty}^{\cd}\cap L^\infty}  \lVert b \rVert_{\mathbb{B}_{2,\infty}^{s}}
	     \end{equation}
	      If $ s > d/2 $ (or $s= d/2 $ and $r=1$), 
	     \begin{equation}
	         \label{product_propo3}
	         \NB{s}{2}{r}{ab} \le C\NB{s}{2}{r}{a}\NB{s}{2}{r}{b}.
	     \end{equation}
	     Finally,  for all $s>0,$ we have
	         \begin{equation}\label{product_propo4}
	         \NB{s}{2}{r}{ab} \le C\bigl(\|a\|_{L^\infty}\NB{s}{2}{r}{b} +\|b\|_{B^{-1}_{\infty,\infty}} \NB{s+1}{2}{r}{a}\bigr)\cdotp
	     \end{equation}
  	 	      \end{proposition}
  The first above three estimates are classical (see e.g. \cite[Chap. 2]{HajDanChe11}).
  The fourth one follows from the first one and embedding. 
  Inequality  \eqref{product_propo3} can  be  proved by using Bony's decomposition and 
  then suitable continuity results for the paraproduct and remainder operators (see \cite[Section 2.8]{HajDanChe11}).
    We have similar results for the spaces $L^\ro_T(\B{s}{2}{r})$, $\widetilde{L}^\ro_T(\B{s}{2}{r})$, $L^\ro_T(\BH{s}{2}{r})$ and $\widetilde{L}^\ro_T(\BH{s}{2}{r}),$ see \cite{Dan01loc,Haspot11}.
	\medbreak
	 We also needed the following composition estimates.
	 	 \begin{proposition}
	 \label{propo_compo_BH}
  	     Let $f$ be a function in $C^\infty(\mathbb{R})$.          Let $r\in [1,\infty]$ and  $s \in ]0,\infty[$. If $f(0)=0$ then, for every real-valued function $u$ in $\B{s}{2}{r}\cap L^\infty$, the function $f\circ u $ belongs to $\B{s}{2}{r}\cap L^\infty$ and satisfies
       \begin{align}
           \label{compo:propo:base}
          \Vert  f\circ u\Vert_{{B}_{2,r}^{s}} \le C(f', \normeinf{u})\ \Vert u \Vert_{{B}_{2,r}^{s}}.
       \end{align}
       If both $u$ and $v$ are in $B^s_{2,1}\cap B^{\cd}_{2,1}$  with  $s>-d/2,$  then we have
      \begin{align}
          \label{comp:uv:propo:1}
          \NB{s}{2}{1}{f\circ u- f\circ v } \le C(f', \normeinf{u,v})(1+\NB{\max(s,\frac{d}{2})}{2}{1}{(u,v)})\NB{s}{2}{1}{ u-v} .
      \end{align}
      Furthermore, if $-\cd\le s < \frac{d}{2}$ then the last  inequality remains valid for $r=\infty$, that is,
         \begin{align}
          \label{comp:uv:propo:inf}
         \NB{s}{2}{\infty}{f\circ u- f\circ v } \le C(f', \normeinf{u,v})(1+\NB{\frac{d}{2}}{2}{1}{(u,v)})\NB{s}{2}{\infty}{ u-v}.
      \end{align}
      Similar results hold true  for homogeneous Besov spaces and Chemin-Lerner spaces.
	 \end{proposition}
  \begin{proof}
      The proof of \eqref{compo:propo:base} can be found in \cite[pages 94 and 104]{HajDanChe11} while \eqref{comp:uv:propo:1}, and \eqref{comp:uv:propo:inf}  can be obtained by adapting the proof of first inequality of \cite[page 449]{HajDanChe11}.
  \end{proof}
    Finally, the following composition estimates
    enabled us to handle the lower  order terms:
  \begin{proposition}
	 \label{propo_Compo_u_v_BH}
  Let $0\le n_1\le n$ and $m$ be three integers. 	  Let $f: (X,Y)\in \mathbb{R}^{n_1}\times\mathbb{R}^{n-n_1} \mapsto f(X,Y)\in\mathbb{R}^m $ be a smooth function on $\mathbb{R}^n$. Assume that $f$ is affine with respect to $ Y.$ 
  
     If $f$  vanishes at $0_{\mathbb{R}^n}$, then for any $0<s\le \cd $  the following inequality holds true
  \begin{align}
      \label{comp:u_v:BH:ine:prop:1}
      \NB{s}{2}{1}{f(u,v)}\le   C(f', \normeinf{u})( \NB{s}{2}{1}{ v}(1+\NB{\cd}{2}{1}{u})+\NB{s}{2}{1}{ u} ) . 
  \end{align}
  Furthermore if $-\cd< s\le \cd,$ then we have
    for some  $C= C(f', \normeinf{u_1,u_2})$: 
    \begin{multline}
      \label{comp:u_v_1:2:BH:ine:prop:1}
      \NB{s}{2}{1}{f(u_1,v_1)-f(u_2, v_2)}\le  C  \NB{s}{2}{1}{ v_2-v_1}(1+\NB{\cd}{2}{1}{u_2})\\
      +C(1+\NB{\cd}{2}{1}{u_1}+\NB{\cd}{2}{1}{u_2})\biggl( \NB{\cd}{2}{1}{u_2-u_1}\NB{s}{2}{1}{ v_1}+ \NB{s}{2}{1}{ u_1-u_2}\biggr)\cdotp
  \end{multline}
    Finally  if   $ -\cd\le  s< \frac{d}{2}$ then we have
  \begin{multline}
      \label{comp:u_v_1:2:BH:ine:prop:inf}
      \NB{s}{2}{\infty}{f(u_1,v_1)-f(u_2, v_2)}\le  C  \NB{s}{2}{\infty}{ v_2-v_1}(1+\NB{\cd}{2}{1}{u_2})\\
      +C(1+\NB{\cd}{2}{1}{u_1}+\NB{\cd}{2}{1}{u_2}) \biggl(\NB{s}{2}{\infty}{u_2-u_1}\NB{\cd}{2}{1}{ v_1}+ \NB{s}{2}{\infty}{ u_1-u_2}\biggr),
      \end{multline}
      where $C= C(f', \normeinf{u_1,u_2}) $.
	 \end{proposition}
	\begin{proof}
	By assumption, there exist  two smooth functions $\Lambda$ and $\Gamma$ defined on $\mathbb{R}^{n_1}$ such that
     \[    f(u,v) = \Lambda(u) v+ \Gamma(u) \esp{(with} \Gamma(0_{\mathbb{R}^{n_1}})=0_{\mathbb{R}^m}\;\; \text{if}\;\; f(0_{\mathbb{R}^{n_1}},0_{\mathbb{R}^{n-n_1}})=0_{\mathbb{R}^{m}}).   \]

If  $ 0<s\le \frac{d}{2}$ then, applying the inequalities \eqref{product_propo2} and \eqref{compo:propo:base} to the term $ \Lambda(u) v $ yields the first term of the right-hand side of  inequality \eqref{comp:u_v:BH:ine:prop:1}.
Next, using Proposition \ref{propo_compo_BH} (recall that $s>0$) for the term $  \Lambda(u) $ gives the second term of \eqref{comp:u_v:BH:ine:prop:1}. 

 To prove \eqref{comp:u_v_1:2:BH:ine:prop:1}  (resp.\eqref{comp:u_v_1:2:BH:ine:prop:inf}), we use  the above decomposition to get
     \[  f(u_2,v_2)- f(u_1,v_1)= \Lambda(u_2)(v_2-v_1) +(\Lambda(u_2)-\Lambda(u_1))v_1 +(\Gamma(u_2)-\Gamma(u_1)).   \]
  Having this decomposition at hand, the first two terms of the last equality may be bounded from Inequalities \eqref{product_propo2}  (resp. \eqref{product:propo:uniq})  and  \eqref{compo:propo:base} (resp. \eqref{comp:uv:propo:inf}). Concerning the last one, we use Inequality \eqref{comp:uv:propo:1} (resp. \eqref{comp:uv:propo:inf}).
	\end{proof}

  
\section{Some inequalities}
\label{appendix:resu}
Here we gather a few inequalities that have been used repeatedly in the paper. The first one is the following well known result about a differential  inequality.
\begin{lemma}
\label{lem_der_int}
    Let $X$ be an a.e.  differentiable nonnegative function on $[0,T].$
    Assume that there exists a nonnegative constant $B$ and a measurable function $A :  [0, T] \to\mathbb{R}_+ $  such that
\[ \frac{1}{2}\frac{d}{dt} X +BX\le AX^{\frac{1}{2}} \esp{a.e on}   [0, T].\]
Then, for all $t\in [0,T]$, we have 
\[  X^{\frac{1}{2}}(t)+B\int^t_0 X^{\frac{1}{2}} \le X^{\frac{1}{2}}(0) + \int^t_0 A.\]
\end{lemma}

The second class of inequalities concerns the following  linear parabolic equation:   
	    \begin{equation}
	 \label{lin:parabolic}
	     \left\{
	 \begin{array}{l}
	    S\pt V  - Z(D) V=0  \\
	     V(0)= V_0
	 \end{array}
	 \right.
	 \end{equation}
  where $S$ is a  positive definite Hermitian matrix and $Z \in C^\infty(\Rd\setminus\{0\};\mathcal{M}_n(\mathbb{C})) $ is  homogeneous of degree $\gamma \in \mathbb{R}$ and such that the matrix $Z(\xi)$  satisfies for some constant $\kappa>0$
    \begin{align}
        \label{posi:cond:Z}
       {\rm Re} (Z(\xi)z\cdot z)\ge \kappa |\xi|^\gamma|z|^2,\;\; \xi\in \Rd\backslash \{0\},\; z\in \mathbb{C}^n.
    \end{align}
\begin{proposition}
 \label{lem_lin_est}
     There exist  two   positive constants $c,C_0$ such that for all $s\in \mathbb{R},$   $T,h\in \mathbb{R}^+$ and $m\in \mathbb{N}$, the following estimates hold:
     \begin{align}
         \label{L-inf_j}
         \Vert V \Vert_{\Tilde{L}^\infty_T(\B{s}{2}{1})} \le C_0\NB{s}{2}{1}{V_0}\andf
        \sum_{j\ge m} 2^{js}\Vert \Dj V \Vert_{L^\infty_T(L^2)} \le C_0  \sum_{j\ge m} 2^{js} \Vert \Dj V_0 \Vert_{L^2},
     \end{align}
     \begin{align}
         \label{L-1_j}
        \int_T^{T+h}\sum_{j\ge m} \!2^{j(s+2)} \Vert\!\left(\Dj D^\gamma V,\Dj \pt V\right)\!\Vert_{L^2}\le C_0\sum_{j\ge m} 
       e^{-c2^{j\gamma}T}\left(1-e^{-c2^{j\gamma}h}\right)2^{js}\Vert \Dj V_0 \Vert_{L^2},
     \end{align}
     \begin{align}
         \label{L-1_las}
        \int_T^{T+h}\Vert (\Delta_{-1} V,\pt \Delta_{-1}V)\Vert_{L^2}\le C_0h\NB{s}{2}{1}{V_0}.
     \end{align}
    Inequalities \eqref{L-inf_j} and \eqref{L-1_j} are also valid  with $\DDj$, $j,m\in \mathbb{Z}$ and homogeneous
Besov norms. 
 \end{proposition}
 \begin{proof}
 It is only a matter of adapting the proof  for the  heat equation given in \cite{Chemin1999}. First, applying $\Dj$ to \eqref{lin:parabolic}, then 
 the Fourier transform with respect to the space variable gives:
 $$ \partial_t\widehat V_j- Z \widehat V_j=0.$$
 Then, taking  the $L^2$ inner product with 
     $\widehat V_j$ then keeping the real part gives
     \begin{align*}
\dfrac12        \dfrac{d}{dt} \Vert \widehat{V_j} \Vert^2_{L^2_S(\Rd)}+ {\rm Re}\intd Z(\xi)  \widehat{V_j}(\xi)\cdot \widehat{V_j}(\xi)d\xi=0\; \text{ with }\; \Vert V \Vert^2_{L^2_S(\Rd)}\defn \intd S V\cdot V.
     \end{align*}
     Next, using the strong  ellipticity  condition \eqref{posi:cond:Z},
      we get for all $j\ge -1$,
      \begin{align}
      \label{est_first}
     \dfrac12    \dfrac{d}{dt} \Vert \widehat{V_j} \Vert^2_{L^2_S(\Rd)} +\kappa \intd |\xi|^\gamma  |\widehat{V}_j(\xi)|^2 d\xi \le 0.
     \end{align}   
     Since  owing to the spectral localization, we have   for some positive constant $c,$  
      $$ |\xi|^\gamma   |\widehat{V_j}(\xi)|^2 \ge c2^{j\gamma}| \widehat{V_j}(\xi)|^2\esp{for all}j\ge 0,$$ 
    combining with  the fact $S$ is a constant, symmetric, positive definite  matrix, we get 
 \begin{align*}
     \dfrac12   \dfrac{d}{dt} \Vert \widehat{V_j} \Vert^2_{L^2_S(\Rd)} + c\kappa 2^{j\gamma} \Vert \widehat{V_j} \Vert^2_{L^2_S(\Rd)}\le 0, \ \text{for all}\; j\ge 0.
 \end{align*}
 This  leads to (up to a slight modification of $c$):
     \begin{align}
     \label{est:inf:Vj:j:ge:0}
    \Vert V_j(t) \Vert_{L^2_S(\Rd)}  
       \le e^{-c2^{j\gamma}t}    \Vert V_{0,j}\Vert_{L^2_S(\Rd)}\; \text{ for all }\; j\ge 0,\; t\geq0.
     \end{align}
     This  readily gives  \eqref{L-inf_j} and  \eqref{L-1_j} after taking suitable time-Lebesgue norms and summing on $j.$
    The case  $j=-1$  stems from  \eqref{est_first}, after omitting the second term in \eqref{est_first}.
    \end{proof}
Next, let us explain how to handle a second order operator with \emph{variable} coefficients. 
\begin{lemma}[G\r{a}rding inequality]
\label{gard_lem}
   Let $U:\mathbb{R}^d \mapsto \mathbb{R}^n$ be a   bounded function.  
   Assume that  the (real valued) operator $Z(U)\n_x$ is strongly elliptic in the sense of \eqref{strong_elli}.  
   Then, there exists a positive constant $c$ depending only on  the ellipticity constant   
   such that for all small enough $\varepsilon>0 $ the following inequality  holds true
   for all smooth function $f:{\mathbb R}^d\to \mathbb{R}^{n_2}$:
    \begin{multline}
    \label{gain_eli}
       -\sum_{\alpha,\beta,i,j}\intd Z^{\alpha\beta}_{ij}(U(x))\partial_\alpha\partial_\beta f^i(x) f^j(x)dx\\\ge c\Vert \nabla f\Vert^2_{L^2(\mathbb{R}^d)}-\varepsilon\Vert \nabla^2 f\Vert_{L^2(\mathbb{R})}\Vert  f\Vert_{L^2(\mathbb{R}^d)} -C \Vert  f\Vert^2_{L^2(\mathbb{R}^d)},
    \end{multline}
  where $C=C(c,\varepsilon, U)>0$ depends   only on $\varepsilon,$ the range of $U$ and  the ellipticity constant.
\end{lemma} 
\begin{remark}
    \label{rmq:garb:in}
    The `standard' G\r{a}rding inequality reads:
     \begin{equation}
    \label{gain_elli}
     -  \sum_{\alpha,\beta,i,j}\intd 
     \partial_\beta\Bigl(
     Z^{\alpha\beta}_{ij}(U)\partial_\alpha f^i\Bigr)(x)\:\partial_\beta f^j(x)dx\ge c\Vert \nabla f\Vert^2_{L^2(\mathbb{R}^d)} -C \Vert  f\Vert^2_{L^2(\mathbb{R}^d)}.
    \end{equation} 
Although Inequality  \eqref{gain_eli}  may seem weaker since there remain second order derivatives in the right-hand side, it will be useful for us once combined with Bernstein inequality, 
since it will be applied  only to spectrally localized functions $f.$
\end{remark}

\begin{proof}[Proof of Lemma \ref{gard_lem}.] If the functions $ Z^{\alpha\beta}_{ij}$ are constant 
    then, in light of  Fourier-Plancherel  theorem and of  \eqref{strong_elli},
    we have
   $$    -\sum_{\alpha,\beta,i,j}\intd Z^{\alpha\beta}_{ij}\partial_\alpha\partial_\beta  f^i(x)f^j(x)dx=\mathcal{R}e  \sum_{\alpha,\beta,i,j}\intd Z^{\alpha\beta}_{ij}\xi_\alpha \xi_\beta  \widehat{f^i}\overline{\widehat{f^j}}d\xi         \ge c_1 \Vert \nabla f\Vert^2_{L^2(\mathbb{R}^d)}$$
    where $c_1$ in the constant appearing in \eqref{strong_elli}. Hence, \eqref{gain_eli} is true in this special case.
\smallbreak
   In the case of variable coefficients, if  the function $U$  has range in a  small  ball $B(\overline U,\eta)$ about 
    $\overline U,$  leveraging  the preceding case gives:
\begin{align*}
         -\sum_{\alpha,\beta,i,j}\intd Z^{\alpha\beta}_{ij}(U(x))\partial_\alpha\partial_\beta f^i \,f^j\,dx&=-\sum_{\alpha,\beta,i,j}\intd Z^{\alpha\beta}_{ij}(\overline U)\partial_\alpha\partial_\beta f^i\, f^j\,dx \nonumber\\
         &-\sum_{\alpha,\beta,i,j}\intd (Z^{\alpha\beta}_{ij}(U(x))-Z^{\alpha\beta}_{ij}(\overline U))\partial_\alpha\partial_\beta f^i\,f^j\,dx\nonumber\\
         \ge c_1 \Vert \nabla f\Vert^2_{L^2(\mathbb{R}^d)}& -\sum_{\alpha,\beta,i,j}\intd (Z^{\alpha\beta}_{ij}(U(x))-Z^{\alpha\beta}_{ij}(\overline U))\partial_\alpha\partial_\beta  f^i\,f^j\,dx.\nonumber\\
    \end{align*}
    If $\eta$ is small enough  that the $Z^{\alpha\beta}_{ij}(U)$'s have an oscillation  of size at most $\varepsilon,$ 
    then the second term of the last inequality may be bounded by $\varepsilon\Vert \nabla^2 f\Vert_{L^2(\mathbb{R})}\Vert  f\Vert_{L^2(\mathbb{R})}.$
\smallbreak
    Finally, in  the general case, for all $\varepsilon>0,$ the (bounded) range $G$ of $U$  
   may be recovered by a finite family $(B_k)_{1\leq k\leq N}$ of balls of radius $\eta.$ 
   Denoting  $\Omega_k \defn U^{-1}(B_k),$ one can thus consider 
     a partition of unity in $\mathbb{R}^d$  such that 
\begin{equation}\label{partition}
1=\sum_{k=1}^N{\omega_k^2} (x) \esp{} \forall x\in \mathbb{R}^d \esp{with} \omega_k \ge 0\andf {\rm Supp}\, \omega_k\subset \!\subset\Omega_k. 
\end{equation}
    Then, by the Leibniz' rule of differentiation of the product of functions, Cauchy-Schwarz inequality and the estimate of the case treated just above, we have,
    \begin{align*}
         -\sum_{\alpha,\beta,i,j}\intd Z^{\alpha\beta}_{ij}(U(x))\,\partial_\alpha\partial_\beta f^i\, f^j\, dx&= - \sum_{\alpha,\beta,i,j}\sum_{k=1}^N \intd Z^{\alpha\beta}_{ij}(U(x))\,\omega_k^2\, \partial_\alpha\partial_\beta f^i\,f^j\,dx\nonumber\\
         &= -\sum_{\alpha,\beta,i,j}\sum_{k=1}^N \int_{\Omega_k} Z^{\alpha\beta}_{ij}(U(x))\,\partial_\alpha\partial_\beta(\omega_k f^i)\,\omega_k f^j\,dx\nonumber\\
         &\qquad\qquad+2\!\sum_{\alpha,\beta,i,j}\sum_{k=1}^N \intd \!Z^{\alpha\beta}_{ij}(U(x))\,f^i\,\partial_\alpha\omega_k \,\partial_\beta f^j\,\omega_k\,dx \nonumber\\
         &\qquad\qquad+\sum_{\alpha,\beta,i,j}\sum_{k=1}^N \intd Z^{\alpha\beta}_{ij}(U(x))\,
         \partial_\alpha\partial_\beta \omega_k \, f^i\,f^j\,\omega_k\, dx\nonumber\\
         &\ge \sum_{k=1}^N\left(c_1\Vert \nabla (\omega_k f)\Vert^2_{L^2}-\varepsilon C(\omega_k) \Vert\nabla^2  f\Vert_{L^2}\Vert  f\Vert_{L^2}\right.\nonumber\\
         &\left.\qquad\qquad\qquad\quad
         - C(\omega_k)\Vert \nabla f\Vert_{L^2}\Vert f\Vert_{L^2}-C(\omega_k)\Vert f\Vert^2_{L^2}\right)\cdotp
    \end{align*}
    Relation \eqref{partition} is used for the first term of the right-hand side, after 
    observing that
    \begin{align*}   2 \normede{\n (\omega_k f) }^2\ge 
        \normede{\omega_k \n f}^2 -2\normede{f\n \omega_k}^2\ge \normede{\omega_k \n f}^2 -2C(\omega_k)\normede{f}^2.
    \end{align*}
    Then, using  Young's inequality for the term $C(w_k)\Vert \nabla f\Vert_{L^2}\Vert f\Vert_{L^2}$, that is,
    \[ C(w_k)\Vert \nabla f\Vert_{L^2}\Vert f\Vert_{L^2} \le \frac{c_1}{4} \Vert \nabla f\Vert_{L^2}^2 +  C(w_k) \Vert f\Vert_{L^2}^2, \]
    allows to get 
      the desired  result.        
    \end{proof}

\bibliographystyle{plain} 
\bibliography{boblio}

\begin{thebibliography}{10}

\bibitem{DanADOglob}
J.-P. Adogbo and R.~Danchin.
\newblock Global existence results for partially diffusive systems, and asymptotics.
\newblock {\em In progress}.

\bibitem{Angel23}
F.~Angeles.
\newblock The {C}auchy problem for a quasilinear system of equations with coupling in the linearization.
\newblock {\em Commun. Pure Appl. Anal.}, 22(10):2960--2999, 2023.

\bibitem{HajDanChe11}
H.~Bahouri, J.-Y. Chemin, and R.~Danchin.
\newblock {\em Fourier Analysis and Nonlinear Partial Differential Equations}, volume 343.
\newblock Springer, 2011.

\bibitem{GavSerre07}
S.~Benzoni-Gavage and D.~Serre.
\newblock {\em Multidimensional hyperbolic partial differential equations}.
\newblock Oxford Mathematical Monographs. The Clarendon Press, Oxford University Press, Oxford, 2007.
\newblock First-order systems and applications.

\bibitem{BurCrinTan23}
C.~Burtea, T.~Crin-Barat, and J.~Tan.
\newblock Pressure-relaxation limit for a one-velocity {B}aer-{N}unziato model to a {K}apila model.
\newblock {\em Math. Models Methods Appl. Sci.}, 33(4):687--753, 2023.

\bibitem{Chemin1999}
J.-Y. Chemin.
\newblock Th\'eor\`emes d'unicit\'{e} pour le syst\`{e}me de {N}avier-{S}tokes tridimensionnel.
\newblock {\em J. Anal. Math.}, 77:27--50, 1999.

\bibitem{Dan01glob}
R.~Danchin.
\newblock Global existence in critical spaces for compressible viscous and heat conductive gases.
\newblock {\em Archive for Rational Mechanics and Analysis}, 160:1--39, 2001.

\bibitem{Dan01loc}
R.~Danchin.
\newblock Local theory in critical spaces for compressible viscous and heat conductive gases.
\newblock {\em Communications in Partial Differential Equations}, 26:1183--1233, 2001.

\bibitem{Danchin05}
R.~Danchin.
\newblock On the uniqueness in critical spaces for compressible {N}avier-{S}tokes equations.
\newblock {\em NoDEA Nonlinear Differential Equations Appl.}, 12(1):111--128, 2005.

\bibitem{Danchin07}
R.~Danchin.
\newblock Well-posedness in critical spaces for barotropic viscous fluids with truly not constant density.
\newblock {\em Communications in Partial Differential Equations}, 32(9):1373--1397, 2007.

\bibitem{GioMat13}
V.~Giovangigli and L.~Matuszewski.
\newblock Structure of entropies in dissipative multicomponent fluids.
\newblock {\em Kinet. Relat. Models}, 6(2):373--406, 2013.

\bibitem{Haspot11}
B.~Haspot.
\newblock Existence of global strong solutions in critical spaces for barotropic viscous fluids.
\newblock {\em Archive for Rational Mechanics and Analysis}, 202(2):427--460, 2011.

\bibitem{Kawashima83}
S.~Kawashima.
\newblock {\em Systems of a hyperbolic parabolic type with applications to the equations of magnetohydrodynamics.}
\newblock PhD thesis, Kyoto University, 1983.

\bibitem{KawaSui88}
S.~Kawashima and Y.~Shizuta.
\newblock On the normal form of the symmetric hyperbolic-parabolic systems associated with the conservation laws.
\newblock {\em Tohoku Math. J. (2)}, 40(3):449--464, 1988.

\bibitem{Madja84}
A.~Majda.
\newblock Compressible fluid flow and systems of conservation laws in several space variable.
\newblock {\em Springer}, 1984.

\bibitem{Serre97}
D.~Serre.
\newblock {\em Syst\`emes de lois de conservation. {I}}.
\newblock Fondations. Diderot Editeur, Paris, 1996.
\newblock Hyperbolicit\'{e}, entropies, ondes de choc.

\bibitem{Serr10}
D.~Serre.
\newblock Local existence for viscous system of conservation laws: {$H^s$}-data with {$s>1+d/2$}.
\newblock In {\em Nonlinear partial differential equations and hyperbolic wave phenomena}, volume 526 of {\em Contemp. Math.}, pages 339--358. Amer. Math. Soc., Providence, RI, 2010.

\bibitem{Serre09}
D.~Serre.
\newblock The structure of dissipative viscous system of conservation laws.
\newblock {\em Phys. D}, 239(15):1381--1386, 2010.

\end{thebibliography}
\end{document}